\documentclass[a4paper,reqno,12pt]{amsart}
\usepackage[margin=1.2in]{geometry}
\usepackage{indentfirst}
\usepackage{amsfonts}
\usepackage{amsthm}
\usepackage{amssymb}
\usepackage{amsmath}
\usepackage{thmtools}
\usepackage{mathtools}
\usepackage{tikz}
\usepackage{tikz-cd}
\usepackage{xypic}
\usepackage{upgreek}

\usepackage[disable]{todonotes}
\usetikzlibrary{decorations.markings}
\usetikzlibrary{patterns}
\usepackage{enumitem}
\usepackage{setspace}
\usepackage{hyperref}
\usetikzlibrary{matrix,arrows,positioning}

\newtheorem{Lem}{Lemma}[section]
\newtheorem{Prop}[Lem]{Proposition}
\newtheorem*{Def}{Definition}

\theoremstyle{plain}
\newtheorem{Thm}[Lem]{Theorem}
\newtheorem{Cor}[Lem]{Corollary}

\newtheorem{thm}{Theorem}

\theoremstyle{definition}
\declaretheorem[numbered=no,name=Example,qed={\lower-0.3ex\hbox{$\triangleleft$}}]{Ex}
\newtheorem*{Rem}{Remark}
\newtheorem*{Rems}{Remarks}

\newcommand{\Hom}{\text{\textnormal{Hom}}}
\newcommand{\Mor}{\text{\textnormal{Mor}}}
\newcommand{\End}{\text{\textnormal{End}}}
\newcommand{\Aut}{\text{\textnormal{Aut}}}

\newcommand{\rank}{\text{\textnormal{rank}}}

\newcommand{\ev}{\text{\textnormal{ev}}}
\newcommand{\coev}{\text{\textnormal{coev}}}

\newcommand{\tr}{\text{\textnormal{tr}}}

\mathchardef\mhyphen="2D

\newcommand{\Obj}{\text{\textnormal{Obj}}}
\newcommand{\git}{/\!\!/}
\newcommand{\HInd}{\text{\textnormal{HInd}}}

\newcommand{\refl}{\text{\textnormal{ref}}}
\newcommand{\op}{\text{\textnormal{op}}}

\newcommand{\Map}{\text{\textnormal{Map}}}
\newcommand{\Mod}{\mhyphen\text{\textnormal{mod}}}
\newcommand{\ModR}{\mhyphen\text{\textnormal{mod}}^{\textnormal{R}}}
\newcommand{\Vect}{\mathsf{Vect}}
\newcommand{\pr}{\text{\textnormal{pr}}}
\newcommand{\sh}{\mathsf{sh}}
\newcommand{\Grp}{\mathsf{Grp}}

\newcommand{\id}{\text{\textnormal{id}}}

\newcommand{\Tr}{\mathbb{T}\text{\textnormal{r}}}

\makeatletter
\providecommand*{\twoheadrightarrowfill@}{%
  \arrowfill@\relbar\relbar\twoheadrightarrow
}
\providecommand*{\xtwoheadrightarrow}[2][]{%
  \ext@arrow 0579\twoheadrightarrowfill@{#1}{#2}%
}
\makeatother

\makeatletter
\providecommand\@dotsep{5}
\renewcommand{\listoftodos}[1][\@todonotes@todolistname]{%
  \@starttoc{tdo}{#1}}
\makeatother

\begin{document}

\title[Twisted loop transgression]{Twisted loop transgression and higher Jandl gerbes over finite groupoids}

\author[B. Noohi]{Behrang Noohi}
\address{School of Mathematical Sciences \\Queen Mary University of London\\ Mile End Road, London E1 4NS, United Kingdom}
\email{b.noohi@qmul.ac.uk}

\author[M.\,B. Young]{Matthew B. Young}
\address{Max Planck Institute for Mathematics\\
Vivatsgasse 7\\
53111 Bonn, Germany}
\email{myoung@mpim-bonn.mpg.de}

\date{\today}

\keywords{Groupoids. Twisted equivariant $KR$-theory. Real representation theory.}
\subjclass[2010]{Primary: 57R56; Secondary 19L50, 20C30}

\begin{abstract}
Given a double cover $\pi: \mathcal{G} \rightarrow \hat{\mathcal{G}}$ of finite groupoids, we explicitly construct twisted loop transgression maps, $\uptau_{\pi}$ and $\uptau_{\pi}^{\refl}$, thereby associating to a Jandl $n$-gerbe $\hat{\lambda}$ on $\hat{\mathcal{G}}$ a Jandl $(n-1)$-gerbe $\uptau_{\pi}(\hat{\lambda})$ on the quotient loop groupoid of $\mathcal{G}$ and an ordinary $(n-1)$-gerbe $\uptau^{\refl}_{\pi}(\hat{\lambda})$ on the unoriented quotient loop groupoid of $\mathcal{G}$. For $n =1,2$, we interpret the character theory (resp. centre) of the category of Real $\hat{\lambda}$-twisted $n$-vector bundles over $\hat{\mathcal{G}}$ in terms of flat sections of the $(n-1)$-vector bundle associated to $\uptau_{\pi}^{\refl}(\hat{\lambda})$ (resp. the Real $(n-1)$-vector bundle associated to $\uptau_{\pi}(\hat{\lambda})$). We relate our results to Real versions of twisted Drinfeld doubles and pointed fusion categories and to discrete torsion in orientifold string and M-theory.
\end{abstract}

\maketitle

\tableofcontents

\setcounter{footnote}{0}

\section*{Introduction}

The goal of this paper is to construct and compute twisted versions of loop transgression maps on the cohomology of finite groupoids and to explain their relevance to Real (categorical) representation theory, monoidal categories and discrete torsion in string and M-theory with orientifolds.

To put our results into context, recall that the ordinary loop transgression map $\uptau$ associates to a degree $n$ cohomology class on a sufficiently nice space $X$ a degree $(n-1)$ class on its free loop space $\Map(S^1, X)$. See, for example, \cite{brylinski1993}. When $X$ is replaced with a finite groupoid $\mathcal{G}$, which is the focus of this paper, Willerton \cite{willerton2008} constructed and computed a loop transgression cochain map $\uptau : C^{\bullet} (\mathcal{G}) \rightarrow C^{\bullet -1} (\Lambda \mathcal{G})$ for $\mathsf{U}(1)$-valued simplicial cochains. Here $\Lambda \mathcal{G} = \Hom_{\mathsf{Cat}}(B \mathbb{Z}, \mathcal{G})$ is the loop groupoid of $\mathcal{G}$, whose geometric realization $\vert \Lambda \mathcal{G} \vert$ is homotopy equivalent to $\Map(S^1, \vert \mathcal{G} \vert)$. In the setting of orbifolds, a similar map was constructed by Lupercio--Uribe \cite{lupercio2002}. The map $\uptau$ has found applications in many areas of mathematics, including representation theory and topological field theory. See, for example, \cite{freed1994}, \cite{brylinski1994}, \cite{willerton2008}, \cite{fuchs2014}, \cite{kong2019}.

Suppose now that $\pi: \hat{\mathcal{G}} \rightarrow B \mathbb{Z}_2$ is a finite $\mathbb{Z}_2$-graded groupoid, such as the classifying space of a $\mathbb{Z}_2$-graded group $\hat{\mathsf{G}} \rightarrow \mathbb{Z}_2$. Let $\mathcal{G} \rightarrow \hat{\mathcal{G}}$ be the associated double cover and $C^{\bullet + \pi}(\hat{\mathcal{G}})$ the $\pi$-twisted cochain complex of $\hat{\mathcal{G}}$. Define quotient loop groupoids $\Lambda_{\pi} \hat{\mathcal{G}} := \Lambda \mathcal{G} \git \mathbb{Z}_2$ and $\Lambda^{\refl}_{\pi} \hat{\mathcal{G}} := \Lambda \mathcal{G} \git \mathbb{Z}_2$ by $\mathbb{Z}_2$-actions on $\Lambda \mathcal{G}$ given by deck transformations of $\mathcal{G}$ and the diagonal action of deck transformations and reflection of the circle $B \mathbb{Z}$, respectively. Our first result is as follows.

\begin{thm}[{Theorems \ref{thm:quotTrans} and \ref{thm:oriTwistTrans}}]
\label{thm:twistTransIntro}
There exist twisted loop transgression maps
\[
\uptau^{\refl}_{\pi} : C^{\bullet + \pi} (\hat{\mathcal{G}}) \rightarrow C^{\bullet -1} (\Lambda^{\refl}_{\pi} \hat{\mathcal{G}})
\]
and
\[
\uptau_{\pi} : C^{\bullet + \pi} (\hat{\mathcal{G}}) \rightarrow C^{\bullet  -1 + \pi_{\Lambda_{\pi} \hat{\mathcal{G}}}} (\Lambda_{\pi} \hat{\mathcal{G}}),
\]
both of which are cochain maps and have explicit combinatorial expressions. Here $\pi_{\Lambda_{\pi} \hat{\mathcal{G}}} : \Lambda \mathcal{G} \rightarrow \Lambda_{\pi} \hat{\mathcal{G}}$ is the canonical double cover.
\end{thm}

For example, when $\hat{\theta} \in C^{2 + \pi}(\hat{\mathcal{G}})$ and $[\omega] \gamma \in C_1(\Lambda_{\pi}^{\refl} \hat{\mathcal{G}})$, we have
\[
\uptau^{\refl}_{\pi}(\hat{\theta})([\omega]\gamma) = \hat{\theta}([\gamma^{-1} \vert \gamma])^{\frac{\pi(\omega) -1}{2}} \frac{\hat{\theta}([\omega \gamma^{\pi(\omega)} \omega^{-1} \vert \omega])}{\hat{\theta}([\omega \vert \gamma^{\pi(\omega)}])}.
\]
The map $\uptau_{\pi}^{\refl}$ has already appeared, in the form of its explicit expressions in low degrees, in work on Real $2$-representation theory \cite{mbyoung2018c} and unoriented topological field theory \cite{mbyoung2019}, where its geometry was also foreshadowed. Theorem \ref{thm:twistTransIntro} gives an \textit{a priori} construction of $\uptau_{\pi}^{\refl}$ and $\uptau_{\pi}$, in all degrees, and establishes that they are cochain maps, which is important for applications in \cite{mbyoung2018c}, \cite{mbyoung2019} and difficult to verify directly in all but the simplest cases. Upon restriction along $\mathcal{G} \rightarrow \hat{\mathcal{G}}$, both maps $\uptau_{\pi}^{\refl}$, $\uptau_{\pi}$ recover Willerton's transgression map $\uptau$. Our proof of Theorem \ref{thm:twistTransIntro} is necessarily different from Willerton's proof for $\uptau$, however, since the latter relies on an explicit homotopy equivalence $\vert \Lambda \mathcal{G} \vert \sim \Map(S^1, \vert \mathcal{G} \vert)$ which is not equivariant for circle reflection. We instead work categorically, without passing to geometric realizations, where all constructions are equivariant.

In the remainder of the paper we explain the relevance of Theorem \ref{thm:twistTransIntro} to Real (categorical) representation theory and related areas. Following Atiyah \cite{atiyah1966}, the term ``Real" indicates a $\mathbb{C}$-linear object with a $\mathbb{C}$-anti-linear involution. We take a geometric approach and, using terminology of \cite{schreiber2007}, \cite{willerton2008}, interpret a twisted cocycle $\hat{\lambda} \in Z^{n +1 + \pi_{\hat{\mathcal{G}}}}(\hat{\mathcal{G}})$ as a (flat) Jandl (or Real) $n$-gerbe on $\hat{\mathcal{G}}$, that is, an $n$-gerbe on $\mathcal{G}$ with conjugation-twisted equivariance data for the double cover $\mathcal{G} \rightarrow \hat{\mathcal{G}}$. For $n=-1,0,1$, this is a Real $\mathsf{U}(1)$-valued function, a Real $\mathsf{U}(1)$-bundle and a Jandl gerbe on $\hat{\mathcal{G}}$, respectively. Jandl gerbes play an important role in unoriented topological field theory \cite{kapustin2017}, \cite{mbyoung2019}, orientifold string theory \cite{schreiber2007}, \cite{distler2011} and topological phases of matter with time reversal symmetry \cite{barkeshli2019}. The twisted transgression maps associate to the Jandl $n$-gerbe $\hat{\lambda}$ an ordinary $(n-1)$-gerbe $\uptau^{\refl}_{\pi}(\hat{\lambda})$ on $\Lambda_{\pi}^{\refl} \hat{\mathcal{G}}$ and a Jandl $(n-1)$-gerbe $\uptau_{\pi}(\hat{\lambda})$ on $\Lambda_{\pi} \hat{\mathcal{G}}$. Our results indicate that $\uptau^{\refl}_{\pi}(\hat{\lambda})$ and $\uptau_{\pi}(\hat{\lambda})$ control the character theory and centre, respectively, of the category of $\hat{\lambda}$-twisted $n$-vector bundles over $\hat{\mathcal{G}}$. We formulate precisely and prove these statements for $n \leq 2$.

In Section \ref{sec:degTwoCocyc}, after briefly treating the rather simple case of $n=0$, we study the case $n=1$. As is well-known \cite{freed2011}, \cite{moutuou2012}, \cite{fok2014}, complex vector bundles over $\hat{\mathcal{G}}$ can be twisted by a Jandl gerbe $\hat{\theta} \in Z^{2 + \pi_{\hat{\mathcal{G}}}}(\hat{\mathcal{G}})$. The collection of such twisted bundles forms an $\mathbb{R}$-linear category $\Vect_{\mathbb{C}}^{\hat{\theta}}(\hat{\mathcal{G}})$ whose Grothendieck group $K^{\hat{\theta}}(\hat{\mathcal{G}})$ is a twisted form of the $KR$-theory of $\mathcal{G}$. We prove that the character theory of $\Vect_{\mathbb{C}}^{\hat{\theta}}(\hat{\mathcal{G}})$ is naturally described in terms of the transgressed $\mathsf{U}(1)$-bundle $\uptau_{\pi}^{\refl}(\hat{\theta})$ on $\Lambda_{\pi}^{\refl} \hat{\mathcal{G}}$.

\begin{thm}[{Theorem \ref{thm:charIso}}]
\label{thm:charIsoIntro}
The Real character map induces an isomorphism of complex inner product spaces
\[
\chi: K^{\hat{\theta}}(\hat{\mathcal{G}}) \otimes_{\mathbb{Z}} \mathbb{C} \xrightarrow[]{\sim} \Gamma_{\Lambda_{\pi}^{\refl} \hat{\mathcal{G}}} (\uptau^{\refl}_{\pi}(\hat{\theta})_{\mathbb{C}})
\]
with target the space of flat sections of the associated complex line bundle $\uptau^{\refl}_{\pi}(\hat{\theta})_{\mathbb{C}}$.
\end{thm}

The proof of Theorem \ref{thm:charIsoIntro} uses Theorem \ref{thm:twistTransIntro} to identify $\Gamma_{\Lambda_{\pi}^{\refl} \hat{\mathcal{G}}} (\uptau^{\refl}_{\pi}(\hat{\theta})_{\mathbb{C}})$ as the appropriate space of class functions. For certain simple choices of $\hat{\mathcal{G}}$ and $\hat{\theta}$, Theorem \ref{thm:charIsoIntro} reduces to known results in the real and quaternionic representation theory of finite groups \cite{reynolds1971}, \cite{brocker1995}. As an application of Theorem \ref{thm:charIsoIntro}, we prove in Corollary \ref{cor:RealSchur} a Real version of a theorem of Schur, computing the number of simple objects of $\Vect_{\mathbb{C}}^{\hat{\theta}}(\hat{\mathcal{G}})$ as a $\hat{\theta}$-weighted count of Real conjugacy classes.

Unlike the ordinary (non-Real) case, the centre $Z(\Vect_{\mathbb{C}}^{\hat{\theta}}(\hat{\mathcal{G}}))$ and Grothendieck group $K^{\hat{\theta}}(\hat{\mathcal{G}})$ of $\Vect_{\mathbb{C}}^{\hat{\theta}}(\hat{\mathcal{G}})$ are not directly related. Instead,  $Z(\Vect_{\mathbb{C}}^{\hat{\theta}}(\hat{\mathcal{G}}))$ can be understood in terms of the twisted transgression map $\uptau_{\pi}$.

\begin{thm}[{Theorem \ref{thm:RealCentre}}]
\label{thm:RealCentreIntro}
There is an $\mathbb{R}$-algebra isomorphism
\[
Z(\Vect_{\mathbb{C}}^{\hat{\theta}}(\hat{\mathcal{G}})) \simeq \Gamma_{\Lambda_{\pi}\hat{\mathcal{G}}}(\uptau_{\pi}(\hat{\theta})^{-1}_{\mathbb{C}}),
\]
where the right hand side is the space of flat sections of the Real line bundle $\uptau_{\pi}(\hat{\theta})^{-1}_{\mathbb{C}}$.
\end{thm}

In Section \ref{sec:degThreeCocyc} we study $2$-categorical analogues of the results of Section \ref{sec:degTwoCocyc}. Fix a Jandl $2$-gerbe $\hat{\eta} \in Z^{3 + \pi}(B\hat{\mathsf{G}})$. Motivated by Willerton's interpretation of the twisted Drinfeld double $D^{\eta}(\mathsf{G})$, $\eta \in Z^3(B \mathsf{G})$, of a finite group $\mathsf{G}$ as the $\uptau(\eta)$-twisted groupoid algebra of $\Lambda B \mathsf{G}$ \cite{willerton2008}, we use $\uptau_{\pi}^{\refl}(\hat{\lambda})$ and $\uptau_{\pi}(\hat{\lambda})$ to define thickened twisted Drinfeld doubles $\mathbb{D}^{\hat{\eta}}(\hat{\mathsf{G}})$ and $D^{\hat{\eta}}(\hat{\mathsf{G}})$, respectively. These are complex vector spaces with possibly anti-linear associative multiplications which, for $\mathsf{G} = \ker(\pi)$, contain $D^{\eta}(\mathsf{G})$ as a subalgebra. Characters of $D^{\hat{\eta}}(\hat{\mathsf{G}})$-modules, which because of their connection with two dimensional topological field theory we call twisted one-loop characters, are shown in Proposition \ref{prop:charThickDD} to give an interesting extension of twisted elliptic characters of $\mathsf{G}$ by a Klein bottle sector. Next, we construct from $\hat{\eta}$ a Real fusion category $\Vect_{\mathbb{C}}^{\hat{\eta}}(\hat{\mathsf{G}})$, which we view as a categorified twisted Real group algebra of $\mathsf{G}$, and show in Proposition \ref{prop:RPFusGrpd} that categories of this form exhaust Real pointed fusion categories. We also identify their Drinfeld centres, giving the following $2$-categorical version of Theorem \ref{thm:RealCentreIntro}.

\begin{thm}[{Theorem \ref{thm:realFusCat}}]
\label{thm:realFusCatIntro}
There is an $\mathbb{R}$-linear monoidal equivalence
\[
Z_{D}(\Vect_{\mathbb{C}}^{\hat{\eta}}(\hat{\mathsf{G}}))
\simeq
\Vect_{\mathbb{C}}^{\uptau_{\pi}(\hat{\eta})^{-1}}(\Lambda_{\pi} B \hat{\mathsf{G}}).
\]
\end{thm}

We prove that the monoidal structure of $\Vect_{\mathbb{C}}^{\uptau_{\pi}(\hat{\eta})}(\Lambda_{\pi} B \hat{\mathsf{G}})$ arising from Theorem \ref{thm:realFusCatIntro} is induced by a Real quasi-bialgebra structure on $\mathbb{D}^{\hat{\eta}}(\hat{\mathsf{G}})$, providing a tighter connection with $D^{\eta}(\mathsf{G})$. In Section \ref{sec:Real2CharThy} we explain that Real $\Vect_{\mathbb{C}}^{\hat{\eta}}(\hat{\mathsf{G}})$-module categories recover the bicategory $2\Vect_{\mathbb{C}}^{\hat{\eta}}(\hat{\mathcal{G}})$ of $\hat{\eta}$-twisted $2$-vector bundles over $B \hat{\mathsf{G}}$, as developed in \cite{mbyoung2018c} in the form of Real $2$-representation theory. One of the main results of \cite{mbyoung2018c} is the existence of a categorical version of the Real character map from Theorem \ref{thm:charIsoIntro}. This Real categorical character theory is most naturally formulated in terms of $D^{\hat{\eta}}(\hat{\mathsf{G}})$-modules and twisted one-loop characters, illustrating the character theoretic meaning of $\uptau_{\pi}^{\refl}(\hat{\eta})$. Finally, in Section \ref{sec:discTor} we interpret $\uptau_{\pi}^{\refl}$ in terms of discrete torsion phase factors of non-orientable $2$- and $3$-manifolds in string and M-theory with orientifolds, thereby recovering and providing a new perspective on computations of Bantay \cite{bantay2003} and Sharpe \cite{sharpe2011}.

In this paper we have restricted attention to finite groupoids, both for simplicity and because of our applications. In work in progress, we study more general twisted transgression maps in the geometric setting of topological stacks.

\subsubsection*{Acknowledgements}
The authors thank Gr\'{e}gory Ginot and Mahmoud Zeinalian for discussions. M. B. Y. thanks Chi-Kwong Fok, Konrad Waldorf and Siye Wu for discussions. The authors are grateful to the Max Planck Institute for Mathematics for its hospitality and financial support during the preparation of this paper.

\section{Loop groupoids and their quotients}
\addtocontents{toc}{\protect\setcounter{tocdepth}{2}}
\label{sec:background}

\subsection{Essentially finite groupoids}
\label{sec:groupoids}

We recall some basic results about groupoids.

A groupoid $\mathcal{G}$ is said to be essentially finite if its set of connected components, $\pi_0(\mathcal{G})$, is finite and each object of $\mathcal{G}$ has finitely many automorphisms. Unless mentioned otherwise, all groupoids in this paper are assumed to be essentially finite.

\begin{Ex}
Let $\mathsf{G}$ be a group and $X$ a (left) $\mathsf{G}$-set. The action groupoid $X \git \mathsf{G}$ has objects $X$ and morphisms $\Hom_{X \git \mathsf{G}}(x_1,x_2)
=
\{ g \in \mathsf{G} \mid g x_1 = x_2\}$. Morphisms are composed using multiplication in $\mathsf{G}$. 
The groupoid $X \git \mathsf{G}$ is essentially finite if and only if $X$ has finitely many $\mathsf{G}$-orbits and each orbit has finite stabilizer.

When $X$ consists of a single point, we write $B \mathsf{G}$ in place of $X \git \mathsf{G}$.  
\end{Ex}

Let $C_{\bullet}(\mathcal{G}) := C_{\bullet}(\mathcal{G}; \mathbb{Z})$ be the complex of simplicial chains on $\mathcal{G}$. Explicitly, $C_n(\mathcal{G})$ is the free abelian group generated by symbols $[g_n \vert \cdots \vert g_1]$, where $x_1 \xrightarrow[]{g_1} \cdots \xrightarrow[]{g_n} x_{n+1}$ is a diagram in $\mathcal{G}$. The differential of $C_{\bullet}(\mathcal{G})$ is
\[
\partial [g_n \vert \cdots \vert g_1] = [g_{n-1} \vert \cdots \vert g_1] + \sum_{j=1}^{n-1} (-1)^{n-j} [g_n \vert \cdots \vert g_{j+1} g_j \vert \cdots \vert g_1] + (-1)^n [g_n \vert \cdots \vert g_2].
\]

Let $\mathsf{A}$ be an abelian group, written multiplicatively, and $\kappa: \mathcal{G} \rightarrow B \Aut (\mathsf{A})$ a functor. The complex of $\kappa$-twisted cochains on $\mathcal{G}$ is the abelian group $\Hom_{\mathbb{Z}}(C_{\bullet}(\mathcal{G}), \mathsf{A})$. The differential $d$ sends $\lambda \in \Hom_{\mathbb{Z}}(C_{n-1}(\mathcal{G}), \mathsf{A})$ to the $n$-cochain
\begin{multline*}
d \lambda ([g_n \vert \cdots \vert g_1])
=
\kappa(g_n) \cdot \lambda([g_{n-1} \vert \cdots \vert g_1])
\prod_{j=1}^{n-1} \lambda([g_n \vert \cdots \vert g_{j+1} g_j\vert \cdots \vert g_1])^{(-1)^{n-j}} \\
\times \lambda([g_n \vert \cdots \vert g_2])^{(-1)^n}.
\end{multline*}
The inclusion of the subcomplex $C^{\bullet + \kappa}(\mathcal{G}; \mathsf{A}): = \Hom^{\textnormal{n}}_{\mathbb{Z}}(C_{\bullet}(\mathcal{G}), \mathsf{A})$ of normalized cochains of $\Hom_{\mathbb{Z}}(C_{\bullet}(\mathcal{G}), \mathsf{A})$ is a quasi-isomorphism. Denote by $Z^{\bullet + \kappa}(\mathcal{G}; \mathsf{A})$ and $H^{\bullet + \kappa}(\mathcal{G}; \mathsf{A})$ the cocycles and cohomology of  $C^{\bullet + \kappa}(\mathcal{G}; \mathsf{A})$, respectively. Write $C^{\bullet + \kappa}(\mathcal{G})$ for $C^{\bullet+ \kappa}(\mathcal{G}; \mathsf{A})$ if $\mathsf{A}$ is fixed and omit $\kappa$ from the notation if it is trivial.

\subsection{Essentially finite groupoids over \texorpdfstring{$B\mathbb{Z}_2$}{}}
\label{sec:groupoidsOverBZ2}

Let $\mathbb{Z}_2$ be the multiplicative group $\{ \pm 1\}$. We sometimes denote its non-identity element by $\zeta$.

A groupoid over $B \mathbb{Z}_2$, or a $\mathbb{Z}_2$-graded groupoid, is a functor $\pi_{\hat{\mathcal{G}}}: \hat{\mathcal{G}} \rightarrow B \mathbb{Z}_2$. Write $\pi$ for $\pi_{\hat{\mathcal{G}}}$ if it will not cause confusion. The degree of a morphism $\omega \in \Mor(\hat{\mathcal{G}})$ is $\pi(\omega) \in \mathbb{Z}_2$. We assume that $\pi$ is strongly non-trivial in the sense that, for each $x \in \hat{\mathcal{G}}$, there exists a morphism of degree $-1$ with source $x$. This assumption is easily removed at the cost of slightly complicating the statements which follow.

There is a canonical decomposition $\pi_0(\hat{\mathcal{G}}) = \pi_0(\hat{\mathcal{G}})_{-1} \sqcup \pi_0(\hat{\mathcal{G}})_{1}$, with $\pi_0(\hat{\mathcal{G}})_{-1}$ the subset of connected components consisting of objects which have at least one automorphism of degree $-1$. By strong non-triviality, for each representative $x \in \pi_0(\hat{\mathcal{G}})_1$, we can choose a morphism of degree $-1$ with source $x$, whose target we denote by $\overline{x}$. Note that $x \neq \overline{x}$. Let $\hat{\mathcal{G}}_{\{x,\overline{x}\}} \subset \hat{\mathcal{G}}$ be the full subcategory on $\{x,\overline{x}\}$.

\begin{Prop}
\label{prop:grpdDecomp}
There is an equivalence of $\mathbb{Z}_2$-graded groupoids
\[
%\label{eq:grpdDecomp}
\hat{\mathcal{G}}
\simeq
\bigsqcup_{x \in \pi_0(\hat{\mathcal{G}})_{-1}} B \Aut_{\hat{\mathcal{G}}}(x) \sqcup \bigsqcup_{x \in \pi_0(\hat{\mathcal{G}})_1} \hat{\mathcal{G}}_{\{x,\overline{x}\}}.
\]
\end{Prop}

\begin{proof}
It suffices to prove the statement for connected $\hat{\mathcal{G}}$. The case $\pi_0(\hat{\mathcal{G}}) = \pi_0(\hat{\mathcal{G}})_{-1}$ is left to the reader. Suppose that $\pi_0(\hat{\mathcal{G}}) = \pi_0(\hat{\mathcal{G}})_1$. For each $y \in \hat{\mathcal{G}}$, there exists a morphism $g_y$ of degree $+1$ from $y$ to $x$ or from $y$ to $\overline{x}$, but not both. Fix a choice of such morphisms. Then the functor $\hat{\mathcal{G}} \rightarrow \hat{\mathcal{G}}_{\{x,\overline{x}\}}$ which sends an object $y$ to the target of $g_y$ and a morphism $\omega: y_1 \rightarrow y_2$ to $g_{y_2} \omega g_{y_1}^{-1}$ is quasi-inverse to the inclusion $\hat{\mathcal{G}}_{\{x,\overline{x}\}} \hookrightarrow \hat{\mathcal{G}}$ and is compatible with the structure maps to $B \mathbb{Z}_2$.
\end{proof}

The functor $\pi: \hat{\mathcal{G}} \rightarrow B \mathbb{Z}_2$ classifies a double cover $\pi: \mathcal{G}_{\pi} \rightarrow \hat{\mathcal{G}}$. The use of $\pi$ for both the classifying map and the double cover should not cause confusion. We often write $\mathcal{G}$ for $\mathcal{G}_{\pi}$. The functor $\pi: \mathcal{G} \rightarrow \hat{\mathcal{G}}$ admits the following explicit model \cite[\S 10.4]{freed2013b}. The objects and morphisms of $\mathcal{G}$ are $\Obj(\hat{\mathcal{G}}) \times \mathbb{Z}_2$ and
\[
\Hom_{\mathcal{G}}((x_1, \epsilon_1), (x_2, \epsilon_2)) = \{ \omega \in \Hom_{\hat{\mathcal{G}}}(x_1,x_2) \mid \pi(\omega) \epsilon_1 = \epsilon_2 \}.
\]
Morphisms are composed as in $\hat{\mathcal{G}}$. The functor $\pi$ sends a diagram $(x_1, \epsilon_1) \xrightarrow[]{\omega} (x_2, \epsilon_2)$ in $\mathcal{G}$ to $x_1 \xrightarrow[]{\omega} x_2$. The deck transformation $\sigma_{\mathcal{G}}: \mathcal{G} \rightarrow \mathcal{G}$ is the strict involution which sends a diagram $(x_1, \epsilon_1) \xrightarrow[]{\omega} (x_2, \epsilon_2)$ to $(x_1, -\epsilon_1) \xrightarrow[]{\omega} (x_2, -\epsilon_2)$.

\begin{Ex}
Let $1 \rightarrow \mathsf{G} \rightarrow \hat{\mathsf{G}} \xrightarrow[]{\pi} \mathbb{Z}_2 \rightarrow 1$ be an exact sequence of groups. We call $\hat{\mathsf{G}}$ a $\mathbb{Z}_2$-graded group. Let $X$ be a $\hat{\mathsf{G}}$-set. There is an induced $\mathbb{Z}_2$-grading $\pi: X \git \hat{\mathsf{G}} \rightarrow B \mathbb{Z}_2$. The choice of an element $\varsigma \in \hat{\mathsf{G}} \setminus \mathsf{G}$ identifies $X \git \mathsf{G} \rightarrow X \git \hat{\mathsf{G}}$ with the double cover $(X \git \hat{\mathsf{G}})_{\pi} \rightarrow X \git \hat{\mathsf{G}}$, under which the deck transformation $\sigma^{\varsigma}_{X \git \mathsf{G}} : X \git \mathsf{G} \rightarrow X \git \mathsf{G}$ is the weak involution given on morphisms by $(x_1 \xrightarrow[]{g} x_2) \mapsto \varsigma x_1 \xrightarrow[]{\varsigma g \varsigma^{-1}} \varsigma x_2$. The action of $\varsigma^2 \in \mathsf{G}$ defines a natural isomorphism $1_{X \git \hat{\mathsf{G}}} \Rightarrow (\sigma_{X \git \mathsf{G}}^{\varsigma})^2$.
\end{Ex}

For an abelian group $\mathsf{A}$, let $\mathbb{Z}_2 \hookrightarrow \Aut(\mathsf{A})$ be the subgroup generated by inversion. A $\mathbb{Z}_2$-grading $\pi_{\hat{\mathcal{G}}}: \hat{\mathcal{G}} \rightarrow B \mathbb{Z}_2$ therefore defines a twisted cochain complex $C^{\bullet + \pi_{\hat{\mathcal{G}}}}(\hat{\mathcal{G}})$. A description of $C^{\bullet + \pi_{\hat{\mathcal{G}}}}(\hat{\mathcal{G}})$ in terms of the double cover $\mathcal{G}$ is as follows. Let $[\omega_n \vert \cdots \vert \omega_1]_{\epsilon_1} \in C_n(\mathcal{G})$ be the chain associated to the diagram
\[
(x_1,\epsilon_1) \xrightarrow[]{\omega_1}  \cdots \xrightarrow[]{\omega_n} (x_{n+1},\epsilon_{n+1})
\]
in $\mathcal{G}$. The notation is unambiguous, since $\epsilon_i = \pi(\omega_{\leq i-1}) \epsilon_1$, where $\omega_{\leq i} := \omega_i \cdots \omega_1$. The complex $C_{\bullet}(\mathcal{G})$ is a $\mathbb{Z}[\mathbb{Z}_2]$-module via
\begin{equation}
\label{eq:groupActionChain}
\zeta \cdot [\omega_n \vert \cdots \vert \omega_1]_{\epsilon_1} = [\omega_n \vert \cdots \vert \omega_1]_{-\epsilon_1}.
\end{equation}
Write $\mathsf{A}$ and $\mathsf{A}_-$ for the trivial and non-trivial $\mathbb{Z}[\mathbb{Z}_2]$-module structures on $\mathsf{A}$, respectively. For example, $c_0 + c_1 \zeta \in \mathbb{Z}[\mathbb{Z}_2]$ acts on $\mathsf{A}_-$ by $z \mapsto z^{c_0 - c_1}$.

\begin{Lem}
\label{lem:topVsAlg}
There are mutually inverse cochain isomorphisms
\[
\Phi_-: 
C^{\bullet + \pi_{\hat{\mathcal{G}}}}(\hat{\mathcal{G}}) \leftrightarrows \Hom^{\textnormal{n}}_{\mathbb{Z}[\mathbb{Z}_2]}(C_{\bullet}(\mathcal{G}),  \mathsf{A}_-)
: \Psi_-
\]
given by
\[
\Phi_-(\hat{\lambda})([\omega_n \vert \cdots \vert \omega_1]_{\epsilon_1})
=
\hat{\lambda}([\omega_n \vert \cdots \vert \omega_1])^{\epsilon_{n+1}}
\]
and
\[
\Psi_-(\hat{\mu})([\omega_n \vert \cdots \vert \omega_1])
=
\hat{\mu}([\omega_n \vert \cdots \vert \omega_1]_{\pi(\omega_{\leq n})}).
\]
\end{Lem}

\begin{proof}
This is a direct calculation.
\end{proof}

There is an untwisted version of Lemma \ref{lem:topVsAlg}, in which $\mathsf{A}$ replaces $\mathsf{A}_-$ and maps $\Phi$ and $\Psi$ are defined as for $\Phi_-$ and $\Psi_-$, but with the signs removed.

\subsection{Loop groupoids}
\label{sec:loopGroupoids}

Following Willerton \cite{willerton2008} (see Lupercio--Uribe \cite{lupercio2002}), the loop groupoid of an essentially finite groupoid $\mathcal{G}$ is defined to be the functor category
\[
\Lambda \mathcal{G}
:=
\Hom_{\mathsf{Cat}}(B \mathbb{Z}, \mathcal{G}).
\]
Objects of $\Lambda \mathcal{G}$ can be identified with loops $(x,\gamma)$ in $\mathcal{G}$, that is, morphisms $\gamma: x \rightarrow x$, in which case a morphism $(x_1, \gamma_1) \rightarrow (x_2,\gamma_2)$ is a morphism $g: x_1 \rightarrow x_2$ in $\mathcal{G}$ which satisfies $\gamma_2 = g \gamma_1 g^{-1}$. Composition of morphisms is as in $\mathcal{G}$. The groupoid $\Lambda \mathcal{G}$ is essentially finite if $\mathcal{G}$ is so. Since $\mathcal{G}$ is essentially finite, its loop groupoid and inertia groupoid coincide, hence our notation.

It is known (\cite[Theorem 2]{willerton2008}, \cite[Proposition 6.1.1]{lupercio2002}) that the geometric realization of $\Lambda \mathcal{G}$ is homotopy equivalent to the free loop space of the geometric realization of $\mathcal{G}$, that is, $\vert \Lambda \mathcal{G} \vert \sim \Map(S^1, \vert \mathcal{G} \vert)$.

\subsection{Quotients of loop groupoids}
\label{sec:tLGGroupoids}

Let $\pi: \hat{\mathcal{G}} \rightarrow B \mathbb{Z}_2$ be an essentially finite $\mathbb{Z}_2$-graded groupoid with double cover $\mathcal{G} \rightarrow \hat{\mathcal{G}}$. We introduce two quotients of $\Lambda \mathcal{G}$.

\begin{Def}
The quotient loop groupoid of $\mathcal{G}$ is the groupoid $\Lambda_{\pi} \hat{\mathcal{G}}$ with objects the loops $(x, \gamma)$ of degree $+1$ in $\hat{\mathcal{G}}$ and morphisms
\[
\Hom_{\Lambda_{\pi} \hat{\mathcal{G}}} ((x_1,\gamma_1), (x_2,\gamma_2)) = \{ \omega \in \Hom_{\hat{\mathcal{G}}}(x_1, x_2) \mid \gamma_2 = \omega \gamma_1 \omega^{-1}\}.
\]
\end{Def}

There is a strongly non-trivial grading $\pi_{\Lambda_{\pi} \hat{\mathcal{G}}} : \Lambda_{\pi} \hat{\mathcal{G}} \rightarrow B \mathbb{Z}_2$ which sends a morphism $\omega$ to $\pi(\omega)$. To identify the associated double cover, let $p : \Lambda \mathcal{G} \rightarrow \Lambda_{\pi} \hat{\mathcal{G}}$ be the functor which sends a diagram $((x_1, \epsilon_1), \gamma_1) \xrightarrow[]{\omega} ((x_2, \epsilon_2), \gamma_2)$ to $(x_1, \gamma_1) \xrightarrow[]{\omega} (x_2, \gamma_2)$.

\begin{Lem}
\label{lem:loopGrpdCover}
The double cover classified by $\pi_{\Lambda_{\pi} \hat{\mathcal{G}}}$ is equivalent to $p$.
\end{Lem}

\begin{proof}
Let $\Lambda \mathcal{G} \rightarrow 
(\Lambda_{\pi} \hat{\mathcal{G}})_{\pi_{\Lambda_{\pi} \hat{\mathcal{G}}}}$ be the functor defined on diagrams by
\[
\left[
((x_1, \epsilon_1),\gamma_1) \xrightarrow[]{\omega} ((x_2, \epsilon_2),\gamma_2)
\right]
\mapsto ((x_1, \gamma_1),\epsilon_1) \xrightarrow[]{\omega}
((x_2, \gamma_2),\epsilon_2). 
\]
This is an equivalence and is compatible with the structure maps to $\Lambda_{\pi} \hat{\mathcal{G}}$.
\end{proof}

Because of Lemma \ref{lem:loopGrpdCover}, we henceforth write $\pi_{\Lambda_{\pi} \hat{\mathcal{G}}}$ for $p$. Under the equivalence of Lemma \ref{lem:loopGrpdCover} (and its obvious inverse functor), the deck transformation $\sigma_{\Lambda \mathcal{G}}: \Lambda \mathcal{G} \rightarrow \Lambda \mathcal{G}$ is the strict involution given on objects by $((x, \epsilon), \gamma) \mapsto ((x, -\epsilon), \gamma)$ and on morphisms by the identity or, in terms of functors,
\[
\sigma_{\Lambda \mathcal{G}} : \Hom_{\mathsf{Cat}}(B \mathbb{Z}, \mathcal{G}) \rightarrow \Hom_{\mathsf{Cat}}(B \mathbb{Z}, \mathcal{G}),
\qquad
F \mapsto \sigma_{\mathcal{G}} \circ F.
\]
In particular, $\Lambda_{\pi} \hat{\mathcal{G}}$ is the quotient of $\Lambda \mathcal{G}$ by $\sigma_{\Lambda \mathcal{G}}$.

Next, we define a modification of $\Lambda_{\pi} \hat{\mathcal{G}}$ which incorporates reflection of the circle.

\begin{Def}
The unoriented quotient loop groupoid of $\mathcal{G}$ is the groupoid $\Lambda^{\refl}_{\pi} \hat{\mathcal{G}}$ with objects the loops $(x, \gamma)$ of degree $+1$ in $\hat{\mathcal{G}}$ and morphisms
\[
\Hom_{\Lambda^{\refl}_{\pi} \hat{\mathcal{G}}} ((x_1,\gamma_1), (x_2,\gamma_2)) = \{ \omega \in \Hom_{\hat{\mathcal{G}}}(x_1, x_2) \mid \gamma_2 = \omega \gamma_1^{\pi(\omega)} \omega^{-1}\}.
\]
\end{Def}

There is a strongly non-trivial grading $\pi_{\Lambda^{\refl}_{\pi} \hat{\mathcal{G}}} : \Lambda^{\refl}_{\pi} \hat{\mathcal{G}} \rightarrow B \mathbb{Z}_2$. Let $p^{\refl}: \Lambda \mathcal{G} \rightarrow \Lambda_{\pi}^{\refl} \hat{\mathcal{G}}$ be the functor given by $ \big[ ((x_1, \epsilon_1), \gamma_1) \xrightarrow[]{\omega} ((x_2, \epsilon_2), \gamma_2) \big] \mapsto (x_1, \gamma^{\epsilon_1}_1) \xrightarrow[]{\omega} (x_2, \gamma^{\epsilon_2}_2)$.

\begin{Lem}
\label{lem:refLoopGrpdCover}
The double cover classified by $\pi_{\Lambda^{\refl}_{\pi} \hat{\mathcal{G}}}$ is equivalent to $p^{\refl}$.
\end{Lem}

\begin{proof}
The equivalence $\Lambda \mathcal{G} \rightarrow 
(\Lambda_{\pi}^{\refl} \hat{\mathcal{G}})_{\pi_{\Lambda_{\pi}^{\refl} \hat{\mathcal{G}}}}$ is defined by
\[
\left[
((x_1, \epsilon_1),\gamma_1) \xrightarrow[]{\omega} ((x_2, \epsilon_2),\gamma_2)
\right]
\mapsto ((x_1, \gamma_1^{\epsilon_1}),\epsilon_1) \xrightarrow[]{\omega}
((x_2, \gamma_2^{\epsilon_2}),\epsilon_2). 
\]
\end{proof}

We henceforth write $\pi_{\Lambda_{\pi}^{\refl} \hat{\mathcal{G}}}$ for $p^{\refl}$. Under the equivalence of Lemma \ref{lem:refLoopGrpdCover}, the deck transformation $\sigma_{\Lambda \mathcal{G}}^{\refl}: \Lambda \mathcal{G} \rightarrow \Lambda \mathcal{G}$ is the strict involution $\sigma_{\Lambda \mathcal{G}}^{\refl}((x, \epsilon), \gamma) = ((x, -\epsilon), \gamma^{-1})$ or, in terms of functors, $\sigma_{\Lambda \mathcal{G}}^{\refl}(F) =  \sigma_{\mathcal{G}} \circ F \circ B \iota$, where $\iota: \mathbb{Z} \rightarrow \mathbb{Z}$ is negation.

\begin{Ex}
Let $\hat{\mathsf{G}}$ be a $\mathbb{Z}_2$-graded group. The Real conjugation action of $\hat{\mathsf{G}}$ on $\mathsf{G}$ is $\omega \cdot g = \omega g^{\pi(\omega)}  \omega^{-1}$, $(\omega,g) \in \hat{\mathsf{G}} \times \mathsf{G}$. With this notation, there are equivalences
\[
\Lambda B \mathsf{G} \simeq \mathsf{G} \git \mathsf{G}, \qquad \Lambda_{\pi} B \hat{\mathsf{G}} \simeq \mathsf{G} \git \hat{\mathsf{G}}, \qquad \Lambda^{\refl}_{\pi} B \hat{\mathsf{G}} \simeq \mathsf{G} \git_{\textnormal{R}} \hat{\mathsf{G}},
\]
where the group actions on $\mathsf{G}$ are by (Real) conjugation. The choice of an element $\varsigma \in \hat{\mathsf{G}} \setminus \mathsf{G}$ identifies the double covers
\[
\Lambda B \mathsf{G} \rightarrow \Lambda_{\pi} B \hat{\mathsf{G}},
\qquad
\Lambda B \mathsf{G} \rightarrow \Lambda^{\refl}_{\pi} B \hat{\mathsf{G}}
\]
with the canonical functors
\[
\mathsf{G} \git \mathsf{G} \rightarrow \mathsf{G} \git \hat{\mathsf{G}},
\qquad
\mathsf{G} \git \mathsf{G} \rightarrow \mathsf{G} \git_{\textnormal{R}} \hat{\mathsf{G}}.
\]
The deck transformations are given on objects by $\sigma_{\Lambda_{\pi} B \hat{\mathsf{G}}}(\gamma) = \varsigma \gamma \varsigma^{-1}$ and $\sigma_{\Lambda^{\refl}_{\pi} B \hat{\mathsf{G}}}(\gamma) = \varsigma \gamma^{-1} \varsigma^{-1}$ and on morphisms by $\sigma_{\Lambda^{(\refl)}_{\pi} B \hat{\mathsf{G}}}(\omega) = \varsigma \omega \varsigma^{-1}$.
\end{Ex}

\section{Twisted loop transgression}
\label{sec:essFinGroupoids}

\subsection{Oriented loop transgression}
\label{sec:transGrpd}

We begin by recalling the ordinary (oriented) loop transgression map in the setting of essentially finite groupoids. For detailed discussions, the reader is referred to \cite{willerton2008} or, in the geometric setting, \cite{brylinski1993}, \cite{lupercio2002}.

Let $\mathcal{G}$ be an essentially finite groupoid. The evaluation functor $\ev : B \mathbb{Z} \times \Lambda \mathcal{G} \rightarrow \mathcal{G}$ is given on morphisms by
\[
\left[
(x_1,\gamma_1) \xrightarrow[]{(n,g)} (x_2,\gamma_2)
\right]
\mapsto
x_1 \xrightarrow[]{g \gamma_1^n = \gamma_2^n g} x_2.
\]
Let $\pr_{\Lambda \mathcal{G}} : B \mathbb{Z} \times \Lambda \mathcal{G} \rightarrow \Lambda \mathcal{G}$ be the projection and consider the span of groupoids
\begin{equation}
\label{diag:oriLoopDiagGrpd}
\begin{tikzpicture}[baseline= (a).base]
\node[scale=1] (a) at (0,0){
\begin{tikzcd}[column sep={8.0em,between origins}, row sep={2.0em,between origins}]
{} & B \mathbb{Z} \times \Lambda \mathcal{G}  \arrow{dl}[above]{\ev} \arrow{dr}[above]{\pr_{\Lambda \mathcal{G}}} & {} \\
\mathcal{G} & {} & \Lambda \mathcal{G}.
\end{tikzcd}
};
\end{tikzpicture}
\end{equation}
%\begin{equation}
%\label{diag:oriLoopDiagGrpd}
%\begin{tikzpicture}[baseline= (a).base]
%\node[scale=1] (a) at (0,0){
%\begin{tikzcd}[column sep={8.0em,between origins}, row sep={2.0em,between origins}]
%\mathcal{G} & B \mathbb{Z} \times \Lambda \mathcal{G}  \arrow{l}[above]{\ev} \arrow{r}[above]{\pr_{\Lambda \mathcal{G}}} & \Lambda \mathcal{G}.
%\end{tikzcd}
%};
%\end{tikzpicture}
%\end{equation}
Passing to geometric realizations gives a diagram which is homotopy equivalent to the standard evaluation-projection correspondence for $\Map(S^1, \vert \mathcal{G} \vert)$.

Let $[1] : = [\bullet \xrightarrow[]{1} \bullet] \in C_1(B \mathbb{Z})$, which we view as a fundamental cycle of $\vert B \mathbb{Z} \vert \sim S^1$. The composition
\begin{equation}
\label{eq:homologyTrans}
C_{\bullet}(\Lambda \mathcal{G})
\xrightarrow[]{[1] \otimes -}
C_1(B \mathbb{Z}) \otimes_{\mathbb{Z}} C_{\bullet}(\Lambda \mathcal{G})
\xrightarrow{\mathsf{EZ}}
C_{\bullet+1}(B \mathbb{Z} \times \Lambda \mathcal{G})
\xrightarrow[]{\ev_*}
C_{\bullet+1}(\mathcal{G})
\end{equation}
defines the chain level loop transgression map. Here $\mathsf{EZ}$ is the Eilenberg--Zilber shuffle map. Let $\mathsf{ez}_{[1]}$ be the composition $\mathsf{EZ} \circ ([1] \otimes -)$. Then pushforward along $\pr_{\Lambda \mathcal{G}}$ is the map
\begin{equation}
\label{eq:untwistPush}
\pr_{\Lambda \mathcal{G} !} : C^{\bullet} (B \mathbb{Z} \times \Lambda \mathcal{G}) \rightarrow C^{\bullet-1}(\Lambda \mathcal{G}), \qquad \lambda \mapsto \lambda \circ \mathsf{ez}_{[1]}
\end{equation}
and the loop transgression map is defined by push-pull along the diagram \eqref{diag:oriLoopDiagGrpd}:
\[
\uptau : C^{\bullet}(\mathcal{G}) \xrightarrow[]{\pr_{\Lambda \mathcal{G} !} \circ \ev^*} C^{\bullet-1}(\Lambda \mathcal{G}).
\]
The map $\uptau$ anti-commutes with the differentials, $d \uptau(\lambda) = \uptau(d\lambda)^{-1}$, and so descends to a map on cocycles and cohomologies.

To compute $\uptau$, let $[g_n \vert \cdots \vert g_1]\gamma \in C_n(\Lambda \mathcal{G})$ be the chain associated to the diagram $\gamma = \gamma_1 \xrightarrow[]{g_1} \gamma_2 \xrightarrow[]{g_2} \cdots \xrightarrow[]{g_n} \gamma_{n+1}$ in $\Lambda \mathcal{G}$. Then we have
\[
\mathsf{ez}_{[1]}([g_n \vert \cdots \vert g_1]\gamma)
=
\sum_{i=0}^n (-1)^{n-i}
[g_n \vert \cdots \vert g_{i+1} \vert 1 \vert g_i \vert \cdots \vert g_1] \gamma_1,
\]
%\[
%%\label{eq:bigChain}
%\begin{tikzpicture}[baseline= (a).basescale=0.85, every node/.style={scale=0.85}]
%\node[scale=1] (a) at (0,0){
%\begin{tikzcd}[column sep={6.0em,between origins}, row sep={2.5em,between origins}]
%{} & {} & {} & (x_{i+1},\gamma_{i+1}) \arrow{r}[above]{g_{i+1}} & (x_{i+2},\gamma_{i+2}) \arrow{r}[above]{g_{i+2}} & \cdots \arrow{r}[above]{g_n} & (x_{n+1},\gamma_{n+1}) \\
%\gamma=(x_1,\gamma_1) \arrow{r}[below]{g_1} & (x_2,\gamma_2) \arrow{r}[below]{g_2} & \cdots \arrow{r}[below]{g_i} & (x_{i+1},\gamma_{i+1}) \arrow{u}[right]{1}
%\end{tikzcd}
%};
%\end{tikzpicture}
%,
%\]
%\[
%%\label{eq:bigChain}
%\sum_{i=0}^n (-1)^{n-i-1}
%\Bigg[
%\begin{tikzpicture}[baseline= (a).basescale=0.85, every node/.style={scale=0.85}]
%\node[scale=1] (a) at (0,0){
%\begin{tikzcd}[column sep={5.0em,between origins}, row sep={2.5em,between origins}]
%{} & {} & {} & \gamma_{i+1} \arrow{r}[above]{g_{i+1}} & \gamma_{i+2} \arrow{r}[above]{g_{i+2}} & \cdots \arrow{r}[above]{g_n} & \gamma_{n+1} \\
%\gamma=\gamma_1 \arrow{r}[below]{g_1} & \gamma_2 \arrow{r}[below]{g_2} & \cdots \arrow{r}[below]{g_i} & \gamma_{i+1} \arrow{u}[right]{1}
%\end{tikzcd}
%};
%\end{tikzpicture}
%\Bigg],
%\]
where we have introduced the shorthands $g_i = (0,g_i)$ and $1 = (1,\textnormal{id}_x)$ for morphisms in $B \mathbb{Z} \times \Lambda \mathcal{G}$.
%We have
%\[
%\ev_* \mathsf{ez}_i([g_n \vert \cdots \vert g_1]\gamma)
%=
%[g_n \vert \cdots \vert g_{i+1} \vert \gamma_{i+1} \vert g_i \vert \cdots \vert g_1].
%\]
The result of applying the map \eqref{eq:homologyTrans} to $[g_n \vert \cdots \vert g_1]\gamma$ is thus
\[
\sum_{i=0}^n
(-1)^{(n-i)} [g_n \vert \cdots \vert g_{i+1} \vert \gamma_{i+1} \vert g_i \vert \cdots \vert g_1].
\]
Dualizing and passing to $\mathsf{A}$-coefficients gives for $\lambda \in C^{n+1}(\mathcal{G})$ the expression
\begin{equation}
\label{eq:willertonTrans}
\uptau(\lambda)([g_n \vert \cdots \vert g_1]\gamma)
=
\prod_{i=0}^n \lambda ([g_n \vert \cdots \vert g_{i+1} \vert \gamma_{i+1} \vert g_i \vert \cdots \vert g_1])^{(-1)^{n-i}},
\end{equation}
which is precisely the result of Willerton \cite[Theorem 3]{willerton2008}. Willerton's derivation of equation \eqref{eq:willertonTrans} is rather different from that presented here, relying on a particular homotopy equivalence $\vert \Lambda \mathcal{G} \vert \sim \Map(S^1, \vert \mathcal{G} \vert)$, the so-called Parmesan map \cite[Theorem 2]{willerton2008}. The Parmesan map is not equivariant for the reflection action on the circle, and so is not well-suited for the purposes of this paper.

\subsection{Twisted pushforwards for groupoids}
\label{sec:pushGrpd}

We begin by generalizing the untwisted pushforward map \eqref{eq:untwistPush} so as to include a twist on the codomain.

\begin{Lem}
\label{lem:twistedPush}
Let $\mathcal{G}$ be an essentially finite groupoid and $\kappa: \mathcal{G} \rightarrow B \Aut(\mathsf{A})$ a functor. Then the abelian group homomorphism
\[
\pr_{\mathcal{G}!}: 
C^{\bullet + \pr_{\mathcal{G}}^*\kappa}(B \mathbb{Z} \times \mathcal{G}) \rightarrow C^{\bullet -1 + \kappa}(\mathcal{G}),
\qquad
\lambda \mapsto \lambda \circ \mathsf{ez}_{[1]}
\]
anti-commutes with the differentials.
\end{Lem}

\begin{proof}
This follows from the equality $\partial \circ \mathsf{ez}_{[1]} = - \mathsf{ez}_{[1]} \circ \partial$.
\end{proof}

Let now $\hat{\mathcal{G}}$ be an essentially finite $\mathbb{Z}_2$-graded groupoid. Consider the strict $\mathbb{Z}_2$-action on $B \mathbb{Z}$ which negates morphisms. Then $\pr_{\mathcal{G}} : B \mathbb{Z} \times \mathcal{G} \rightarrow \mathcal{G}$ is strictly equivariant for the diagonal $\mathbb{Z}_2$-action on $B \mathbb{Z} \times \mathcal{G}$ and so descends to a functor
\[
\widetilde{\pr}_{\mathcal{G}}: B \mathbb{Z} \times_{\mathbb{Z}_2} \mathcal{G} \rightarrow \hat{\mathcal{G}}.
\]
The category $B \mathbb{Z} \times_{\mathbb{Z}_2} \mathcal{G} := (B \mathbb{Z} \times \mathcal{G}) \git \mathbb{Z}_2$ is the naive quotient, which is well-defined because the $\mathbb{Z}_2$-action is free on objects. Write $\pi_{B \mathbb{Z} \times_{\mathbb{Z}_2} \mathcal{G}} : B \mathbb{Z} \times \mathcal{G} \rightarrow B \mathbb{Z} \times_{\mathbb{Z}_2} \mathcal{G}$ for the canonical double cover. The objective of this section is to construct a pushforward
\[
\widetilde{\pr}_{\mathcal{G}!} : C^{\bullet + \pi_{B \mathbb{Z} \times_{\mathbb{Z}_2} \mathcal{G}} + \pi_{B \mathbb{Z} \times_{\mathbb{Z}_2} \mathcal{G}}^* \kappa}(B \mathbb{Z} \times_{\mathbb{Z}_2} \mathcal{G})
\rightarrow
C^{\bullet-1 + \kappa}(\hat{\mathcal{G}}).
\]
In view of our later applications, we consider only the case in which $\kappa$ is trivial or $\kappa = \pi_{\hat{\mathcal{G}}}$. Note that, in the latter case, $\pi_{B \mathbb{Z} \times_{\mathbb{Z}_2} \mathcal{G}} + \pi_{B \mathbb{Z} \times_{\mathbb{Z}_2} \mathcal{G}}^* \kappa$ is the trivial twist.

We begin with some notation. Given morphisms $\omega_1, \dots, \omega_n$ in $\hat{\mathcal{G}}$, set
\[
\Delta_{\omega_n, \dots, \omega_1}
=
\begin{cases}
1 & \mbox{if } \pi(\omega_n) = \cdots = \pi(\omega_1)=-1,\\
0& \mbox{otherwise.}
\end{cases}
\]
By convention, $\Delta_{\varnothing} =1$. For each $\epsilon \in \mathbb{Z}_2$ and  $i \geq 1$, let
\[
s^{\epsilon}_i =[\epsilon \vert -\epsilon \vert \cdots \vert (-1)^{i+1} \epsilon] \in C_i(B \mathbb{Z}).
\]
Define a $\mathbb{Z}$-linear map
\[
f_n: C_n(\mathcal{G}) \rightarrow \bigoplus_{k=1}^{n+1}  C_k(B \mathbb{Z}) \otimes_{\mathbb{Z}} C_{n+1-k}(\mathcal{G})
\]
by
\begin{multline*}
f_n([\omega_n \vert \cdots \vert \omega_1]_{\epsilon_1})
=
\epsilon_{n+1} s_1^{\epsilon_{n+1}} \otimes [\omega_n \vert \cdots \vert \omega_1]_{\epsilon_1} +
\\
\sum_{i=0}^{n-1} (-1)^{i} \epsilon_{n+1-i} \Delta_{\omega_n, \dots, \omega_{n-i}} s^{\epsilon_{n+1}}_{i+2} \otimes [\omega_{n-1-i} \vert \cdots \vert \omega_1]_{\epsilon_1}.
\end{multline*}
For notational simplicity, we often write $s_i^j$ for $s_i^{\epsilon_j}$ and $\Delta_{n, \dots, 1}$ for $\Delta_{\omega_n, \dots, \omega_1}$.

We regard $C_{\bullet}(B \mathbb{Z}) \otimes_{\mathbb{Z}} C_{\bullet}(\mathcal{G})$ as a $\mathbb{Z}[\mathbb{Z}_2]$-module with its standard tensor product differential. An element of $C_{\bullet}(B \mathbb{Z}) \otimes_{\mathbb{Z}} C_{\bullet}(\mathcal{G})$ is called degenerate if it can be written in the form $\sum_i a_i \otimes b_i$ where, for each $i$, at least one of $a_i$ or $b_i$ is a degenerate chain.

\begin{Prop}
\label{prop:invarianceProperties}
Let $c \in C_{\bullet}(\mathcal{G})$. The following equalities hold:
\begin{enumerate}[label=(\roman*)]
\item $f(\zeta \cdot c) = -\zeta \cdot f(c)$.

\item $f(\partial c) = -\partial f(c) + \textnormal{(degenerate elements)}$.
\end{enumerate}
\end{Prop}

\begin{proof}
The first statement follows from a direct calculation using equation \eqref{eq:groupActionChain}.

Consider then the second statement. To begin, note that
\[
\partial s_i^{\epsilon} = s_{i-1}^{-\epsilon} + (-1)^i s_{i-1}^{\epsilon}  + \textnormal{(degenerate chains)}.
\]
Since we are not concerned with degenerate elements, we henceforth omit them from all equalities. Using this, we find that the terms of $\partial f_n([\omega_n \vert \cdots \vert \omega_1])$ and $-f_{n-1}(\partial [\omega_n \vert \cdots \vert \omega_1])$ which involve $s_1^{\pm}$ are
\[
- \epsilon_{n+1} s_1^{n+1} \otimes \partial [\omega_n \vert \cdots \vert \omega_1] +
\epsilon_{n+1} \Delta_{n}
\Big(
s^{-(n+1)}_{1} + s^{n+1}_1
\Big) \otimes [\omega_{n-1} \vert \cdots \vert \omega_1]
\]
and
\begin{multline*}
-\epsilon_n s_1^n \otimes [\omega_{n-1} \vert \cdots \vert \omega_1]
- \epsilon_{n+1} \sum_{j=1}^{n-1} (-1)^{n-j} s_1^{n+1} \otimes [\omega_n \vert \cdots \vert \omega_{j+1} \omega_j \vert \cdots \vert \omega_1]
\\
- \epsilon_{n+1} (-1)^{n} s_1^{n+1} \otimes [\omega_n \vert \cdots \vert \omega_2],
\end{multline*}
respectively. The term $[\omega_n \vert \cdots \vert \omega_2]$ appears with coefficient $(-1)^{n+1} \epsilon_{n+1} s_1^{n+1}$ in both of these expressions while $[\omega_{n-1} \vert \cdots \vert \omega_1]$ appears with coefficients
\[
- \epsilon_{n+1} s_1^{n+1} +
 \epsilon_{n+1} \Delta_{n}
\Big(
s^{-(n+1)}_{1} + s^{n+1}_1
\Big)
\qquad
\mbox{and}
\qquad
- \epsilon_n s_1^n.
\]
When $\Delta_n=0$, these are plainly equal. When $\Delta_n=1$, the first becomes
\[
\epsilon_n s_1^{-n} -
\epsilon_{n}
\Big(
s^{n}_{1} + s^{-n}_1
\Big)
=
- \epsilon_n s_1^{n}.
\]

Fix now $i \geq 1$. The terms of $\partial f_n([\omega_n \vert \cdots \vert \omega_1])$ and $-f_{n-1}(\partial [\omega_n \vert \cdots \vert \omega_1])$ which involve $s_{i+1}^{\pm}$ are
\begin{align}
\begin{split}
\label{eq:firstTerm}
(-1)^{i} \epsilon_{n+1-i} \Delta_{n, \dots, n-i}
\Big(
s^{-(n+1)}_{i+1} + & (-1)^{i} s^{n+1}_{i+1}  \Big) \otimes [\omega_{n-1-i} \vert \cdots \vert \omega_1]
\\
& +
\epsilon_{n+2-i} \Delta_{n, \dots, n+1-i}
s^{n+1}_{i+1} \otimes \partial [\omega_{n-i} \vert \cdots \vert \omega_1]
\end{split}
\end{align}
and
\begin{align}
\begin{split}
\label{eq:secondTerm}
- & (-1)^{i-1} \epsilon_{n+1-i} \Delta_{n-1, \dots, n-i} s_{i+1}^n \otimes [\omega_{n-1-i} \vert \cdots \vert \omega_1]
\\
& -(-1)^{n-1}
\sum_{j < n - i} (-1)^{j +i} \epsilon_{n-i} \Delta_{n, \dots, n+1-i} s_{i+1}^{n+1} \otimes [\omega_{n-i} \vert \cdots \vert \omega_{j+1} \omega_j \vert \cdots \vert \omega_1]
\\
& -
(-1)^{n-1+i} \epsilon_{j_*+2}\Delta_{n, \dots, (j_*+1)j_*} s^{n+1}_{i+1} \otimes [\omega_{j_*-1} \vert \cdots \vert \omega_1]
\\
& -
\sum_{j > n-i} (-1)^{n-1+j +i} \epsilon_{n+1-i} \Delta_{n, \dots, (j+1) j, \dots, n-i} s_{i+1}^{n+1} \otimes [\omega_{n-i-1} \vert \cdots \vert \omega_1],
\end{split}
\end{align}
respectively, where $j_*=n-i$. The coefficients of $[\omega_{n-i} \vert \cdots \vert \omega_{k+1} \omega_k \vert \cdots \vert \omega_1]$ in \eqref{eq:firstTerm} and \eqref{eq:secondTerm}
are both $(-1)^{n-i-k}\epsilon_{n+2-i} \Delta_{n, \dots, n+1-i} s^{n+1}_{i+1}$ while the coefficients of $[\omega_{n-1-i} \vert \cdots \vert \omega_1]$ are
\begin{equation}
\label{eq:firstTermDiff}
(-1)^{i} \epsilon_{n+1-i} \Delta_{n, \dots, n-i}
\Big(
s^{-(n+1)}_{i+1} + (-1)^{i} s^{n+1}_{i+1} \Big)+ \epsilon_{n+2-i} \Delta_{n, \dots, n+1-i}
s^{n+1}_{i+1}
\end{equation}
and
\begin{align}
\begin{split}
\label{eq:secondTermDiff}
-(-1)^{i-1} \epsilon_{n+1-i} \Delta_{n-1, \dots, n-i} & s_{i+1}^n - 
(-1)^{n-1+i + j_*} \epsilon_{j_*+2} \Delta_{n, \dots, (j_*+1)j_*} s^{n+1}_{i+1}
\\
& -
\sum_{j > n-i} (-1)^{n-1+j+i} \epsilon_{n+1-i} \Delta_{n, \dots, (j+1) j, \dots, n-i} s_{i+1}^{n+1}.
\end{split}
\end{align}
The sum \eqref{eq:firstTermDiff} is non-zero in exactly two cases:
\begin{itemize}
\item $\Delta_{n, \dots, n-i} =1$, in which case \eqref{eq:firstTermDiff} is
\[
%&&
%t_i^{(n)} \epsilon_{n+1-i} 
%\Big(
%s^{-(n+1)}_{i+1} + (-1)^{i+2} s^{n+1}_{i+1} \Big)+ (-1)^{i+1} t_{i-1}^{(n)}\epsilon_{n+2-i}
%s^{n+1}_{i+1}
%\\
%&=&
%t_i^{(n)} \epsilon_{n+1-i} 
%\Big(
%s^{n}_{i+1} + (-1)^{i} s^{-n}_{i+1} \Big)+ (-1)^{i+1} t_{i-1}^{(n)}\epsilon_{n+2-i}
%s^{-n}_{i+1} \\
%&=&
(-1)^i \epsilon_{n+1-i} s^{n}_{i+1} + (-1)^{i}  \epsilon_{n+1-i} ((-1)^i + (-1)^{i-1}  ) s^{-n}_{i+1}
=
(-1)^i \epsilon_{n+1-i} s^{n}_{i+1}
\]
which is equal to \eqref{eq:secondTermDiff}.

\item $\Delta_{n, \dots, n+1-i} =1$ and $\pi(\omega_{n-i})=1$, in which case  \eqref{eq:firstTermDiff} is $\epsilon_{n+2-i} 
s^{n+1}_{i+1}$, which is equal to \eqref{eq:secondTermDiff} (corresponding to the term with $j_* = n-i$).
\end{itemize}

It remains to consider the case in which \eqref{eq:firstTermDiff} is zero. It suffices to assume that exactly one of $\omega_n, \dots, \omega_{n-i}$ has degree $+1$; otherwise \eqref{eq:secondTermDiff} is zero. We can also assume that $\pi(\omega_{n-i}) = -1$, the case $\pi(\omega_{n-i}) =1$ having been treated above. We need to show that \eqref{eq:secondTermDiff} is zero. If $\pi(\omega_n) = 1$, then \eqref{eq:secondTermDiff} is equal to (take $j=n-1$)
\[
-(-1)^{i-1} \epsilon_{n+1-i} s_{i+1}^n -(-1)^{n-1+n-1+i}
\epsilon_{n+1-i} s_{i+1}^{n+1}
=0.
\]
In all other cases, \eqref{eq:secondTermDiff} is equal to
\[
-\sum_{j > n-i} (-1)^{n-1+j+i} \epsilon_{n+1-i} \Delta_{n, \dots, (j+1) j, \dots, n-i} s_{i+1}^{n+1} \otimes [\omega_{n-i-1} \vert \cdots \vert \omega_1].
\]
This sum vanishes, as its two non-zero terms have consecutive $j$ indices.
\end{proof}

Let $\mathsf{ez}_f : C_{\bullet}(\mathcal{G}) \rightarrow C_{\bullet+1}(B \mathbb{Z} \times \mathcal{G})$ be the composition $\mathsf{EZ} \circ f$.

\begin{Prop}
\label{prop:pushforwardDefined}
Let $\kappa : \hat{\mathcal{G}} \rightarrow B \mathbb{Z}_2$ be either the trivial functor or the $\mathbb{Z}_2$-grading $\pi_{\hat{\mathcal{G}}}$. Then the map
\[
\Hom_{\mathbb{Z}}(C_{\bullet}(B \mathbb{Z} \times \mathcal{G}), \mathsf{A}) \rightarrow
\Hom_{\mathbb{Z}}(C_{\bullet-1}(\mathcal{G}), \mathsf{A}),
\qquad
\hat{\phi} \mapsto \hat{\phi} \circ \mathsf{ez}_f
\]
defines an abelian group homomorphism
\[
\widetilde{\pr}_{\mathcal{G}!} : C^{\bullet + \pi_{B \mathbb{Z} \times_{\mathbb{Z}_2} \mathcal{G}} + \pi_{B \mathbb{Z} \times_{\mathbb{Z}_2} \mathcal{G}}^* \kappa}(B \mathbb{Z} \times_{\mathbb{Z}_2} \mathcal{G})
\rightarrow
C^{\bullet-1 + \kappa}(\hat{\mathcal{G}}).
\]
which anti-commutes with the differentials.
\end{Prop}

\begin{proof}
We consider the case in which $\kappa$ is trivial; the other case is completely analogous. Let $\hat{\phi} \in \Hom^{\textnormal{n}}_{\mathbb{Z}[\mathbb{Z}_2]}(C_{\bullet}(B \mathbb{Z} \times \mathcal{G}), \mathsf{A}_-)$ and $c \in C_{\bullet-1}(\mathcal{G})$. We compute
\begin{multline*}
\widetilde{\pr}_{\mathcal{G}!} (\hat{\phi})(\zeta \cdot c)
=
\hat{\phi} (\mathsf{EZ}(f( \zeta \cdot c)))
=
\hat{\phi} (\mathsf{EZ}(- \zeta \cdot f(c)))
=
\hat{\phi} (- \zeta \cdot \mathsf{EZ}(f(c))) \\
=
\hat{\phi} (\mathsf{EZ}(f(c)))
=
\widetilde{\pr}_{\mathcal{G}!} (\hat{\phi})(c).
\end{multline*}
The second, third and fourth equalities follow from Proposition \ref{prop:invarianceProperties}(i), the naturality (and hence $\mathbb{Z}_2$-equivariance) of $\mathsf{EZ}$ and the $\mathbb{Z}[\mathbb{Z}_2]$-linearity of $\hat{\phi}$, respectively. Lemma \ref{lem:topVsAlg} therefore implies that we obtain a map $C^{\bullet + \pi_{B \mathbb{Z} \times_{\mathbb{Z}_2} \mathcal{G}}}(B \mathbb{Z} \times_{\mathbb{Z}_2} \mathcal{G}) \rightarrow C^{\bullet-1}(\hat{\mathcal{G}})$.

To see that $\widetilde{\pr}_{\mathcal{G}!}$ anti-commutes with the differentials, we compute
\begin{multline*}
(d\widetilde{\pr}_{\mathcal{G}!}(\hat{\phi}))(c)
%=
%\widetilde{\pi}_{2!}(\hat{\phi})(\partial c)
=
\hat{\phi}(\mathsf{EZ}(f(\partial c)))
=
\hat{\phi}(\mathsf{EZ}(-\partial f(c)))
=
\hat{\phi}(-\partial \mathsf{EZ}(f(c))) \\
=
(d\hat{\phi})(\mathsf{EZ}(f(c)))^{-1}
=
\widetilde{\pr}_{\mathcal{G}!}(d \hat{\phi})(c)^{-1}.
\end{multline*}
The second and third equalities follow from Proposition \ref{prop:invarianceProperties}(ii), the normalization of $\hat{\phi}$ and that $\mathsf{EZ}$ is a chain map which sends degenerate elements to degenerate chains.
\end{proof}

\begin{Rem}
The map $[1] \otimes - : C_{\bullet}(\mathcal{G}) \rightarrow C_1(B \mathbb{Z}) \otimes_{\mathbb{Z}} C_{\bullet}(\mathcal{G})$, used to define $\pr_{\mathcal{G}!}$ in Lemma \ref{lem:twistedPush}, satisfies part (ii) of Proposition \ref{prop:invarianceProperties} (without working modulo degenerate elements), but not part (i). Indeed, $[-1]$ and $-[1]$ are homologous but not equal.
\end{Rem}

The map $\widetilde{\pr}_{\mathcal{G}!}$ has the expected functorial properties of a pushforward.

\begin{Prop}
Let $F: \hat{\mathcal{G}} \rightarrow \hat{\mathcal{H}}$ be a functor of $\mathbb{Z}_2$-graded groupoids with induced morphism of double covers $F: \mathcal{G} \rightarrow \mathcal{H}$. Then there is a commutative diagram
\[
\begin{tikzpicture}[baseline= (a).base]
\node[scale=1] (a) at (0,0){
\begin{tikzcd}[column sep={15.0em,between origins}, row sep={4.0em,between origins}]
C^{\bullet + \pi_{B \mathbb{Z} \times_{\mathbb{Z}_2} \mathcal{H}} + \pi_{B \mathbb{Z} \times_{\mathbb{Z}_2} \mathcal{H}}^* \kappa}(B \mathbb{Z} \times_{\mathbb{Z}_2} \mathcal{H}) \arrow{d}[left]{(\id_{B \mathbb{Z}} \times F)^*} \arrow{r}[above]{\widetilde{\pr}_{\mathcal{H}!}} & C^{\bullet-1 + \kappa}(\hat{\mathcal{H}}) \arrow{d}[right]{F^*}\\
C^{\bullet + \pi_{B \mathbb{Z} \times_{\mathbb{Z}_2} \mathcal{G}} + \pi_{B \mathbb{Z} \times_{\mathbb{Z}_2} \mathcal{G}}^* F^*\kappa}(B \mathbb{Z} \times_{\mathbb{Z}_2} \mathcal{G}) \arrow{r}[below]{\widetilde{\pr}_{\mathcal{G}!}} & C^{\bullet-1 + F^* \kappa}(\hat{\mathcal{G}}).
\end{tikzcd}
};
\end{tikzpicture}
\]
\end{Prop}

\begin{proof}
Explicitly, the functor $F: \mathcal{G} \rightarrow \mathcal{H}$ is given on a diagram by
\[
\Big[
(x_1, \epsilon_1) \xrightarrow[]{\omega} (x_2, \epsilon_2)
\Big]
\mapsto
(F(x_1), \epsilon_1) \xrightarrow[]{F(\omega)} (F(x_2), \epsilon_2).
\]
Direct inspection then shows that $(\id_{C_{\bullet}(B \mathbb{Z})} \otimes F_*) \circ f_{\mathcal{G}}
=
f_{\mathcal{H}} \circ F_*$. Now use the naturality of $\mathsf{EZ}$ and Proposition \ref{prop:pushforwardDefined}.
%\[
%\begin{tikzpicture}[baseline= (a).base]
%\node[scale=1] (a) at (0,0){
%\begin{tikzcd}[column sep={14.0em,between origins}, row sep={4.0em,between origins}]
%C_n(\mathcal{G}) \arrow{d}[left]{F_*} \arrow{r}[above]{f_{\mathcal{G},n}} & \bigoplus_{k=1}^{n+1}  C_k(B \mathbb{Z}) \otimes_{\mathbb{Z}} C_{n+1-k}(\mathcal{G}) \arrow{d}[right]{\id \otimes F_*}\\
%C_n(\mathcal{H}) \arrow{r}[below]{f_{\mathcal{H},n}} & \bigoplus_{k=1}^{n+1}  C_k(B \mathbb{Z}) \otimes_{\mathbb{Z}} C_{n+1-k}(\mathcal{H}) .
%\end{tikzcd}
%};
%\end{tikzpicture}
%\]
\end{proof}

\subsection{Twisted loop transgression}
\label{sec:twistedTransGrpdFin}

We define and explicitly compute variants of loop transgression maps, producing possibly twisted cochains on the (unoriented) quotient loop groupoid $\Lambda_{\pi}^{(\refl)} \hat{\mathcal{G}}$ from possibly twisted cochains on $\hat{\mathcal{G}}$.

Let $\pi_{\hat{\mathcal{G}}}: \hat{\mathcal{G}} \rightarrow B \mathbb{Z}_2$ be an essentially finite $\mathbb{Z}_2$-graded groupoid. Consider $B \mathbb{Z}$ with its trivial $\mathbb{Z}_2$-action. Then \eqref{diag:oriLoopDiagGrpd} is a diagram of strictly equivariant functors of groupoids with strict $\mathbb{Z}_2$-actions. It follows that there is a strictly commutative diagram
\begin{equation}
\label{diag:quotCorr}
\begin{tikzpicture}[baseline= (a).base]
\node[scale=1.0] (a) at (0,0){
\begin{tikzcd}[column sep={10.0em,between origins}, row sep={2.0em,between origins}]
{} & B \mathbb{Z} \times \Lambda \mathcal{G} \arrow{ld}[above]{\ev} \arrow{rd}[above]{\pr_{\Lambda \mathcal{G}}} \arrow{dd}[left]{\pi_{B \mathbb{Z} \times_{\mathbb{Z}_2} \Lambda \mathcal{G}}}& {} \\
\mathcal{G} \arrow{dd}[left]{\pi_{\hat{\mathcal{G}}}} & {} & \Lambda \mathcal{G} \arrow{dd}[right]{\pi_{\Lambda_{\pi}} \hat{\mathcal{G}}} \\ 
{} & B \mathbb{Z} \times_{\mathbb{Z}_2} \Lambda \mathcal{G} \arrow{ld}[below]{\widetilde{\ev}} \arrow{rd}[below]{\widetilde{\pr}_{\Lambda \mathcal{G}}} & {} \\
\hat{\mathcal{G}} & {} & \Lambda_{\pi} \hat{\mathcal{G}}
\end{tikzcd}
};
\end{tikzpicture}
\end{equation}
whose squares are Cartesian. There is a natural equivalence $B \mathbb{Z} \times_{\mathbb{Z}_2} \Lambda \mathcal{G} \simeq B \mathbb{Z} \times \Lambda_{\pi} \hat{\mathcal{G}}$ under which $\widetilde{\pr}_{\Lambda \mathcal{G}}$ is identified with $\pr_{\Lambda_{\pi} \hat{\mathcal{G}}}$ and $\pi_{B \mathbb{Z} \times_{\mathbb{Z}_2} \Lambda \mathcal{G}}$ with $\id_{B\mathbb{Z}} \times \pi_{\Lambda_{\pi} \hat{\mathcal{G}}}$.

Define the twisted loop transgression map as the composition
%\[
%\uptau_{\pi}: C^{\bullet + \pi_{\hat{\mathcal{G}}}}(\hat{\mathcal{G}}) \xrightarrow[]{\widetilde{\pr}_{\Lambda \mathcal{G}!} \circ \widetilde{\ev}^*} C^{\bullet-1 + \pi_{\Lambda_{\pi} \hat{\mathcal{G}}}}(\Lambda_{\pi} \hat{\mathcal{G}}).
%\]
\begin{multline*}
\uptau_{\pi}:
C^{\bullet + \pi_{\hat{\mathcal{G}}}}(\hat{\mathcal{G}})
\xrightarrow[]{\widetilde{\ev}^*}
C^{\bullet +\widetilde{\ev}^*\pi_{\hat{\mathcal{G}}}}(B \mathbb{Z} \times_{\mathbb{Z}_2} \Lambda \mathcal{G})
\simeq
C^{\bullet +\widetilde{\pr}_{\Lambda \mathcal{G}}^*\pi_{\Lambda_{\pi} \hat{\mathcal{G}}}}(B \mathbb{Z} \times_{\mathbb{Z}_2} \Lambda \mathcal{G})
\\
\xrightarrow[]{\widetilde{\pr}_{\Lambda \mathcal{G}!}}
C^{\bullet-1 + \pi_{\Lambda_{\pi} \hat{\mathcal{G}}}}(\Lambda_{\pi} \hat{\mathcal{G}}).
\end{multline*}
The middle isomorphism is constructed using the Cartesian squares of diagram \eqref{diag:quotCorr}. The final map $\widetilde{\pr}_{\Lambda \mathcal{G}!} = (\id_{B\mathbb{Z}} \times \pi_{\Lambda_{\pi} \hat{\mathcal{G}}})_!$ is that of Lemma \ref{lem:twistedPush}. The map $\uptau_{\pi}$ anti-commutes with the differentials.

\begin{Thm}
\label{thm:quotTrans}
Let $\hat{\lambda} \in C^{n+1 + \pi_{\hat{\mathcal{G}}}}(\hat{\mathcal{G}})$ and $[\omega_n \vert \cdots \vert \omega_1]\gamma \in C_n(\Lambda_{\pi} \hat{\mathcal{G}})$. There is an equality
\[
\uptau_{\pi}(\hat{\lambda})([\omega_n \vert \cdots \vert \omega_1]\gamma)
= 
\prod_{i=0}^{n} \hat{\lambda}([\omega_n \vert \cdots \vert \omega_{i+1} \vert \gamma_{i+1} \vert \omega_i \vert \cdots \vert \omega_1])^{(-1)^{n-i}}.
\]
\end{Thm}

\begin{proof}
Consider the commutative diagram
\[
\begin{tikzpicture}[baseline= (a).base]
\node[scale=0.85] (a) at (0,0){
\begin{tikzcd}[column sep={15.0em,between origins}, row sep={2.0em,between origins}]
{} & \Hom^{\textnormal{n}}_{\mathbb{Z}[\mathbb{Z}_2]}(C_{\bullet}(B \mathbb{Z} \times \Lambda \mathcal{G}), \mathsf{A}_-) \arrow{rd}[above]{\pr_{\Lambda \mathcal{G}!}} \arrow{dd}[left]{\Psi_-} & {} \\
\Hom^{\textnormal{n}}_{\mathbb{Z}[\mathbb{Z}_2]}(C_{\bullet}(\mathcal{G}), \mathsf{A}_-) \arrow{dd}[left]{\Psi_-} \arrow{ur}[above]{\ev^*} & {} & \Hom^{\textnormal{n}}_{\mathbb{Z}[\mathbb{Z}_2]}(C_{\bullet-1}(\Lambda \mathcal{G}), \mathsf{A}_-) \arrow{dd}[right]{\Psi_-}
\\
{} & C^{\bullet + \pi_{B \mathbb{Z} \times_{\mathbb{Z}_2} \Lambda \mathcal{G}}}(B \mathbb{Z} \times_{\mathbb{Z}_2} \Lambda \mathcal{G}) \arrow{rd}[below]{\widetilde{\pr}_{\Lambda \mathcal{G}!}} & {} \\
C^{\bullet + \pi_{\hat{\mathcal{G}}}}(\hat{\mathcal{G}}) \arrow{ur}[below]{\widetilde{\ev}^*} & {} & C^{\bullet -1 + \pi_{\Lambda_{\pi} \hat{\mathcal{G}}}}(\Lambda_{\pi} \hat{\mathcal{G}}).
\end{tikzcd}
};
\end{tikzpicture}
\]
The vertical maps are chain isomorphisms by Lemma \ref{lem:topVsAlg}. Setting $\epsilon_1 = +1$, we find
\begin{eqnarray*}
\uptau_{\pi}(\hat{\lambda})([\omega_n \vert \cdots \vert \omega_1]\gamma)
&=&
(\Psi_- \circ \pr_{\Lambda \mathcal{G}!} \circ \ev^* \circ \Phi_-)(\hat{\lambda})([\omega_n \vert \cdots \vert \omega_1]\gamma)\\
&=&
\Phi_-(\hat{\lambda})(\ev_* \mathsf{EZ}([1] \otimes [\omega_n \vert \cdots \vert \omega_1](\gamma,\epsilon_{n+1}))) \\
&=&
\prod_{i=0}^{n} \hat{\lambda}([\omega_n \vert \cdots \vert \omega_{i+1} \vert \gamma_{i+1} \vert \omega_i \vert \cdots \vert \omega_1])^{(-1)^{n-i}}.
\end{eqnarray*}
The last equality follows from calculations similar to Section \ref{sec:transGrpd} and the equality $\pi(\omega_{\leq n}) = \epsilon_{n+1}$, which ensures that the sign introduced by $\Phi_-$ cancels with $\epsilon_{n+1}$.
\end{proof}

Suppose now that $\mathbb{Z}_2$ acts by negation on $B \mathbb{Z}$. Then $\zeta \in \mathbb{Z}_2$ acts on morphisms in $B \mathbb{Z} \times \Lambda \mathcal{G}$ by
\[
\zeta \cdot \left[
((x_1,\epsilon_1),\gamma_1) \xrightarrow[]{(n,\omega)} ((x_2,\epsilon_2),\gamma_2)
\right]
=
((x_1,-\epsilon_1),\gamma_1^{-1}) \xrightarrow[]{(-n,\omega)} ((x_2,-\epsilon_2),\gamma_2^{-1}).
\]
Again, both functors $\ev : B \mathbb{Z} \times \Lambda \mathcal{G} \rightarrow \mathcal{G}$ and $\pr_{\Lambda \mathcal{G}}: B \mathbb{Z} \times \Lambda \mathcal{G} \rightarrow \Lambda \mathcal{G}$ are strictly equivariant and we obtain a strictly commutative diagram of Cartesian squares similar to \eqref{diag:quotCorr}, but with $\Lambda_{\pi} \hat{\mathcal{G}}$ replaced by $\Lambda_{\pi}^{\refl} \hat{\mathcal{G}}$. Passing to (twisted) cochains gives the commutative diagram
\begin{equation}
\label{diag:twistedTransCover}
\begin{tikzpicture}[baseline= (a).base]
\node[scale=0.85] (a) at (0,0){
\begin{tikzcd}[column sep={14.0em,between origins}, row sep={2.0em,between origins}]
{} & \Hom^{\textnormal{n}}_{\mathbb{Z}[\mathbb{Z}_2]}(C_{\bullet}(B \mathbb{Z} \times \Lambda \mathcal{G}), \mathsf{A}_-) \arrow{rd}[above]{\widetilde{\pr}_{\Lambda \mathcal{G}!}} \arrow{dd}[left]{\Psi_-} & {} \\
\Hom^{\textnormal{n}}_{\mathbb{Z}[\mathbb{Z}_2]}(C_{\bullet}(\mathcal{G}), \mathsf{A}_-) \arrow{dd}[left]{\Psi_-} \arrow{ur}[above]{\ev^*} & {} & \Hom^{\textnormal{n}}_{\mathbb{Z}[\mathbb{Z}_2]}(C_{\bullet-1}(\Lambda \mathcal{G}), \mathsf{A}) \arrow{dd}[right]{\Psi}
\\
{} & C^{\bullet + \pi_{B \mathbb{Z} \times_{\mathbb{Z}_2} \Lambda \mathcal{G}}}(B \mathbb{Z} \times_{\mathbb{Z}_2} \Lambda \mathcal{G}) \arrow{rd}[below]{\widetilde{\pr}_{\Lambda \mathcal{G}!}} & {} \\
C^{\bullet + \pi_{\hat{\mathcal{G}}}}(\hat{\mathcal{G}}) \arrow{ur}[below]{\widetilde{\ev}^*} & {} & C^{\bullet -1}(\Lambda_{\pi}^{\refl} \hat{\mathcal{G}}).
\end{tikzcd}
};
\end{tikzpicture}
\end{equation}
By Lemma \ref{lem:topVsAlg}, the vertical arrows are isomorphisms. Using Proposition \ref{prop:pushforwardDefined}, we define the reflection twisted loop transgression map
\[
\uptau_{\pi}^{\refl}: C^{\bullet + \pi_{\mathcal{G}}}(\hat{\mathcal{G}}) \xrightarrow[]{\widetilde{\ev}^*}
C^{\bullet + \pi_{B \mathbb{Z} \times_{\mathbb{Z}_2} \Lambda \mathcal{G}}}(B \mathbb{Z} \times_{\mathbb{Z}_2} \Lambda \mathcal{G})
\xrightarrow[]{\widetilde{\pr}_{\Lambda \mathcal{G}!}}
C^{\bullet -1}(\Lambda_{\pi}^{\refl} \hat{\mathcal{G}})
\]
which anti-commutes with the differentials. The map $\widetilde{\pr}_{\Lambda \mathcal{G}!}$ is that of Proposition \ref{prop:pushforwardDefined} with trivial twist $\kappa$.

To compute $\uptau_{\pi}^{\refl}$, we introduce some notation. For $1 \leq i \leq n+1$, let $\mathfrak{S}_{i,n+1} \subset \mathfrak{S}_{n+1}$ be the subset of $i$-shuffles. Given $\mathfrak{s}\in \mathfrak{S}_{i,n+1}$, denote by $\mathfrak{s} \cdot [\omega_n \vert \cdots \vert \omega_1]\gamma$ the $(n+1)$-simplex of $\hat{\mathcal{G}}$ whose $\mathfrak{s}(j)$\textsuperscript{th} entry, $1 \leq j \leq i$, is $\gamma^{(-1)^{j+1}\pi(\omega_{\leq n})}_{\mathfrak{s}(j) - j}$ and whose remaining entries are $\omega_{n-i+1}, \dots, \omega_1$, with $\omega_{k+1}$ appearing after $\omega_k$. In symbols,
\[
\mathfrak{s} \cdot [\omega_n \vert \cdots \vert \omega_1]\gamma
=
[\cdots \vert \gamma^{- \pi(\omega_{\leq n})}_{\mathfrak{s}(2) -2}\vert \omega_{\mathfrak{s}(2)-2} \vert \cdots \vert \omega_{\mathfrak{s}(1)} \vert \gamma^{\pi(\omega_{\leq n})}_{\mathfrak{s}(1)-1} \vert \omega_{\mathfrak{s}(1)-1} \vert \cdots \vert \omega_1].
\]
Given $\hat{\lambda} \in C^{n+1+\pi_{\hat{\mathcal{G}}}}(\hat{\mathcal{G}})$, put
\[
\sh_i(\hat{\lambda})([\omega_n \vert \cdots \vert \omega_1] \gamma)
:=
\prod_{\mathfrak{s} \in \mathfrak{S}_{i,n+1}} \hat{\lambda}(\mathfrak{s} \cdot [\omega_n \vert \cdots \vert \omega_1]\gamma)^{\textnormal{sgn}(\mathfrak{s})}.
\]
The map $\uptau_{\pi}^{\refl}$ can now be described as follows.

\begin{Thm}
\label{thm:oriTwistTrans}
Let $\hat{\lambda} \in C^{n+1 + \pi_{\hat{\mathcal{G}}}}(\hat{\mathcal{G}})$ and $[\omega_n \vert \cdots \vert \omega_1]\gamma \in C_n(\Lambda_{\pi}^{\refl} \hat{\mathcal{G}})$. There is an equality
%\begin{multline*}
%\uptau_{\pi}^{\refl} (\hat{\lambda})
%=
%\prod_{i=0}^{n-1} \left( \sh_{n+1-i}(\hat{\lambda})([\omega_n \vert \cdots \vert \omega_1] \gamma) \right)^{\Delta_{n, \dots, n-i}}
%\\
%\times \prod_{i=0}^{n-1} \hat{\lambda}([\omega_{n-1} \vert \cdots \vert \omega_{i+1} \vert \omega_{\leq i} \gamma^{\pi(\omega_{\leq n})} \omega_{\leq i}^{-1} \vert \omega_i \vert \cdots \vert \omega_1])^{(-1)^{n-i-1}}.
%\end{multline*}
%or
\begin{equation}
\label{eq:oriTwistTransForm}
\uptau_{\pi}^{\refl} (\hat{\lambda})([\omega_n \vert \cdots \vert \omega_1]\gamma)
=
\prod_{j=0}^{n} \big( \sh_{n+1-j}(\hat{\lambda})([\omega_n \vert \cdots \vert \omega_1] \gamma) \big)^{(-1)^{n+j}\Delta_{n, \dots, n-j}}.
\end{equation}
\end{Thm}

\begin{proof}
Setting $\epsilon_1=+1$, we have
\begin{eqnarray*}
\uptau_{\pi}^{\refl}(\hat{\lambda})([\omega_n \vert \cdots \vert \omega_1]\gamma)
&=&
(\Psi \circ \widetilde{\pr}_{\Lambda \mathcal{G}!} \circ \ev^* \circ \Phi_-)(\hat{\lambda})([\omega_n \vert \cdots \vert \omega_1]\gamma) \\
&=&
\Phi_-(\hat{\lambda})(\ev_* \mathsf{EZ}(f_n([\omega_n \vert \cdots \vert \omega_1](\gamma,\epsilon_1))))
\end{eqnarray*}
where
\begin{multline*}
f_n ([\omega_n \vert \cdots \vert \omega_1 ](\gamma, \epsilon_1))
=
\epsilon_{n+1} s_1^{n+1} \otimes [\omega_n \vert \cdots \vert \omega_1](\gamma,\epsilon_1) +
\\
\sum_{i=0}^{n-1} (-1)^{i} \epsilon_{n+1-i} \Delta_{n, \dots, n-i} s^{n+1}_{i+2} \otimes [\omega_{n-1-i} \vert \cdots \vert \omega_1](\gamma,\epsilon_1).
\end{multline*}
The definition of $\Phi_-$ shows that the first term of the right hand side contributes to $\uptau_{\pi}^{\refl} (\hat{\lambda})([\omega_n \vert \cdots \vert \omega_1]\gamma)$ with an overall sign of $\epsilon_{n+1}^2 = 1$, yielding the $j=n$ factor of the product \eqref{eq:oriTwistTransForm} while, if at all, the $i$\textsuperscript{th} term of the sum contributes with an overall sign of
\[
(-1)^{i} \epsilon_{n+1-i} \epsilon_{n-i} =(-1)^{i} \pi(\omega_{n-i}) =(-1)^{i+1}.
\]
This gives the $j=n-1-i$ factor of the product \eqref{eq:oriTwistTransForm}. In each of these statements we have used that, by construction, the map $\mathsf{sh}$ realizes the composition $\ev_* \circ \mathsf{EZ}$.
\end{proof}

For example, when $\hat{\alpha} \in C^{1+\pi_{\hat{\mathcal{G}}}}(\hat{\mathcal{G}})$ Theorem \ref{thm:oriTwistTrans} gives $\uptau^{\refl}_{\pi}(\hat{\alpha})([ \, ]\gamma) = \hat{\alpha}([\gamma])$. If instead $\hat{\theta} \in C^{2+\pi_{\hat{\mathcal{G}}}}(\hat{\mathcal{G}})$ and $\hat{\eta} \in C^{3 + \pi_{\hat{\mathcal{G}}}}(\hat{\mathcal{G}})$, then
\[
\uptau^{\refl}_{\pi}(\hat{\theta})([\omega]\gamma) = \hat{\theta}([\gamma^{-1} \vert \gamma])^{- \Delta_{\omega}} \frac{\hat{\theta}([\omega \gamma^{\pi(\omega)} \omega^{-1} \vert \omega])}{\hat{\theta}([\omega \vert \gamma^{\pi(\omega)}])}
\]
while $\uptau_{\pi}^{\refl}(\hat{\eta})([\omega_2 \vert \omega_1 ]\gamma)$ is equal to
\begin{multline*}
\hat{\eta}([\gamma \vert \gamma^{-1} \vert \gamma])^{\Delta_{\omega_2,\omega_1}}  \left(
\frac{\hat{\eta}([\omega_1 \gamma^{-\pi(\omega_1)} \omega_1^{-1} \vert \omega_1 \gamma^{\pi(\omega_1)} \omega_1^{-1} \vert \omega_1]) \hat{\eta}([\omega_1 \vert \gamma^{-\pi(\omega_1)} \vert \gamma^{\pi(\omega_1)}])}{\hat{\eta}([\omega_1 \gamma^{-\pi(\omega_1)} \omega_1^{-1} \vert \omega_1 \vert \gamma^{\pi(\omega_1)}])}
\right)^{- \Delta_{\omega_2}} \\
\times \frac{\hat{\eta}([\omega_2 \vert \omega_1 \vert \gamma^{\pi(\omega_2 \omega_1)}]) \hat{\eta}([\omega_2 \omega_1 \gamma^{\pi(\omega_2 \omega_1)} \omega_1^{-1} \omega_2^{-1} \vert \omega_2 \vert \omega_1])}{\hat{\eta}([\omega_2 \vert \omega_1 \gamma^{\pi(\omega_2 \omega_1)} \omega_1^{-1} \vert \omega_1])}.
\end{multline*}

By direct observation, there is a commutative diagram
\begin{equation}
\label{diag:restrToWillerton}
\begin{tikzpicture}[baseline= (a).base]
\node[scale=1] (a) at (0,0){
\begin{tikzcd}[column sep={12.0em,between origins}, row sep={3.5em,between origins}]
C^{\bullet + \pi_{\hat{\mathcal{G}}}}(\hat{\mathcal{G}}) \arrow{d}[left]{\pi_{\hat{\mathcal{G}}}^*} \arrow{r}[above]{\uptau_{\pi}^{\refl}} & C^{\bullet-1}(\Lambda_{\pi}^{\refl} \hat{\mathcal{G}})\arrow{d}[right]{\pi_{\Lambda_{\pi}\hat{\mathcal{G}}}^*}\\
C^{\bullet}(\mathcal{G}) \arrow{r}[below]{\uptau} & C^{\bullet-1}(\Lambda \mathcal{G}).
\end{tikzcd}
};
\end{tikzpicture}
\end{equation}
This allows us to interpret the terms involving $\Delta_{n, \dots, n-i}$, $1 \leq i \leq n$, in equation \eqref{eq:oriTwistTransForm} as corrections to Willerton's expression \eqref{eq:willertonTrans} which take into account the failure of the map $\widetilde{\pr}_{\Lambda \mathcal{G}}$ to be orientable.

Continuing, define a third twisted loop transgression map
\[
\tilde{\uptau}_{\pi}^{\refl}: C^{\bullet}(\hat{\mathcal{G}}) \xrightarrow[]{\widetilde{\ev}^*}
C^{\bullet}(B \mathbb{Z} \times_{\mathbb{Z}_2} \Lambda \mathcal{G})
\simeq
C^{\bullet + 2\pi_{B \mathbb{Z} \times_{\mathbb{Z}_2} \Lambda \mathcal{G}}}(B \mathbb{Z} \times_{\mathbb{Z}_2} \Lambda \mathcal{G})
\xrightarrow[]{\widetilde{\pr}_{\Lambda \mathcal{G}!}}
C^{\bullet -1 + \pi_{\Lambda_{\pi}^{\refl} \hat{\mathcal{G}}}}(\Lambda_{\pi}^{\refl} \hat{\mathcal{G}}).
\]
The final map is that of Proposition \ref{prop:pushforwardDefined} when the twist $\kappa$ is $\pi_{\Lambda_{\pi}^{\refl} \hat{\mathcal{G}}}: \Lambda \mathcal{G} \rightarrow \Lambda_{\pi}^{\refl} \hat{\mathcal{G}}$. Again, $\tilde{\uptau}_{\pi}^{\refl}$ anti-commutes with the differentials.

\begin{Thm}
\label{thm:oriExtraTwistTrans}
Let $\tilde{\lambda} \in C^{n+1}(\hat{\mathcal{G}})$ and $[\omega_n \vert \cdots \vert \omega_1]\gamma \in C_n(\Lambda_{\pi}^{\refl} \hat{\mathcal{G}})$. There is an equality
\[
\tilde{\uptau}_{\pi}^{\refl} (\tilde{\lambda})([\omega_n \vert \cdots \vert \omega_1]\gamma)
=
\prod_{j=0}^{n} \big( \sh_{n+1-j}(\tilde{\lambda})([\omega_n \vert \cdots \vert \omega_1] \gamma) \big)^{\Delta_{n, \dots, n-j}}.
\]
\end{Thm}

\begin{proof}
Set $\epsilon_1=+1$. The obvious analogue of diagram \eqref{diag:twistedTransCover}  gives
%We have a commutative diagram
%\[
%\begin{tikzpicture}[baseline= (a).base]
%\node[scale=0.9] (a) at (0,0){
%\begin{tikzcd}[column sep={15.0em,between origins}, row sep={2.0em,between origins}]
%{} & \Hom^{\textnormal{n}}_{\mathbb{Z}[\mathbb{Z}_2]}(C_{\bullet}(B \mathbb{Z} \times \Lambda \mathcal{G}), \mathsf{A}) \arrow{rd}[above]{\widetilde{\pr}_{\Lambda \mathcal{G}!}} \arrow{dd}[left]{\Psi} & {} \\
%\Hom^{\textnormal{n}}_{\mathbb{Z}[\mathbb{Z}_2]}(C_{\bullet}(\mathcal{G}), \mathsf{A}) \arrow{dd}[left]{\Psi} \arrow{ur}[above]{\ev^*} & {} & \Hom^{\textnormal{n}}_{\mathbb{Z}[\mathbb{Z}_2]}(C_{\bullet-1}(\Lambda \mathcal{G}), \mathsf{A}_-) \arrow{dd}[right]{\Psi_-}
%\\
%{} & C^{\bullet}(B \mathbb{Z} \times_{\mathbb{Z}_2} \Lambda \mathcal{G}) \arrow{rd}[below]{\widetilde{\pr}_{\Lambda \mathcal{G}!}} & {} \\
%C^{\bullet}(\hat{\mathcal{G}}) \arrow{ur}[below]{\widetilde{\ev}^*} & {} & C^{\bullet -1 + \pi_{\Lambda_{\pi}^{\refl} \hat{\mathcal{G}}}}(\Lambda_{\pi}^{\refl} \hat{\mathcal{G}}).
%\end{tikzcd}
%};
%\end{tikzpicture}
%\]
\begin{eqnarray*}
\tilde{\uptau}_{\pi}^{\refl}(\tilde{\lambda})([\omega_n \vert \cdots \vert \omega_1]\gamma)
&=&
(\Psi_- \circ \widetilde{\pr}_{\Lambda \mathcal{G}!} \circ \ev^* \circ \Phi)(\tilde{\lambda})([\omega_n \vert \cdots \vert \omega_1]\gamma) \\
&=&
\Phi(\tilde{\lambda})(\ev_* \mathsf{EZ}(f_n([\omega_n \vert \cdots \vert \omega_1](\gamma,\epsilon_{n+1}))))
\end{eqnarray*}
where, because of the initial object $(\gamma, \epsilon_{n+1})$, we have
\begin{multline*}
f_n ([\omega_n \vert \cdots \vert \omega_1 ](\gamma, \epsilon_{n+1}))
=
 s_1^{1} \otimes [\omega_n \vert \cdots \vert \omega_1](\gamma,\epsilon_{n+1}) +
\\
 \sum_{i=0}^{n-1} (-1)^{i} \pi(\omega_{\geq n-i+1}) \Delta_{n, \dots, n-i} s^{\pi(\omega_{\geq n-i+1})}_{i+2} \otimes [\omega_{n-1-i} \vert \cdots \vert \omega_1](\gamma,\epsilon_{n+1}).
\end{multline*}
After noting that if $\Delta_{n, \dots, n-i} \neq 0$, then $\pi(\omega_{\geq n-i+1}) = (-1)^{i}$, the remainder of the proof is similar to that of Theorem \ref{thm:oriTwistTrans}.
\end{proof}

For example, when $\tilde{\theta} \in Z^2(\hat{\mathcal{G}})$ Theorem \ref{thm:oriExtraTwistTrans} gives
\[
\tilde{\uptau}^{\refl}_{\pi}(\tilde{\theta})([\omega]\gamma) = \tilde{\theta}([\gamma^{-1} \vert \gamma])^{ \Delta_{\omega}}
\frac{\tilde{\theta}([\omega \gamma^{\pi(\omega)} \omega^{-1} \vert \omega])}{\tilde{\theta}([\omega \vert \gamma^{\pi(\omega)}])}.
\]

\section{Jandl twisted vector bundles}
\label{sec:degTwoCocyc}

We explain the appearance of twisted transgression in the study of Jandl gerbes.

Throughout this section, $\mathcal{G}$ is an essentially finite groupoid and $\hat{\mathcal{G}}$ is an essentially finite groupoid over $B \mathbb{Z}_2$. If $\hat{\mathcal{G}}$ is the object of interest, then $\mathcal{G}$ is its associated double cover. The coefficient group is $\mathsf{A} = \mathsf{U}(1)$. Given a complex vector space $V$, let $\overline{V}$ be its complex conjugate. Set ${^{+1}}V := V$ and ${^{-1}}V: = \overline{V}$, with similar notation ${^{\epsilon}}z$ for $z \in \mathbb{C}$. A map $\varphi: V \rightarrow W$ between complex vector spaces is called $\epsilon$-linear if $\varphi: {^{\epsilon}}V \rightarrow W$ is $\mathbb{C}$-linear.

\subsection{Real functions and Real line bundles}
\label{sec:0and1cocycles}

A $0$-cocycle $\beta \in Z^0(\mathcal{G})$ is a locally constant $\mathsf{U}(1)$-valued function on (the objects of) $\mathcal{G}$. The integral of $\beta$ is
\[
\int_{\mathcal{G}} \beta := \sum_{x \in \mathcal{G}} \frac{\beta(x)}{\vert x \rightarrow \vert} = \sum_{x \in \pi_0(\mathcal{G})} \frac{\beta(x)}{\vert \Aut_{\mathcal{G}}(x) \vert},
\]
where $\vert x \rightarrow \vert$ is the number of morphisms in $\mathcal{G}$ with source $x$. The equality follows from the closedness of $\beta$.

Similarly, $\hat{\beta} \in Z^{0 + \pi_{\hat{\mathcal{G}}}}(\hat{\mathcal{G}})$ is a $\mathsf{U}(1)$-valued function on $\hat{\mathcal{G}}$ which satisfies $\hat{\beta}(x_2)=\hat{\beta}(x_1)^{\pi(\omega)}$ for each morphism $\omega: x_1 \rightarrow x_2$.

As explained in \cite[\S 2.2]{willerton2008}, a $1$-cocycle $\alpha \in Z^1(\mathcal{G})$ defines a trivialized flat complex line bundle $\alpha_{\mathbb{C}}$ over $ \mathcal{G}$. This is the data of trivialized complex lines $L_x$, $x \in \mathcal{G}$, and linear (multiplication) maps
\[
\alpha(x_1 \xrightarrow[]{g} x_2): L_{x_1} \rightarrow L_{x_2}, \qquad g \in \Mor(\mathcal{G})
\]
which satisfy the obvious associativity constraints.\footnote{Note that $\alpha_{\mathbb{C}}$ is also the associated complex line bundle of a $\mathsf{U}(1)$-bundle with connection on $\mathcal{G}$ determined by $\alpha$. Similar comments apply below.} Flat sections of $\alpha_{\mathbb{C}}$, that is, collections of complex numbers $s_x \in L_x$, $x \in \mathcal{G}$, satisfying $\alpha(x_1 \xrightarrow[]{g} x_2) s_{x_1} = s_{x_2}$, $g \in \Mor(\mathcal{G})$, form a complex vector space $\Gamma_{\mathcal{G}}(\alpha_{\mathbb{C}})$. Given $s_1, s_2 \in \Gamma_{\mathcal{G}}(\alpha_{\mathbb{C}})$, the fibrewise product $\overline{s}_1 s_2$ is in $Z^0(\mathcal{G})$. Define an inner product on $\Gamma_{\mathcal{G}}(\alpha_{\mathbb{C}})$ by
\[
\langle s_1, s_2 \rangle = \int_{\mathcal{G}} \overline{s}_1 s_2.
\]

Similarly, $\hat{\alpha} \in Z^{1+\pi_{\hat{\mathcal{G}}}}(\hat{\mathcal{G}})$ defines a trivialized flat Real line bundle $\hat{\alpha}_{\mathbb{C}}$ over $\hat{\mathcal{G}}$. This is the data of trivialized complex lines $L_x$, $x \in \hat{\mathcal{G}}$, and linear maps
\[
\hat{\alpha}(x_1 \xrightarrow[]{\omega} x_2): {^{\pi(\omega)}}L_{x_1} \rightarrow L_{x_2}, \qquad \omega \in \Mor(\hat{\mathcal{G}})
\]
which satisfy the associativity condition
\[
\hat{\alpha}(x_1 \xrightarrow[]{\omega_2 \omega_1} x_3) =
\hat{\alpha}(x_2 \xrightarrow[]{\omega_2} x_3) \cdot {^{\pi(\omega_2)}}
\hat{\alpha}(x_1 \xrightarrow[]{\omega_1} x_2) .
\]
A flat section of $\hat{\alpha}_{\mathbb{C}}$ is a collection of complex numbers $s_x \in L_x$, $x \in \hat{\mathcal{G}}$, satisfying
\[
\hat{\alpha}(x_1 \xrightarrow[]{\omega} x_2) ({^{\pi(\omega)}}s_{x_1} ) = s_{x_2}, 
\qquad
\omega \in \Mor(\hat{\mathcal{G}}).
\]
Setting $\langle s_1, s_2 \rangle = \int_{\hat{\mathcal{G}}} \overline{s}_1 s_2$, flat sections of $\hat{\alpha}_{\mathbb{C}}$ form a real inner product space $\Gamma_{\hat{\mathcal{G}}}(\hat{\alpha}_{\mathbb{C}})$.

\begin{Prop}
\label{prop:dimFlatSect}
For each $\hat{\alpha} \in Z^{1+\pi_{\hat{\mathcal{G}}}}(\hat{\mathcal{G}})$, there is an equality
\[
\frac{1}{2} \dim_{\mathbb{R}} \Gamma_{\hat{\mathcal{G}}}(\hat{\alpha}_{\mathbb{C}}) =  \int_{\Lambda_{\pi}^{\refl} \hat{\mathcal{G}}} \uptau^{\refl}_{\pi}(\hat{\alpha}).
\]
\end{Prop}

\begin{proof}
As both sides of the equality are additive with respect to disjoint union and equivalence of groupoids over $B \mathbb{Z}_2$, it suffices to consider the model cases of Proposition \ref{prop:grpdDecomp}. When $\hat{\mathcal{G}} = B \hat{\mathsf{G}}$, a section $s \in \mathbb{C} \setminus 0$ of $\hat{\alpha}_{\mathbb{C}}$ is flat if and only if $\hat{\alpha}([g]) s = s$ for $g \in \mathsf{G}$ and $\hat{\alpha}([\omega]) \overline{s} = s$ for $\omega \in \hat{\mathsf{G}} \setminus \mathsf{G}$. The first condition implies $\hat{\alpha}_{\vert \mathsf{G}} = 1$. By the cocycle condition, $\hat{\alpha}_{\vert \hat{\mathsf{G}} \setminus \mathsf{G}}$ is constant. The second condition is
\[
\mbox{Arg}(s) \equiv \frac{1}{2} \mbox{Arg}(\hat{\alpha}_{\vert \hat{\mathsf{G}} \setminus \mathsf{G}}) \mod \pi \mathbb{Z}.
\]
It follows that $\Gamma_{B \hat{\mathsf{G}}}(\hat{\alpha}_{\mathbb{C}})$ is $\{0\}$ unless $\hat{\alpha}_{\vert \mathsf{G}} = 1$, in which case $\Gamma_{B \hat{\mathsf{G}}}(\hat{\alpha}_{\mathbb{C}}) \simeq \mathbb{R}$. On the other hand, Theorem \ref{thm:oriTwistTrans} gives
\[
\int_{\Lambda_{\pi}^{\refl} B \hat{\mathsf{G}}} \uptau^{\refl}_{\pi}(\hat{\alpha}) = \sum_{\gamma \in \mathsf{G}} \frac{\hat{\alpha}([\gamma])}{2 \vert \mathsf{G} \vert}.
\]
As in \cite[Theorem 6]{willerton2008}, the sum is zero unless $\hat{\alpha}_{\vert \mathsf{G}} = 1$, in which case it is $\frac{1}{2}$.

When $\hat{\mathcal{G}} = \hat{\mathcal{G}}_{\{x, \overline{x}\}}$, a flat section of $\hat{\alpha}_{\mathbb{C}}$ consists of $s_x, s_{\overline{x}} \in \mathbb{C}$. The equality $\hat{\alpha}(x \xrightarrow[]{\omega} \overline{x})\overline{s_x} = s_{\overline{x}}$ for $\pi(\omega)=-1$ implies that, if the section is non-zero, both $s_x$ and $s_{\overline{x}}$ are non-zero. Hence, $\hat{\alpha}$ is the identity on degree $1$ morphisms and constant on degree $-1$ morphisms and $\Gamma_{\hat{\mathcal{G}}}(\hat{\alpha}_{\mathbb{C}}) \simeq \mathbb{C}$. On the other hand,
\[
\int_{\Lambda_{\pi}^{\refl} \hat{\mathcal{G}}} \uptau^{\refl}_{\pi}(\hat{\alpha}) = \sum_{\gamma \in \Aut_{\hat{\mathcal{G}}}^1(x)} \frac{\hat{\alpha}([\gamma])}{2 \vert \Aut_{\hat{\mathcal{G}}}^1(x) \vert} + \sum_{\gamma \in \Aut_{\hat{\mathcal{G}}}^1(\overline{x})} \frac{\hat{\alpha}([\gamma])}{2 \vert \Aut_{\hat{\mathcal{G}}}^1(\overline{x}) \vert}
\]
is zero unless $\hat{\alpha}$ is the identity on degree $1$ morphisms, in which case it is one.
\end{proof}

\begin{Rems}
\begin{enumerate}[label=(\roman*)]
\item For later comparison, observe that Proposition \ref{prop:dimFlatSect} can also be stated as the equality $\frac{1}{2} \dim_{\mathbb{R}} \Gamma_{\hat{\mathcal{G}}}(\hat{\alpha}_{\mathbb{C}}) =  \int_{\Lambda_{\pi} \hat{\mathcal{G}}} \uptau_{\pi}(\hat{\alpha})^{-1}$. The freedom to use $\uptau_{\pi}(\hat{\alpha})^{-1}$ or $\uptau^{\refl}_{\pi}(\hat{\alpha})$ is an artifact of the low cohomological degree of $\hat{\alpha}$.

\item As an alternative proof of Proposition \ref{prop:dimFlatSect}, it can be shown that $\Gamma_{\hat{\mathcal{G}}}(\hat{\alpha}_{\mathbb{C}})$ defines a real structure on $\Gamma_{\mathcal{G}}(\alpha_{\mathbb{C}})$, in that $\Gamma_{\hat{\mathcal{G}}}(\hat{\alpha}_{\mathbb{C}}) \otimes_{\mathbb{R}} \mathbb{C} \simeq \Gamma_{\mathcal{G}}(\alpha_{\mathbb{C}})$. It follows that $\dim_{\mathbb{R}} \Gamma_{\hat{\mathcal{G}}}(\hat{\alpha}_{\mathbb{C}}) = \dim_{\mathbb{C}} \Gamma_{\mathcal{G}}(\alpha_{\mathbb{C}})$, the right hand side of which is computed in \cite[Theorem 6]{willerton2008}. This strategy is used in Proposition \ref{prop:antiLinHomo} below.
\end{enumerate}
\end{Rems}

\subsection{Jandl twisted vector bundles}
\label{sec:vbJandl}

We explain how $Z^{2+\pi_{\hat{\mathcal{G}}}}(\hat{\mathcal{G}})$ can be used to twist complex vector bundles over $\hat{\mathcal{G}}$, following the non-Real \cite[\S 2.3]{willerton2008} and continuous cases \cite{schreiber2007}, \cite{gawedzki2011}, \cite{moutuou2012}, \cite{freed2013b}. Since our groupoids are essentially finite, we can give a simplified treatment.

Let $\hat{\theta} \in Z^{2+\pi_{\hat{\mathcal{G}}}}(\hat{\mathcal{G}})$. Define a groupoid ${^{\hat{\theta}}}\hat{\mathcal{G}}$ by
\[
\Obj({^{\hat{\theta}}}\hat{\mathcal{G}}) =\Obj(\hat{\mathcal{G}}), \qquad
\Hom_{{^{\hat{\theta}}}\hat{\mathcal{G}}}(x,y) = \mathsf{U}(1) \times \Hom_{\hat{\mathcal{G}}}(x,y).
\]
Morphisms in ${^{\hat{\theta}}}\hat{\mathcal{G}}$ are composed according to the rule
\[
(a_2,\omega_2) \circ (a_1, \omega_1) = (\hat{\theta}([\omega_2 \vert \omega_1 ]) a_2 a_1^{\pi(\omega_2)} , \omega_2 \omega_1).
\]
The category ${^{\hat{\theta}}}\hat{\mathcal{G}}$ is a twisted central extension of $\hat{\mathcal{G}}$ by $B\mathsf{U}(1)$, in that there are canonical functors
\[
\Obj(\hat{\mathcal{G}}) \times B \mathsf{U}(1) \overset{i}{\longrightarrow} {^{\hat{\theta}}}\hat{\mathcal{G}} \overset{p}{\longrightarrow} \hat{\mathcal{G}}
\]
with $p$ surjective on objects and full and $i$ an isomorphism onto the subgroupoid of morphisms which map to an identity in $\hat{\mathcal{G}}$. The twisted centrality condition is
\[
\omega \circ i(x_1 \xrightarrow[]{a} x_1)=
i(x_2 \xrightarrow[]{a^{\pi(\omega)}} x_2) \circ \omega,
\]
where $a \in \mathsf{U}(1)$ and $\omega: x_1 \rightarrow x_2$ in $\hat{\mathcal{G}}$. Compare with \cite[\S 1.2]{freed2013b}. Following the geometric case \cite{schreiber2007}, a twisted central extension of $\hat{\mathcal{G}}$ by $B \mathsf{U}(1)$ is called a Jandl gerbe over $\hat{\mathcal{G}}$. By choosing sections of $p$, we verify that any Jandl gerbe over $\hat{\mathcal{G}}$ is equivalent to ${^{\hat{\theta}}}\hat{\mathcal{G}}$ for some $\hat{\theta} \in Z^{2+\pi_{\hat{\mathcal{G}}}}(\hat{\mathcal{G}})$, well-defined up to exact $2$-cocycles. In this way, $H^{2+\pi_{\hat{\mathcal{G}}}}(\hat{\mathcal{G}})$ classifies equivalence classes of Jandl gerbes over $\hat{\mathcal{G}}$.

Consider $\Vect_{\mathbb{C}}$ as the defining $2$-representation of $B \mathsf{U}(1)$. The anti-linear complex conjugation functor on $\Vect_{\mathbb{C}}$ is compatible with the complex conjugation action on $B \mathsf{U}(1)$. We can therefore associate to ${^{\hat{\theta}}}\hat{\mathcal{G}}$ a Real $2$-line bundle $p: \hat{\theta}_{\mathbb{C}} \rightarrow \hat{\mathcal{G}}$. Concretely, let $\mathsf{RVect}_{\mathbb{C}}$ be the category of finite dimensional complex vector spaces and their complex linear or anti-linear maps. There is a natural functor $\mathsf{RVect}_{\mathbb{C}} \rightarrow B\mathbb{Z}_2$ which records the linearity of morphisms. Then $\hat{\theta}_{\mathbb{C}}$ is the category with objects $\Obj(\mathsf{RVect}_{\mathbb{C}}) \times \Obj(\hat{\mathcal{G}})$, morphisms
\begin{multline*}
\Hom_{\hat{\theta}_{\mathbb{C}}}((V_1, x_1),(V_2,x_2)) = \\ \{ (\varphi, \omega) \in \Hom_{\mathsf{RVect}_{\mathbb{C}}}(V_1,V_2) \times \Hom_{\hat{\mathcal{G}}}(x_1,x_2) \mid \varphi \mbox{ is } \pi(\omega)\mbox{-linear}\}
\end{multline*}
and composition law
\[
(\varphi_2,\omega_2) \circ (\varphi_1, \omega_1) = (\hat{\theta}([\omega_2 \vert \omega_1]) \varphi_2 \varphi_1, \omega_2 \omega_1).
\]
The functor $p: \hat{\theta}_{\mathbb{C}} \rightarrow \hat{\mathcal{G}}$ sends an object $(V,x)$ to $x$ and a morphism $(\varphi, \omega)$ to $\omega$.

\begin{Def}
A $\hat{\theta}$-twisted vector bundle over $\hat{\mathcal{G}}$ is a functor $F: {^{\hat{\theta}}}\hat{\mathcal{G}} \rightarrow \mathsf{RVect}_{\mathbb{C}}$ over $B \mathbb{Z}_2$ such that, for each $(a, x) \in  \mathsf{U}(1) \times \hat{\mathcal{G}}$, the map $F(a, \id_x): F(x) \rightarrow F(x)$ is multiplication by $a$.
\end{Def}

Let $\Vect_{\mathbb{C}}^{\hat{\theta}}(\hat{\mathcal{G}})$ be the category of $\hat{\theta}$-twisted vector bundles over $\hat{\mathcal{G}}$ and their $\mathbb{C}$-linear natural transformations. This category is $\mathbb{R}$-linear and additive.

\begin{Lem}
\label{lem:generatorRRep}
An object $F \in \Vect_{\mathbb{C}}^{\hat{\theta}}(\hat{\mathcal{G}})$ is equivalent to each of the following data:
\begin{enumerate}[label=(\roman*)]
\item Complex vector spaces $F(x)$, $x \in \hat{\mathcal{G}}$, together with complex linear maps
\[
F(x_1 \xrightarrow[]{\omega} x_2) : {^{\pi(\omega)}}F(x_1) \rightarrow F(x_2),
\qquad
\omega \in \Mor(\hat{\mathcal{G}})
\]
which satisfy $F(\id_x) = \id_{F(x)}$, $x \in \hat{\mathcal{G}}$, and
\begin{equation}
\label{eq:twistedComposition}
F(\omega_2) F(\omega_1) = \hat{\theta}([\omega_2 \vert \omega_1]) F(\omega_2 \omega_1),
\qquad (\omega_1,\omega_2) \in \Mor^{(2)}(\hat{\mathcal{G}}).
\end{equation}

\item A section $F: \hat{\mathcal{G}} \rightarrow \hat{\theta}_{\mathbb{C}}$ of $p: \hat{\theta}_{\mathbb{C}} \rightarrow \hat{\mathcal{G}}$.
\end{enumerate}
\end{Lem}

\begin{proof}
That the first data is equivalent to a $\hat{\theta}$-twisted vector bundle is a direct verification. The second data is clearly equivalent to the first.
\end{proof}

In view of Lemma \ref{lem:generatorRRep}(ii), we regard $\Vect_{\mathbb{C}}^{\hat{\theta}}(\hat{\mathcal{G}})$ as the category of sections $\Gamma_{\hat{\mathcal{G}}}(\hat{\theta}_{\mathbb{C}})$.

\begin{Ex}
Let $\hat{\mathsf{G}} = \mathsf{G} \times \mathbb{Z}_2$ with $\pi$ the projection. Then $\Vect_{\mathbb{C}}^{\hat{\theta}}(B\hat{\mathsf{G}})$ is equivalent to the category of real (resp. quaternionic) representations of $\mathsf{G}$ when $\hat{\theta}=1$ (resp. $\hat{\theta}([\omega_2 \vert \omega_1]) = e^{\pi i \Delta_{\omega_2,\omega_1}}$). Similarly, $\Vect_{\mathbb{C}}^{\hat{\theta} =1}(B\hat{\mathsf{G}})$ consists of Real representations (with respect to $\hat{\mathsf{G}}$) of $\mathsf{G}$, in the sense of Atiyah--Segal \cite{atiyah1969} and Karoubi \cite{karoubi1970}. In general, $\Vect_{\mathbb{C}}^{\hat{\theta}}(B\hat{\mathsf{G}})$ consists of $\hat{\theta}$-projective Real representations of $\mathsf{G}$. For this reason, we often refer to $\hat{\theta}$-twisted vector bundles as $\hat{\theta}$-twisted representations.
\end{Ex}

To give a module theoretic interpretation of $\Vect_{\mathbb{C}}^{\hat{\theta}}(\hat{\mathcal{G}})$, let $\mathbb{C}^{\hat{\theta}}[\hat{\mathcal{G}}]$ be the complex vector space with basis $\{l_{\omega}\}_{\omega \in\Mor(\hat{\mathcal{G}})}$ and associative $\mathbb{C}$-semi-linear product
\[
(c_2 l_{\omega_2}) \cdot (c_1 l_{\omega_1}) =
\begin{cases}
c_2 (\prescript{\pi(\omega_2)}{}{c}_1) \hat{\theta}([\omega_2 \vert \omega_1]) l_{\omega_2 \omega_1} & \mbox{if } (\omega_1,\omega_2) \in \Mor^{(2)}(\hat{\mathcal{G}}), \\
0& \mbox{otherwise}.
\end{cases}
\]
We call $\mathbb{C}^{\hat{\theta}}[\hat{\mathcal{G}}]$ the $\hat{\theta}$-twisted groupoid algebra of $\hat{\mathcal{G}}$. A Real representation of $\mathbb{C}^{\hat{\theta}}[\hat{\mathcal{G}}]$ is a complex vector space $V$ together with a $\mathbb{C}$-linear map $\rho: \mathbb{C}^{\hat{\theta}}[\hat{\mathcal{G}}] \rightarrow \End_{\mathbb{R}}(V)$ which is an $\mathbb{R}$-algebra homomorphism and is such that $\rho(l_{\omega})$ is $\pi(\omega)$-linear. Real representations and their $\mathbb{C}$-linear $\mathbb{C}^{\hat{\theta}}[\hat{\mathcal{G}}]$-module homomorphisms form an $\mathbb{R}$-linear category $\mathbb{C}^{\hat{\theta}}[\hat{\mathcal{G}}] \ModR$.

\begin{Prop}
\label{prop:modInterp}
There is an $\mathbb{R}$-linear equivalence $\Vect_{\mathbb{C}}^{\hat{\theta}}(\hat{\mathcal{G}}) \simeq \mathbb{C}^{\hat{\theta}}[\hat{\mathcal{G}}] \ModR$.
\end{Prop}

\begin{proof}
This follows from Lemma \ref{lem:generatorRRep}(i).
\end{proof}

\begin{Prop}[{Cf. \cite[Theorem 18]{willerton2008}}]
\label{prop:grpdDecompCat}
Fix an equivalence as in Proposition \ref{prop:grpdDecomp}. Then the restriction along this equivalence defines an equivalence
\[
\Vect_{\mathbb{C}}^{\hat{\theta}}(\hat{\mathcal{G}})
\xrightarrow[]{\sim}
\bigoplus_{x \in \pi_0(\hat{\mathcal{G}})_{-1}} \Vect_{\mathbb{C}}^{\hat{\theta}_{\vert x}}(B \Aut_{\hat{\mathcal{G}}}(x)) \oplus \bigoplus_{x \in \pi_0(\hat{\mathcal{G}})_1} \Vect_{\mathbb{C}}^{\hat{\theta}_{\vert \{x,\overline{x}\}}}( \hat{\mathcal{G}}_{\{x,\overline{x}\}}).
\]
\end{Prop}

\begin{proof}
As in Proposition \ref{prop:dimFlatSect}, it suffices to prove the statement for connected $\hat{\mathcal{G}}$. Suppose that $\pi_0(\hat{\mathcal{G}}) = \pi_0(\hat{\mathcal{G}})_{-1}$ and let $\rho \in \Vect_{\mathbb{C}}^{\hat{\theta}_{\vert x}}(B \Aut_{\hat{\mathcal{G}}}(x))$ with fibre $V_x$. For each $y \in \hat{\mathcal{G}}$, fix a degree $1$ morphism $g_y: y \rightarrow x$ and define a $F_{\rho} \in \Vect_{\mathbb{C}}^{\hat{\theta}}(\hat{\mathcal{G}})$ by $F_{\rho}(y) = V_x$, $y \in \hat{\mathcal{G}}$, and (cf. \cite[\S 2.4.1]{willerton2008})
\begin{equation}
\label{eq:vbExtension}
F_{\rho}(y \xrightarrow[]{\omega} z) = \frac{\hat{\theta}([g_z \vert \omega])}{\hat{\theta}([g_z \omega g_y^{-1} \vert g_y])} \rho(g_z \omega g_y^{-1}).
\end{equation}
This defines a functor $\Vect^{\hat{\theta}_{\vert x}}_{\mathbb{C}}(B \Aut_{\hat{\mathcal{G}}}(x)) \rightarrow \Vect_{\mathbb{C}}^{\hat{\theta}}(\hat{\mathcal{G}})$ which is quasi-inverse to restriction to $\{x\}$. Indeed, let $F \in \Vect_{\mathbb{C}}^{\hat{\theta}}(\hat{\mathcal{G}})$ with $F_{\vert x} = \rho$. Then $\Phi_y = F(g_y)$ are the components of a natural isomorphism $\Phi: F \Rightarrow F_{\rho}$.

If instead $\pi_0(\hat{\mathcal{G}}) = \pi_0(\hat{\mathcal{G}})_1$, let $\rho \in \Vect_{\mathbb{C}}^{\hat{\theta}_{\{x, \overline{x}\}}}(\hat{\mathcal{G}}_{\{x, \overline{x}\}})$ with fibres $V_x$ and $V_{\overline{x}}$. Using notation from the proof of Proposition \ref{prop:grpdDecomp}, define $F \in \Vect_{\mathbb{C}}^{\hat{\theta}}(\hat{\mathcal{G}})$ by $F_{\rho}(y) = V_{\textnormal{target}(g_y)}$, $y \in \hat{\mathcal{G}}$, with $F_{\rho}(y \xrightarrow[]{\omega} z)$ as in equation \eqref{eq:vbExtension}. The resulting functor $\Vect_{\mathbb{C}}^{\hat{\theta}_{\vert \{x,\overline{x}\}}}( \hat{\mathcal{G}}_{\{x,\overline{x}\}}) \rightarrow \Vect_{\mathbb{C}}^{\hat{\theta}}(\hat{\mathcal{G}})$ is quasi-inverse to restriction to $\{x,\overline{x}\}$.
\end{proof}

\begin{Prop}
\label{prop:completeReduc}
The category $\Vect_{\mathbb{C}}^{\hat{\theta}}(\hat{\mathcal{G}})$ is semisimple.
\end{Prop}

\begin{proof}
This can be proved in the two model cases by first employing a variation of Weyl's unitary trick to unitarize a $\hat{\theta}$-twisted bundle and then taking orthogonal complements of subbundles. Since semisimplicity is preserved under equivalences, the general case then follows from Proposition \ref{prop:grpdDecompCat}.
\end{proof}

\begin{Def}
The $\hat{\theta}$-twisted $K$-theory of $\hat{\mathcal{G}}$ is the Grothendieck group $K^{\hat{\theta}}(\hat{\mathcal{G}}) : = K_0(\Vect_{\mathbb{C}}^{\hat{\theta}}(\hat{\mathcal{G}}))$.
\end{Def}

This is a special case of the twisted $K$-theory studied in \cite{freed2013b} and reduces to the $KR$-theory of \cite{atiyah1969}, \cite{karoubi1970} when $\hat{\theta}$ is trivial. By Proposition \ref{prop:completeReduc}, $K^{\hat{\theta}}(\hat{\mathcal{G}})$ is the free abelian group generated by isomorphism classes of simple $\hat{\theta}$-twisted vector bundles.

We end this section with a fixed point interpretation of $\Vect_{\mathbb{C}}^{\hat{\theta}}(\hat{\mathcal{G}})$. For simplicity, take $\hat{\mathcal{G}} = B \hat{\mathsf{G}}$. Each $\varsigma \in \hat{\mathsf{G}} \setminus \mathsf{G}$ determines an anti-linear involution $(Q^{\varsigma}, \Theta^{\varsigma})$ of the $\mathbb{C}$-linear category $\Vect_{\mathbb{C}}^{\theta}(B\mathsf{G})$ of $\theta$-twisted vector bundles over $B \mathsf{G}$. The functor $Q^{\varsigma} : \Vect_{\mathbb{C}}^{\theta}(B\mathsf{G}) \rightarrow  \Vect_{\mathbb{C}}^{\theta}(B\mathsf{G})$ is given on objects by $Q^{\varsigma}(V,\rho) = (\overline{V}, \rho^{\varsigma})$, where $\rho^{\varsigma}(g) = \uptau_{\pi}(\hat{\theta})([\varsigma]g) \cdot \overline{\rho(\varsigma g \varsigma^{-1})}$. That $Q^{\varsigma}(V,\rho)$ is indeed a $\theta$-twisted representation follows from the identity
\begin{equation}
\label{eq:keyIden2Cocycle}
\frac{\hat{\theta}([\omega g_2 \omega^{-1} \vert \omega g_1 \omega^{-1}])}{\hat{\theta}([g_2 \vert g_1])^{\pi(\omega)}}
=
\frac{\uptau_{\pi}(\hat{\theta})([\omega]g_2) \uptau_{\pi}(\hat{\theta})([\omega]g_1)}{\uptau_{\pi}(\hat{\theta})([\omega]g_2 g_1)},
\qquad
g_i \in \mathsf{G}, \; \omega \in \hat{\mathsf{G}}
\end{equation}
which can be proved directly or, after slightly changing conventions, by using \cite[Proposition 8.1]{kim2015}. The natural isomorphism $\Theta^{\varsigma}: 1_{\Vect^{\theta}(\mathsf{G})} \Rightarrow (Q^{\varsigma})^2$ has components $\Theta^{\varsigma}_{\rho} = \hat{\theta}([\varsigma \vert \varsigma])^{-2} \rho(\varsigma^{-2})$ and satisfies $Q^{\varsigma}(\Theta_{\rho}) = \Theta_{Q^{\varsigma}(\rho)}$. Up to equivalence of categories with involution, $(\Vect_{\mathbb{C}}^{\theta}(B\mathsf{G}),Q^{\varsigma}, \Theta^{\varsigma})$ depends only on the pair $(\hat{\mathsf{G}}, \hat{\theta})$.

%Given $\varsigma_1, \varsigma_2 \in \hat{\mathsf{G}} \setminus \mathsf{G}$, the natural transformation $\nu^{\varsigma_1, \varsigma_2}: Q^{\varsigma_1} \Rightarrow Q^{\varsigma_2}$ with components $\nu^{\varsigma_1, \varsigma_2}_{\rho} = \overline{\rho(\varsigma_2 \varsigma_1^{-1})}$ lifts the identity functor to an equivalence of categories with involution. In this way, $(\hat{\mathsf{G}}, \hat{\theta})$ determines a $\mathsf{G}$-torsor of exact anti-linear involutions of $\Vect_{\mathbb{C}}^{\theta}(B\mathsf{G})$.

\begin{Prop}
\label{prop:fixedPointEquiv}
There is an equivalence $\Vect_{\mathbb{C}}^{\theta}(B\mathsf{G})^{h(Q^{\varsigma}, \Theta^{\varsigma})} \simeq \Vect_{\mathbb{C}}^{\hat{\theta}}(B\hat{\mathsf{G}})$ of $\mathbb{R}$-linear categories, where the left hand side denotes the homotopy fixed point category.
\end{Prop}

\begin{proof}
At the level of objects, the equivalence $\Vect_{\mathbb{C}}^{\theta}(B\mathsf{G})^{h(Q^{\varsigma}, \Theta^{\varsigma})} \rightarrow \Vect_{\mathbb{C}}^{\hat{\theta}}(B\mathsf{G})$ assigns to a homotopy fixed point $\psi_{\rho}: (V,\rho) \rightarrow Q^{\varsigma}(V,\rho)$ the $\hat{\theta}$-twisted representation $\rho$ which is equal to $\rho$ as a $\theta$-twisted representation and has
\[
\rho(\omega) = \hat{\theta}([\omega \vert \varsigma^{-1}]) \rho(\omega \varsigma^{-1}) \circ \psi_{\rho}, \qquad \omega \in \hat{\mathsf{G}} \setminus \mathsf{G}.
\]
On morphisms $F^{\varsigma}$ is the identity.
\end{proof}

Using Proposition \ref{prop:fixedPointEquiv}, define the hyperbolic induction functor
\[
\HInd_{\mathsf{G}}^{\hat{\mathsf{G}}}: \Vect_{\mathbb{C}}^{\theta}(B\mathsf{G}) \rightarrow \Vect_{\mathbb{C}}^{\hat{\theta}}(B\hat{\mathsf{G}})
\]
so that it assigns to $(V, \rho) \in \Vect_{\mathbb{C}}^{\theta}(B \mathsf{G})$ the $\hat{\theta}$-twisted representation with underlying $\theta$-twisted representation $V \oplus Q^{\varsigma}(V)$ and on which $\omega \in \hat{\mathsf{G}} \setminus \mathsf{G}$ acts on by
%\[
%\rho_{\HInd(\rho)}(
%\omega) (v_1, v_2) = \Big( \hat{\theta}([\varsigma^{-1} \vert \varsigma]) \hat{\theta}([\omega \vert \varsigma^{-1}]) \rho_V(\omega \varsigma^{-1}) v_2, \hat{\theta}([\varsigma \vert \omega]) \cdot  \overline{\rho_V(\varsigma \omega)} v_1 \Big).
%\]
\[
\rho_{\HInd(\rho)}(
\omega) (v_1, v_2) = \Big( \hat{\theta}([\omega \vert \varsigma^{-1}]) \rho(\omega \varsigma^{-1}) v_2, \hat{\theta}([\varsigma \vert \omega]) \cdot  \overline{\rho(\varsigma \omega)} v_1 \Big).
\]

\subsection{Character theory of \texorpdfstring{$\Vect_{\mathbb{C}}^{\hat{\theta}}(\hat{\mathcal{G}})$}{}}
\label{sec:realChar}

We begin by characterizing the equivariance of characters of $\hat{\theta}$-twisted vector bundles.

\begin{Prop}
\label{prop:twistedChar}
The assignment of an object $F \in \Vect_{\mathbb{C}}^{\hat{\theta}}(\hat{\mathcal{G}})$ to the function
\[
\chi_{F} : \Obj(\Lambda_{\pi}^{\refl}\hat{\mathcal{G}}) \rightarrow \mathbb{C},
\qquad
(x, \gamma)
\mapsto
\tr_{F(x)}F(\gamma)
\]
defines an abelian group homomorphism $\chi: K^{\hat{\theta}}(\hat{\mathcal{G}}) \rightarrow \Gamma_{\Lambda_{\pi}^{\refl} \hat{\mathcal{G}}} (\uptau^{\refl}_{\pi}(\hat{\theta})_{\mathbb{C}})$.
\end{Prop}

\begin{proof}
Let $\omega: x_1 \rightarrow x_2$ be a morphism in $\hat{\mathcal{G}}$. Equation \eqref{eq:twistedComposition} and the twisted $2$-cocycle condition on $\hat{\theta}$ imply
%\[
%\hat{\theta}([\omega \gamma^{\pi(\omega)} \vert \omega^{-1}]) \hat{\theta}([\omega \gamma^{\pi(\omega)} \omega^{-1} \vert \omega]) = \hat{\theta}([\omega^{-1} \vert \omega])^{\pi(\omega)}.
%\]
%Combining the above equalities gives
\begin{equation}
\label{eq:Fdecomp}
F(\omega \gamma^{\pi(\omega)} \omega^{-1}) = \frac{\hat{\theta}([\omega \gamma^{\pi(\omega)} \omega^{-1} \vert \omega])}{\hat{\theta}([\omega \vert \gamma^{\pi(\omega)}])} F(\omega) F( \gamma^{\pi(\omega)}) F(\omega)^{-1}.
\end{equation}
When $\pi(\omega) =-1$ we can use equation \eqref{eq:twistedComposition} to replace $F( \gamma^{-1})$ with $\hat{\theta}([\gamma \vert \gamma^{-1}])F(\gamma)^{-1}$. Doing so and taking the trace of equation \eqref{eq:Fdecomp} gives
\[
\tr_{F(x_2)} F(\omega \gamma^{\pi(\omega)} \omega^{-1}) = \uptau^{\refl}_{\pi}(\hat{\theta}) ([\omega] \gamma) \tr_{F(x_2)} ( F(\omega) F(\gamma)^{\pi(\omega)} F(\omega)^{-1}).
\]
Since $F(\omega)$ is $\pi(\omega)$-linear and $\tr_{F(x_1)} (F(\gamma)^{-1}) = \overline{\tr_{F(x_1)}F(\gamma)}$ (see the proof of \cite[Proposition 10]{willerton2008}), we arrive at the equality
%Since $F(\omega)$ is $\pi(\omega)$-linear, we have
%\[
%\tr_{F(x_2)}( F(\omega) F(\gamma)^{\pi(\omega)} F(\omega)^{-1}) ={^{\pi(\omega)}}\big( \tr_{F(x_1)} F(\gamma)^{\pi(\omega)} \big).
%\]
%Using the equality $\tr_{F(x_1)} (F(\gamma)^{-1}) = \overline{\tr_{F(x_1)}F(\gamma)}$ (see the proof of \cite[Proposition 10]{willerton2008}) we arrive at the equality
\[
\tr_{F(x_2)}F( \omega \gamma^{\pi(\omega)} \omega^{-1}) = \uptau^{\refl}_{\pi}(\hat{\theta}) ([\omega] \gamma) \tr_{F(x_1)}F(\gamma),
\]
that is, $\chi_{F} \in \Gamma_{\Lambda_{\pi}^{\refl} \hat{\mathcal{G}}} (\uptau^{\refl}_{\pi}(\hat{\theta})_{\mathbb{C}})$. Since $\chi_{F_1 \oplus F_2} = \chi_{F_1} + \chi_{F_2}$, this completes the proof.
\end{proof}

We call $\chi_{F}$ the Real character of $F$. The pullback of $\chi_{F}$ along $\Lambda \mathcal{G} \rightarrow \Lambda_{\pi}^{\refl} \hat{\mathcal{G}}$ is the character of the $\theta$-twisted vector bundle which underlying $F$, as defined in \cite[\S 2.3.3]{willerton2008}. Proposition \ref{prop:twistedChar} is a twisted reality condition on $\chi_{F}$. Indeed, when $\hat{\mathcal{G}} = B \hat{\mathsf{G}}$ with $\hat{\mathsf{G}} = \mathsf{G} \times \mathbb{Z}_2$ and $\hat{\theta} =1$, Proposition \ref{prop:twistedChar} reduces to the reality of the character of a representation of $\mathsf{G}$ over $\mathbb{R}$.

Given a trivially twisted vector bundle $F \in \Vect_{\mathbb{C}}^{1}(\hat{\mathcal{G}})$, let $\Gamma_{\hat{\mathcal{G}}}(F)$ be the real vector space of its flat sections. We have the following generalization of Proposition \ref{prop:dimFlatSect}. 

\begin{Prop}[{cf. \cite[Proposition 7]{willerton2008}}]
\label{prop:dimTrivFlatSect}
The equality $\frac{1}{2} \dim_{\mathbb{R}} \Gamma_{\hat{\mathcal{G}}} (F) = \int_{\Lambda_{\pi}^{\refl} \hat{\mathcal{G}}} \chi_{F}$ holds.
\end{Prop}

Consider the $\Hom$ bifunctor
\[
\langle-, - \rangle: \Vect_{\mathbb{C}}^{\hat{\theta}}(\hat{\mathcal{G}})^{\op} \times \Vect_{\mathbb{C}}^{\hat{\theta}}(\hat{\mathcal{G}}) \rightarrow \Vect_{\mathbb{R}}, \qquad (F_1, F_2) \mapsto \Hom_{\Vect_{\mathbb{C}}^{\hat{\theta}}(\hat{\mathcal{G}})}(F_1, \hat{F_2}),
\]
which we regard as a categorical inner product on $\Vect_{\mathbb{C}}^{\hat{\theta}}(\hat{\mathcal{G}})$. There is a canonical real vector space isomorphism
\begin{equation}
\label{eq:innProd}
\langle F_1, \hat{F_2} \rangle \simeq \Gamma_{\hat{\mathcal{G}}} (F_1^{\vee} \otimes F_2),
\end{equation}
the right hand side being the space of flat sections of $F_1^{\vee} \otimes F_2 \in \Vect_{\mathbb{C}}^{1}(\hat{\mathcal{G}})$.

\begin{Lem}
\label{lem:innProdCompat}
Let $F_1, F_2 \in \Vect_{\mathbb{C}}^{\hat{\theta}}(\hat{\mathcal{G}})$. The equality $\dim_{\mathbb{R}} \langle F_1, \hat{F_2} \rangle = \frac{1}{2} \langle \chi_{F_1}, \chi_{F_2} \rangle$ holds.
\end{Lem}

\begin{proof}
The proof is as in \cite[Proposition 10]{willerton2008}. We compute
\[
\dim_{\mathbb{R}} \langle F_1, \hat{F_2} \rangle
\overset{\scriptsize \mbox{Eq. } \eqref{eq:innProd}}{=}
\dim_{\mathbb{R}} \Gamma_{\hat{\mathcal{G}}}(F_1^{\vee} \otimes \hat{F_2}) 
\overset{\scriptsize \mbox{Prop. } \eqref{prop:dimTrivFlatSect}}{=}
\frac{1}{2} \int_{\Lambda_{\pi}^{\refl} \hat{\mathcal{G}}} \chi_{F_1^{\vee} \otimes \hat{F_2}}.
\]
Since $\chi_{F_1^{\vee} \otimes \hat{F_2}} = \overline{\chi_{F_1}} \cdot \chi_{\hat{F_2}}$, the right hand side is equal to $\frac{1}{2} \langle \chi_{F_1}, \chi_{F_2} \rangle$.
\end{proof}

It is proved in \cite[Theorem 11]{willerton2008} that, for any $\theta \in Z^2(\mathcal{G})$, the (ordinary) character map induces an isomorphism
\begin{equation}
\label{eq:twistedKIso}
\chi: K^{\theta}(\mathcal{G}) \otimes_{\mathbb{Z}} \mathbb{C} \xrightarrow[]{\sim} \Gamma_{\Lambda \mathcal{G}}(\uptau(\theta)_{\mathbb{C}}).
\end{equation}
The next result is a Real generalization of this isomorphism. Special cases, in terms of projective representations of $\mathsf{G}$ over $\mathbb{R}$ twisted by $Z^2(B \mathsf{G}; \mathbb{R}^{\times})$, can be found in \cite[Theorem 6]{reynolds1971}, \cite[\S 10.2]{karpilovsky1985}.

\begin{Thm}
\label{thm:charIso}
The character map induces an isomorphism
\[
\chi: K^{\hat{\theta}}(\hat{\mathcal{G}}) \otimes_{\mathbb{Z}} \mathbb{C} \xrightarrow[]{\sim} \Gamma_{\Lambda_{\pi}^{\refl} \hat{\mathcal{G}}} (\uptau^{\refl}_{\pi}(\hat{\theta})_{\mathbb{C}})
\]
of complex inner product spaces.
\end{Thm}

\begin{proof}
Proposition \ref{prop:grpdDecompCat} gives an isomorphism of abelian groups
\begin{equation}
\label{eq:grothGrpDecomp}
K^{\hat{\theta}}(\hat{\mathcal{G}}) \simeq \bigoplus_{x \in \pi_0(\mathcal{G})_{-1}} K^{\hat{\theta}_{\vert x}}(B \Aut_{\hat{\mathcal{G}}}(x)) \oplus \bigoplus_{x \in \pi_0(\mathcal{G})_1} K^{\hat{\theta}_{\vert \{x,\overline{x}\}}}( \hat{\mathcal{G}}_{\{x,\overline{x}\}}).
\end{equation}
Lemma \ref{lem:innProdCompat}, together with Schur's Lemma for $\hat{\theta}$-twisted vector bundles, then implies that $\chi$ is injective.

In view of the isomorphism \eqref{eq:grothGrpDecomp}, it suffices to prove surjectivity of $\chi$ in the two model cases. Suppose that $\hat{\mathcal{G}} = B \hat{\mathsf{G}}$. Let $\textnormal{Irr}^{\theta}(\mathsf{G})$ is the set of isomorphism classes of simple objects of $\Vect_{\mathbb{C}}^{\theta}(B \mathsf{G})$. Use the isomorphism \eqref{eq:twistedKIso} to write $s \in \Gamma_{\Lambda_{\pi}^{\refl} B \hat{\mathsf{G}}}(\uptau_{\pi}^{\refl}(\hat{\theta})_{\mathbb{C}}) \subset \Gamma_{\Lambda B \mathsf{G}}(\uptau(\theta)_{\mathbb{C}})$ as
\[
s = \sum_{V \in \textnormal{Irr}^{\theta}(\mathsf{G})} \langle \chi_V, s \rangle \chi_V.
\]
The additional symmetry conditions on $s$, namely
\[
s(\omega \gamma^{-1} \omega^{-1}) = \uptau^{\refl}_{\pi}(\hat{\theta})([\omega] \gamma) s(\gamma), \qquad \gamma \in \mathsf{G}, \;\; \omega \in \hat{\mathsf{G}} \setminus \mathsf{G},
\]
imply that the function $V \mapsto \langle \chi_V, s \rangle$ descends to the orbit space $\textnormal{Irr}^{\theta}(\mathsf{G}) \slash \langle Q^{\varsigma} \rangle$. It follows that for any fixed $\varsigma \in \hat{\mathsf{G}} \setminus \mathsf{G}$,
\[
s = \sum_{\mathcal{O} \in \textnormal{Irr}^{\theta}(\mathsf{G}) \slash \langle Q^{\varsigma} \rangle} a_{\mathcal{O}} \sum_{V \in \mathcal{O}} \chi_V = \frac{1}{2} \sum_{\mathcal{O} \in \textnormal{Irr}^{\theta}(\mathsf{G}) \slash \langle Q^{\varsigma} \rangle} a_{\mathcal{O}} \sum_{V \in \mathcal{O}} (\chi_V + \chi_{\varsigma \cdot V})
\]
for some $a_{\mathcal{O}} \in \mathbb{C}$. Noting that $\chi_{\HInd_{\mathsf{G}}^{\hat{\mathsf{G}}}(V)} = \chi_V + \chi_{\varsigma \cdot V}$, we see that $\sum_{V \in \mathcal{O}} \chi_V$, and hence $s$, is in the image of $K^{\hat{\theta}}(B \hat{\mathsf{G}}) \otimes_{\mathbb{Z}} \mathbb{C}$.

Suppose instead that $\hat{\mathcal{G}} = \hat{\mathcal{G}}_{\{x, \overline{x}\}}$ and let $s \in \Gamma_{\Lambda_{\pi}^{\refl} \hat{\mathcal{G}}}(\uptau_{\pi}^{\refl}(\hat{\theta})_{\mathbb{C}}) \subset \Gamma_{\Lambda \mathcal{G}}(\uptau(\theta)_{\mathbb{C}})$. Write
\[
s = \sum_{V \in \textnormal{Irr}^{\hat{\theta}_{\vert x}} (\Aut_{\hat{\mathcal{G}}}(x))} \langle \chi_V, s \rangle \chi_V + \sum_{V^{\prime} \in  \textnormal{Irr}^{\hat{\theta}_{\vert \overline{x}}} (\Aut_{\hat{\mathcal{G}}}(\overline{x}))} \langle  \chi_{V^{\prime}} , s \rangle \chi_{V^{\prime}}.
\]
Fix a morphism $\varsigma: x \rightarrow \overline{x}$ of degree $-1$ and use the isomorphism $\Aut_{\hat{\mathcal{G}}}(x) \xrightarrow[]{\sim} \Aut_{\hat{\mathcal{G}}}(\overline{x})$, $\gamma \mapsto \varsigma \gamma \varsigma^{-1}$, to identify $\textnormal{Irr}^{\hat{\theta}_{\vert x}}(\Aut_{\hat{\mathcal{G}}}(x))$ and $\textnormal{Irr}^{\hat{\theta}_{\vert x}}(\Aut_{\hat{\mathcal{G}}}(\overline{x}))$. The additional symmetry condition on $s$ implies that
\[
s = \sum_{V \in \textnormal{Irr}^{\hat{\theta}_{\vert x}} (\Aut_{\hat{\mathcal{G}}}(x))} \langle \chi_{V_x} , s\rangle (\chi_{V_x} + \chi_{V_{\overline{x}}}),
\]
where $V_{\overline{x}} \in \Vect_{\mathbb{C}}^{\hat{\theta}_{\vert \overline{x}}}(B \Aut_{\hat{\mathcal{G}}}(\overline{x}))$ is the pullback of $\overline{V}_x$ along $\varsigma$. Explicitly, a loop $\varsigma \gamma \varsigma^{-1}$ at $\overline{x}$ acts on $V_{\overline{x}} = \overline{V}_x$ by
\[
\frac{ \hat{\theta}([\varsigma \vert \gamma]) \hat{\theta}([\varsigma \gamma \vert \varsigma^{-1}])}{\hat{\theta}([\varsigma \gamma \varsigma^{-1} \vert \varsigma])} \rho(\gamma).
\]
The sum $V_x \oplus V_{\overline{x}}$ becomes a $\hat{\theta}$-twisted vector bundle by taking $\rho(\varsigma): V_x \rightarrow V_{\overline{x}}$ to be the identity map (which is anti-linear) and setting $\rho(\varsigma \gamma) := \hat{\theta}([\varsigma \vert \gamma])^{-1} \rho(\varsigma) \rho(\gamma)$. Moreover, $\chi_{V_x \oplus V_{\overline{x}}} = \chi_{V_x} + \chi_{V_{\overline{x}}}$, whence $s$ is in the image of $K^{\hat{\theta}}(\hat{\mathcal{G}}) \otimes_{\mathbb{Z}} \mathbb{C}$.

The statement about inner products follows from Lemma \ref{lem:innProdCompat}.
\end{proof}

\begin{Cor}
\label{cor:twoCycDim}
The category $\Vect_{\mathbb{C}}^{\hat{\theta}}(\hat{\mathcal{G}})$ has exactly $\int_{\Lambda \Lambda_{\pi}^{\refl} \hat{\mathcal{G}}} \uptau \uptau^{\refl}_{\pi}(\hat{\theta})$ simple objects.
\end{Cor}

\begin{proof}
By Proposition \ref{prop:completeReduc}, the rank of $K^{\hat{\theta}}(\hat{\mathcal{G}})$ equals the number of simple $\hat{\theta}$-twisted vector bundles. The corollary now follows from Theorem \ref{thm:charIso} and the equality $\int_{\Lambda \Lambda_{\pi}^{\refl} \hat{\mathcal{G}}} \uptau \uptau^{\refl}_{\pi}(\hat{\theta}) = \dim_{\mathbb{C}}\Gamma_{\Lambda_{\pi}^{\refl} \hat{\mathcal{G}}} (\uptau^{\refl}_{\pi}(\hat{\theta})_{\mathbb{C}})$, which follows from \cite[Theorem 6]{willerton2008}.
\end{proof}

Specializing Corollary \ref{cor:twoCycDim} to $\hat{\mathcal{G}} = B \hat{\mathsf{G}}$ gives the following result.

\begin{Cor}
\label{cor:RealSchur}
The number of simple $\hat{\theta}$-twisted representations of $\hat{\mathsf{G}}$ is
\begin{equation}
\label{eq:RealSchur}
\int_{\Lambda \Lambda_{\pi}^{\refl} B \hat{\mathsf{G}}} \uptau \uptau^{\refl}_{\pi}(\hat{\theta}) = \frac{1}{2 \vert \mathsf{G} \vert} \sum_{\substack{(\gamma, \omega) \in \mathsf{G} \times \hat{\mathsf{G}} \\ \gamma = \omega \gamma^{\pi(\omega)} \omega^{-1}}} \hat{\theta}([\gamma^{-1} \vert \gamma])^{- \Delta_{\omega}} \frac{\hat{\theta}([\gamma \vert \omega ])}{\hat{\theta}([\omega \vert \gamma^{\pi(\omega)} ])}.
\end{equation}
\end{Cor}

The right hand side of equation \eqref{eq:RealSchur} decomposes into two sums, corresponding to $\pi(\omega)=1$ and $\pi(\omega)=-1$. The former sum is one half the number of simple $\theta$-twisted class functions of $\mathsf{G}$ which, by Schur (see \cite[\S 3.6]{karpilovsky1985}), is one half the number of simple $\theta$-twisted representations. The latter sum is one half the number of Real simple $\theta$-twisted class functions, where Real means that $\chi(\omega \gamma^{\pi(\omega)} \omega^{-1}) = \uptau_{\pi}^{\refl}(\hat{\theta)}([\omega]\gamma) \chi(\gamma)$ for some (and hence any) $\omega \in \hat{\mathsf{G}} \setminus \mathsf{G}$. Corollary \ref{cor:RealSchur} is thus a Real version of Schur's result. When $\hat{\mathsf{G}} = \mathsf{G} \times \mathbb{Z}_2$ and $\hat{\theta}_{\vert \mathsf{G}} = 1$, we recover standard results in the real/quaternionic representation theory of $\mathsf{G}$ \cite[Theorem II.6.3]{brocker1995}.

\subsection{The centre of \texorpdfstring{$\Vect_{\mathbb{C}}^{\hat{\theta}}(\hat{\mathcal{G}})$}{}}
\label{sec:twistedCent}

The centre $Z(\Vect_{\mathbb{C}}^{\theta}(\mathcal{G}))$, that is, the algebra of $\mathbb{C}$-linear endofunctors of $\Vect_{\mathbb{C}}^{\theta}(\mathcal{G})$, is isomorphic to $\Gamma_{\Lambda \mathcal{G}}(\uptau(\theta)^{-1}_{\mathbb{C}})$. In fact, the map
\[
K^{\theta}(\mathcal{G}) \times Z(\Vect_{\mathbb{C}}^{\theta}(\mathcal{G})) \rightarrow \mathbb{C}, \qquad (V, \eta) \mapsto \tr_V \eta_V
\]
defines a perfect pairing between $K^{\theta}(\mathcal{G}) \otimes_{\mathbb{Z}} \mathbb{C}$ and $Z(\Vect_{\mathbb{C}}^{\theta}(\mathcal{G}))$, giving a compatibility between two decategorifications of $\Vect_{\mathbb{C}}^{\theta}(\mathcal{G})$ \cite[\S 2.3.4]{willerton2008}. There is no analogous compatibility in the Real setting. Instead, we will describe $Z(\Vect_{\mathbb{C}}^{\hat{\theta}}(\hat{\mathcal{G}}))$ using $\uptau_{\pi}$.

The image of the $\mathbb{R}$-linear embedding $\Gamma_{\Lambda_{\pi} \mathcal{G}}(\uptau_{\pi}(\hat{\theta})^{-1}_{\mathbb{C}}) \rightarrow \mathbb{C}^{\hat{\theta}}[\hat{\mathcal{G}}]$, $s \mapsto \sum_{\gamma \in \Lambda_{\pi} \hat{\mathcal{G}}} s_{\gamma} l_{\gamma}$, is stable under multiplication of $\mathbb{C}^{\hat{\theta}}[\hat{\mathcal{G}}]$, as follows from equation \eqref{eq:keyIden2Cocycle}, and so gives $\Gamma_{\Lambda_{\pi} \mathcal{G}}(\uptau_{\pi}(\hat{\theta})^{-1}_{\mathbb{C}})$ the structure of an $\mathbb{R}$-algebra.

\begin{Prop}
\label{prop:centreTwist}
The centre of the $\mathbb{R}$-algebra $\mathbb{C}^{\hat{\theta}}[\hat{\mathcal{G}}]$ is isomorphic to $\Gamma_{\Lambda_{\pi}\hat{\mathcal{G}}}(\uptau_{\pi}(\hat{\theta})^{-1}_{\mathbb{C}})$.
\end{Prop}

\begin{proof}
For each morphism $\omega: x_1 \rightarrow x_2$ in $\hat{\mathcal{G}}$ and $c_{x_1} \in \mathbb{C}$, we have equalities
\[
l_{\omega} (c_{x_1} l_{\id_{x_1}}) = {^{\pi(\omega)}}c_{x_1} l_{\omega}, \qquad (c_{x_1} l_{\id_{x_1}}) l_{\omega} = \delta_{x_1,x_2} c_{x_1} l_{\omega}
\]
in $\mathbb{C}^{\hat{\theta}}[\hat{\mathcal{G}}]$. Elements of the centre $Z(\mathbb{C}^{\hat{\theta}}[\hat{\mathcal{G}}])$ are therefore of the form $\sum_{\gamma \in \Lambda_{\pi}\hat{\mathcal{G}}} c_{\gamma} l_{\gamma}$. Requiring this element to commute with $l_{\omega}$ gives
\begin{multline*}
l_{\omega} \sum_{\gamma \in \Lambda_{\pi}\hat{\mathcal{G}}} c_{\gamma} l_{\gamma}
=
\Big( \sum_{\gamma \in \Lambda_{\pi}\hat{\mathcal{G}}} c_{\gamma} l_{\gamma} \Big) l_{\omega}
=
\sum_{\delta \in \Lambda_{\pi}\hat{\mathcal{G}}} \frac{\hat{\theta}([\omega \delta \omega^{-1} \vert \omega])}{\hat{\theta}([\omega \vert \delta])} c_{\omega \delta \omega^{-1}}l_{\omega} l_{\delta} \\
=
l_{\omega} \sum_{\delta \in \Lambda_{\pi}\hat{\mathcal{G}}} {^{\pi(\omega)}}\big(\frac{\hat{\theta}([\omega \delta \omega^{-1} \vert \omega])}{\hat{\theta}([\omega \vert \delta])} c_{\omega \delta \omega^{-1}} \big) l_{\delta}.
\end{multline*}
It follows that
\begin{equation}
\label{eq:conjRelations}
c_{\gamma}
=
{^{\pi(\omega)}}
\Big(
\frac{\hat{\theta}([\omega \gamma \omega^{-1} \vert \omega])}{\hat{\theta}([\omega \vert \gamma])} c_{\omega \gamma \omega^{-1}} 
\Big),
\qquad
\gamma \in \Lambda_{\pi} \hat{\mathcal{G}}.
\end{equation}
Conversely, the equalities \eqref{eq:conjRelations} ensure that $\sum_{\gamma \in \Lambda_{\pi}\hat{\mathcal{G}}} c_{\gamma} l_{\gamma}$ commutes with $l_{\omega}$. The map
\[
Z(\mathbb{C}^{\hat{\theta}}[\hat{\mathcal{G}}]) \rightarrow \Gamma_{\Lambda_{\pi}\hat{\mathcal{G}}}(\uptau_{\pi}(\hat{\theta})^{-1}_{\mathbb{C}}), \qquad \sum_{\gamma \in \Lambda_{\pi}\hat{\mathcal{G}}} c_{\gamma} l_{\gamma} \mapsto (\gamma \mapsto c_{\gamma})
\]
is therefore well-defined and gives the desired isomorphism.
\end{proof}

The $\mathbb{R}$-algebra $Z(\mathbb{C}^{\hat{\theta}}[\hat{\mathcal{G}}])$ is isomorphic to the centre of $\mathbb{C}^{\hat{\theta}}[\hat{\mathcal{G}}] \Mod$, the category of modules over the $\mathbb{R}$-algebra $\mathbb{C}^{\hat{\theta}}[\hat{\mathcal{G}}]$. To relate this to the centre of $\mathbb{C}^{\hat{\theta}}[\hat{\mathcal{G}}] \ModR \simeq \Vect_{\mathbb{C}}^{\hat{\theta}}(\hat{\mathcal{G}})$, let $A$ be a finite dimensional Real algebra, that is, a $\mathbb{Z}_2$-graded complex vector space which has the structure of a unital $\mathbb{R}$-algebra which satisfies
\[
(c_2 a_2) \cdot (c_1 a_1) = c_2 ({^{\pi(a_2)}}c_1) a_2 a_1
\]
for all $c_i \in \mathbb{C}$ and homogeneous $a_i \in A$.

\begin{Lem}
\label{lem:genCent}
The $\mathbb{R}$-algebra $Z(A \ModR)$ is isomorphic to $Z(A)_1$, the degree $1$ subalgebra of $Z(A)$.
\end{Lem}

\begin{proof}
This is a straightforward variation of the proof that the centre of the category of modules over a unital algebra is isomorphic to the centre of the algebra.
\end{proof}

\begin{Thm}
\label{thm:RealCentre}
There is a canonical $\mathbb{R}$-algebra isomorphism
\[
Z(\Vect_{\mathbb{C}}^{\hat{\theta}}(\hat{\mathcal{G}}))
\simeq
\Gamma_{\Lambda_{\pi}\hat{\mathcal{G}}}(\uptau_{\pi}(\hat{\theta})^{-1}_{\mathbb{C}}).
\]
\end{Thm}

\begin{proof}
By Proposition \ref{prop:modInterp}, the categories $\mathbb{C}^{\hat{\theta}}[\hat{\mathcal{G}}] \ModR$ and $\Vect_{\mathbb{C}}^{\hat{\theta}}(\hat{\mathcal{G}})$ are equivalent. 
Proposition \ref{prop:centreTwist} and Lemma \ref{lem:genCent} then give algebra isomorphisms
\[
\Gamma_{\Lambda_{\pi}\hat{\mathcal{G}}}(\uptau_{\pi}(\hat{\theta})^{-1}_{\mathbb{C}}) \simeq Z( \mathbb{C}^{\hat{\theta}}[\hat{\mathcal{G}}]) = Z(\mathbb{C}^{\hat{\theta}}[\hat{\mathcal{G}}])_1 \simeq
Z(\Vect^{\hat{\theta}}(\hat{\mathcal{G}})).
\]
The middle equality follows from the explicit description of $Z(\mathbb{C}^{\hat{\theta}}[\hat{\mathcal{G}}])$.
\end{proof}

\begin{Cor}
\label{cor:dimCentre}
For any finite $\mathbb{Z}_2$-graded group $\hat{\mathsf{G}}$, there is an equality
\[
\dim_{\mathbb{R}} Z(\Vect_{\mathbb{C}}^{\hat{\theta}}(B \hat{\mathsf{G}})) = \frac{1}{\vert \mathsf{G} \vert} \sum_{\substack{(g_1, g_2) \in \mathsf{G}^2  \\ g_1 g_2 = g_2 g_1}} \frac{\hat{\theta}([g_1 \vert g_2])}{\hat{\theta}([g_2 \vert g_1])}.
\]
\end{Cor}

\begin{proof}
This follows by combining Proposition \ref{prop:dimFlatSect} and Theorem \ref{thm:RealCentre}.
\end{proof}

In particular, the dimension of $Z(\Vect_{\mathbb{C}}^{\hat{\theta}}(B \hat{\mathsf{G}}))$ is independent of the lift $(\hat{\mathsf{G}}, \hat{\theta})$ of $(\mathsf{G}, \theta)$. In fact, by \cite[Theorem 6]{willerton2008}, we have
\[
\dim_{\mathbb{R}} Z(\Vect_{\mathbb{C}}^{\hat{\theta}}(B \hat{\mathsf{G}})) =  \dim_{\mathbb{C}} Z(\Vect_{\mathbb{C}}^{\theta}(B \mathsf{G})).
\]
A conceptual explanation of this equality is as follows.

\begin{Prop}
\label{prop:antiLinHomo}
For each $\varsigma \in \hat{\mathsf{G}} \setminus \mathsf{G}$, the pair $(Q^{\varsigma}, \Theta^{\varsigma})$ induces a $\varsigma$-independent anti-linear algebra involution
\[
q: Z(\Vect_{\mathbb{C}}^{\theta}(B \mathsf{G})) \rightarrow Z(\Vect_{\mathbb{C}}^{\theta}(B \mathsf{G}))
\]
whose fixed point set is $Z(\Vect_{\mathbb{C}}^{\hat{\theta}}(B \hat{\mathsf{G}}))$.
\end{Prop}

\begin{proof}
Under the equivalence $\Vect_{\mathbb{C}}^{\theta}(B\mathsf{G}) \simeq \mathbb{C}^{\theta}[B\mathsf{G}] \Mod$, the functor $Q^{\varsigma}$ becomes the anti-linear algebra automorphism
\[
q^{\varsigma}: \mathbb{C}^{\theta}[B \mathsf{G}] \rightarrow \mathbb{C}^{\theta}[B \mathsf{G}],
\qquad
\sum_{g \in \mathsf{G}} c_{g} l_{g} \mapsto \sum_{g \in \mathsf{G}} \uptau_{\pi}(\hat{\theta})([\varsigma] g)^{-1}  \overline{c}_{g} l_{\varsigma g \varsigma^{-1}}.
\]
Closedness of $\uptau_{\pi}(\hat{\theta})$ implies that $q^{\varsigma}$ squares to $\textnormal{Ad}_{l_{\varsigma^2}}$. It follows that $q^{\varsigma}$ restricts to an anti-linear algebra involution $q: \Gamma_{\Lambda B \mathsf{G}}(\uptau(\theta)^{-1}_{\mathbb{C}}) \rightarrow \Gamma_{\Lambda B \mathsf{G}}(\uptau(\theta)^{-1}_{\mathbb{C}})$ which, again by the closedness of $\uptau_{\pi}(\hat{\theta})$, is independent of $\varsigma$. The explicit form of $q^{\varsigma}$ shows that the fixed point set of $q$ is $\Gamma_{\Lambda_{\pi} B \hat{\mathsf{G}}}(\uptau_{\pi}(\hat{\theta})^{-1}_{\mathbb{C}})$. To finish the proof, apply Theorem \ref{thm:RealCentre}.
\end{proof}

\section{Jandl twisted $2$-vector bundles}
\label{sec:degThreeCocyc}

We study representation theoretic aspects of $Z^{3+ \pi_{\hat{\mathcal{G}}}}(\hat{\mathcal{G}})$. For simplicity, we restrict attention to $\mathbb{Z}_2$-graded groupoids of the form $\hat{\mathcal{G}} = B \hat{\mathsf{G}}$.

\subsection{Thickened Drinfeld doubles}
\label{sec:drinDoub}

Let $\mathsf{G}$ be a finite group and $\eta \in Z^3(B \mathsf{G})$. The $\eta$-twisted Drinfeld double $D^{\eta}(\mathsf{G})$ is a quasi-Hopf algebra with explicitly defined product and coproduct \cite{dijkgraaf1992}. The starting point of this section is Willerton's algebra isomorphism between $D^{\eta}(\mathsf{G})$ and the twisted groupoid algebra $\mathbb{C}^{\uptau(\eta)}[\Lambda B \mathsf{G}]$ \cite[\S 3.1]{willerton2008}. This provides a conceptual definition of the algebra $D^{\eta}(\mathsf{G})$ and leads to short proofs of a number of its fundamental properties, such as a description of its character theory \cite{willerton2008}.

Turning to the Real setting, fix a finite $\mathbb{Z}_2$-graded group $\hat{\mathsf{G}}$.

\begin{Def}
Let $\hat{\eta} \in Z^{3 + \pi_{\hat{\mathsf{G}}}}(B \hat{\mathsf{G}})$ and $\tilde{\eta} \in Z^3(B \hat{\mathsf{G}})$ be lifts of $\eta \in Z^3(B \mathsf{G})$.
\begin{enumerate}[label=(\roman*)]
\item Define $\mathbb{R}$-algebras by $\mathbb{D}^{\hat{\eta}}(\hat{\mathsf{G}}) :=\mathbb{C}^{\uptau_{\pi}(\hat{\eta})}[\Lambda_{\pi} B \hat{\mathsf{G}}]$ and $\mathbb{D}^{\tilde{\eta}}(\hat{\mathsf{G}}) := \mathbb{C}^{\tilde{\uptau}_{\pi}^{\refl}(\tilde{\eta})}[\Lambda_{\pi}^{\refl} B \hat{\mathsf{G}}]$.

\item Define a $\mathbb{C}$-algebra by $D^{\hat{\eta}}(\hat{\mathsf{G}}) :=\mathbb{C}^{\uptau^{\refl}_{\pi}(\hat{\eta})}[\Lambda^{\refl}_{\pi} B \hat{\mathsf{G}}]$.
\end{enumerate}
\end{Def}

Compatibility of the oriented and twisted loop transgression maps, as in diagram \eqref{diag:restrToWillerton} for $\uptau_{\pi}^{\refl}$, implies that $D^{\eta}(\mathsf{G})$ embeds into each of $\mathbb{D}^{\hat{\eta}}(\hat{\mathsf{G}})$, $\mathbb{D}^{\tilde{\eta}}(\hat{\mathsf{G}})$ and $D^{\hat{\eta}}(\hat{\mathsf{G}})$ as the complex subalgebra of degree $1$ morphisms. We therefore refer to any of the algebras in the above definition as twisted thickened Drinfeld doubles.

The results of Section \ref{sec:realChar} can be applied to the representation theory of thickened Drinfeld doubles. To begin, we identify their representation groups. Motivated by \cite[\S 3.2]{willerton2008}, we view these results as Real counterparts of the Freed--Hopkins--Teleman theorem \cite[Theorem 1]{freed2011} in our (much simpler) finite setting. A Real analogue for compact, connected and simply connected Lie groups is \cite[Theorem 5.12]{fok2018}.

\begin{Prop}
\label{prop:RealFHT}
There are isomorphisms of abelian groups
\[
K^{\uptau^{\refl}_{\pi}(\hat{\eta})}(\mathsf{G} \git_{\textnormal{R}} \hat{\mathsf{G}}) \simeq K_0(D^{\hat{\eta}}(\hat{\mathsf{G}}) \Mod),
\]
\[
K^{\tilde{\uptau}^{\refl}_{\pi}(\tilde{\eta})} (\mathsf{G} \git_{\textnormal{R}} \hat{\mathsf{G}}) \simeq K_0(\mathbb{D}^{\tilde{\eta}}(\hat{\mathsf{G}}) \ModR)
\]
and
\[
K^{\uptau_{\pi}(\hat{\eta})} (\mathsf{G} \git \hat{\mathsf{G}}) \simeq K_0(\mathbb{D}^{\hat{\eta}}(\hat{\mathsf{G}})\ModR).
\]
\end{Prop}

\begin{proof}
Recall that $\Lambda_{\pi}^{\refl} B \hat{\mathsf{G}} \simeq \mathsf{G} \git_{\textnormal{R}} \hat{\mathsf{G}}$ and $\Lambda_{\pi} B \hat{\mathsf{G}} \simeq \mathsf{G} \git \hat{\mathsf{G}}$. The first isomorphism follows from the equivalence $\Vect_{\mathbb{C}}^{\uptau^{\refl}_{\pi}(\hat{\eta})}(\Lambda_{\pi}^{\refl} B \hat{\mathsf{G}}) \simeq \mathbb{C}^{\uptau^{\refl}_{\pi}(\hat{\eta})}[\Lambda_{\pi}^{\refl} B \hat{\mathsf{G}}] \Mod$ (see \cite[Proposition 8]{willerton2008}) and the second from $\Vect_{\mathbb{C}}^{\tilde{\uptau}^{\refl}_{\pi}(\tilde{\eta})}(\Lambda_{\pi}^{\refl} B \hat{\mathsf{G}}) \simeq \mathbb{C}^{\tilde{\uptau}^{\refl}_{\pi}(\tilde{\eta})}[\Lambda_{\pi}^{\refl} B \hat{\mathsf{G}}] \ModR$ (see Proposition \ref{prop:modInterp}). The third isomorphism is proved in the same way.
\end{proof}

Only the final two isomorphisms of Proposition \ref{prop:RealFHT} involve a form of Real equivariant $K$-theory of $\mathsf{G}$. The first isomorphism still has a Real flavour, however, as it involves $\Lambda_{\pi}^{\refl} B \hat{\mathsf{G}}$. The second isomorphism is the finite analogue of the result of Fok \cite{fok2018}; it would be interesting to combine them to describe the twisted $\hat{\mathsf{G}}$-equivariant $K$-theory of $\mathsf{G}$ for $\hat{\mathsf{G}}$ an arbitrary compact $\mathbb{Z}_2$-graded Lie group.

Writing the appropriate loop groupoid as a disjoint union of standard models (see Proposition \ref{prop:grpdDecomp}) leads to a decomposition of the representation categories of thickened Drinfeld doubles. For example,
\[
D^{\hat{\eta}}(\hat{\mathsf{G}}) \Mod \simeq \bigoplus_{g \in \pi_0(\mathsf{G} \git_{\textnormal{R}} \hat{\mathsf{G}})} \Vect_{\mathbb{C}}^{\uptau^{\refl}_{\pi}(\hat{\eta})_{\vert g}}(B Z_{\hat{\mathsf{G}}}^{\textnormal{R}}(g)),
\]
where $Z_{\hat{\mathsf{G}}}^{\textnormal{R}}(g)=\{ \omega \in \hat{\mathsf{G}} \mid \omega g^{\pi(\omega)} \omega^{-1} = g\}$ is the Real centralizer of $g$. Simple $D^{\hat{\eta}}(\hat{\mathsf{G}})$-modules are therefore labelled by a Real conjugacy class of $\mathsf{G}$ and a simple twisted representation of its Real centralizer. Similarly, Proposition \ref{prop:grpdDecompCat} shows that $\mathbb{D}^{\tilde{\eta}}(\hat{\mathsf{G}}) \ModR$ decomposes as
\[
\bigoplus_{g \in \pi_0(\mathsf{G} \git_{\textnormal{R}} \hat{\mathsf{G}})_{-1}} \Vect_{\mathbb{C}}^{\tilde{\uptau}^{\refl}_{\pi}(\tilde{\eta})_{\vert g}}(B Z_{\hat{\mathsf{G}}}^{\textnormal{R}}(g)) \oplus \bigoplus_{\{g,\overline{g}\} \in \pi_0(\mathsf{G} \git_{\textnormal{R}} \hat{\mathsf{G}})_1} \Vect_{\mathbb{C}}^{\tilde{\uptau}^{\refl}_{\pi}(\tilde{\eta})_{\vert \{g,\overline{g}\}}}( \mathsf{G} \git_{\textnormal{R}} \hat{\mathsf{G}}_{\vert \{g, \overline{g}\}}).
\]
The first sum is over conjugacy classes of $\mathsf{G}$ which are fixed by the involution determined by $\hat{\mathsf{G}}$ and the second is over the $\mathbb{Z}_2$-quotient of its complement. The previous two decompositions, and the quasi-inverse from the proof of Proposition \ref{prop:grpdDecompCat}, show that a collection of twisted (Real) representations of the groupoids appearing on the right hand side determines a representation of the thickened Drinfeld double. This gives Real versions of Dijkgraaf--Pasquier--Roche induction \cite[\S 2.2]{dijkgraaf1992}, \cite[\S 3.3]{willerton2008}.

\subsection{Twisted one-loop characters}
\label{sec:twistOneLoopChar}

The following definition is motivated by the definition of twisted elliptic characters in \cite[\S 3.4]{willerton2008}. See also \cite{hopkins2000}.

\begin{Def}
Let $\hat{\eta} \in Z^{3+\pi_{\hat{\mathsf{G}}}}(B \hat{\mathsf{G}})$. Elements of $\Gamma_{\Lambda \Lambda_{\pi}^{\refl} B \hat{\mathsf{G}}}(\uptau \uptau_{\pi}^{\refl}(\hat{\eta})_{\mathbb{C}})$ are called $\hat{\eta}$-twisted one-loop characters of $\mathsf{G}$.
\end{Def}

To make this definition explicit, let $\mathsf{G}^{(2)} \subset \mathsf{G}^2$ be the set of commuting pairs in $\mathsf{G}$ and let
\[
\hat{\mathsf{G}}^{\langle 2 \rangle} : = \{ (g, \omega) \in \mathsf{G} \times \hat{\mathsf{G}} \mid g \omega = \omega g^{\pi(\omega)}\}
\]
be the set of graded commuting pairs in $\mathsf{G}$. Then a twisted one-loop character is a function $\chi: \hat{\mathsf{G}}^{\langle 2 \rangle} \rightarrow \mathbb{C}$ which satisfies
\[
\chi(\sigma g^{\pi(\sigma)} \sigma^{-1}, \sigma \omega \sigma^{-1}) =  \uptau \uptau_{\pi}^{\refl}(\hat{\eta})([\sigma]g \xrightarrow[]{\omega} g) \chi(g, \omega), \qquad \sigma \in \hat{\mathsf{G}}.
\]

The relevance of this definition to the representation theory of $D^{\hat{\eta}}(\hat{\mathsf{G}})$ is as follows.

\begin{Prop}
\label{prop:charThickDD}
The character map is an isometry
\[
\chi : K_0(D^{\hat{\eta}}(\hat{\mathsf{G}})\Mod) \otimes_{\mathbb{Z}} \mathbb{C} \xrightarrow[]{\sim} \Gamma_{\Lambda \Lambda_{\pi}^{\refl} B \hat{\mathsf{G}}}(\uptau \uptau_{\pi}^{\refl}(\hat{\eta})_{\mathbb{C}}).
\]
\end{Prop}

\begin{proof}
Since $D^{\hat{\eta}}(\hat{\mathsf{G}}) = \mathbb{C}^{\uptau_{\pi}^{\refl}(\hat{\eta})}[\Lambda_{\pi}^{\refl} B \hat{\mathsf{G}}]$, this follows from \cite[Theorem 11]{willerton2008}.
\end{proof}

The term one-loop character is motivated by two dimensional unoriented topological (or conformal) field theory, where the (closed) one-loop sector is made up by the $2$-torus $\mathbb{T}^2$ and the Klein bottle $\mathbb{K}$. There is a canonical decomposition
\[
\Gamma_{\Lambda \Lambda_{\pi}^{\refl} B \hat{\mathsf{G}}}(\uptau \uptau_{\pi}^{\refl}(\hat{\eta})_{\mathbb{C}}) = \Gamma^{\mathbb{T}^2}_{\Lambda \Lambda_{\pi}^{\refl} B \hat{\mathsf{G}}}(\uptau \uptau_{\pi}^{\refl}(\hat{\eta})_{\mathbb{C}}) \oplus \Gamma^{\mathbb{K}}_{\Lambda \Lambda_{\pi}^{\refl} B \hat{\mathsf{G}}}(\uptau \uptau_{\pi}^{\refl}(\hat{\eta})_{\mathbb{C}}),
\]
where the first and second summands consists of suitably $\hat{\mathsf{G}}$-equivariant functions on the sets $\mathsf{G}^{(2)}$ and $(\mathsf{G} \times (\hat{\mathsf{G}} \setminus \mathsf{G})) \cap \hat{\mathsf{G}}^{\langle 2 \rangle}$, respectively. The (graded) commuting conditions which define these sets are precisely the defining relations of $\pi_1(\mathbb{T}^2)$ and $\pi_1(\mathbb{K})$ in their standard presentations. The first summand is a subspace of the $\eta$-twisted elliptic characters $\Gamma_{\Lambda^2 B \mathsf{G}}(\uptau^2 (\eta)_{\mathbb{C}})$ while the second consists of what we call $\hat{\eta}$-twisted Klein characters. As suggested by the appearance of fundamental groups, the summands can be interpreted in terms of certain moduli spaces of $\mathsf{G}$-bundles on $\mathbb{T}^2$ or $\mathbb{K}$. See \cite{willerton2008}, \cite[\S 3.2]{mbyoung2019} for details.

\begin{Cor}
\label{cor:rankThickDD}
The number of simple $D^{\hat{\eta}}(\hat{\mathsf{G}})$-modules is $\int_{\Lambda^2 \Lambda_{\pi}^{\refl} B \hat{\mathsf{G}}} \uptau^2 \uptau_{\pi}^{\refl} (\hat{\eta})$.
\end{Cor}

\begin{proof}
Proposition \ref{prop:charThickDD} implies that
\[
\rank \, K_0(D^{\hat{\eta}}(\hat{\mathsf{G}}) \Mod) =
\dim_{\mathbb{C}} \Gamma_{\Lambda \Lambda_{\pi}^{\refl} B \hat{\mathsf{G}}}(\uptau \uptau_{\pi}^{\refl}(\hat{\eta})_{\mathbb{C}}).
\]
By \cite[Theorem 6]{willerton2008} the right hand side is equal to $\int_{\Lambda^2 \Lambda_{\pi}^{\refl} B \hat{\mathsf{G}}} \uptau^2 \uptau_{\pi}^{\refl} (\hat{\eta})$.
\end{proof}

Straightforward modifications of the previous discussion apply to $\mathbb{D}^{\hat{\eta}}(\hat{\mathsf{G}})$ and $\mathbb{D}^{\tilde{\eta}}(\hat{\mathsf{G}})$. We limit ourselves to describing the character theory of $\mathbb{D}^{\hat{\eta}}(\hat{\mathsf{G}})$, which is rather different from that of $D^{\hat{\eta}}(\hat{\mathsf{G}})$. Theorem \ref{thm:charIso} gives
\[
K_0(\mathbb{D}^{\hat{\eta}}(\hat{\mathsf{G}}) \ModR) \otimes_{\mathbb{Z}} \mathbb{C} \simeq  
\Gamma_{\Lambda_{\pi}^{\refl} \Lambda_{\pi} B \hat{\mathsf{G}}}(\uptau_{\pi}^{\refl} \uptau_{\pi}(\hat{\eta})_{\mathbb{C}}).
\]
The right hand side is the set of functions $\chi: \mathsf{G}^{(2)} \rightarrow \mathbb{C}$ which satisfy
\[
\chi(\sigma g_1 \sigma^{-1} , \sigma g_2^{\pi(\sigma)} \sigma^{-1}) = \uptau_{\pi}^{\refl} \uptau_{\pi}(\hat{\eta})([\sigma]g_1 \xrightarrow[]{g_2} g_1) \cdot \chi(g_1, g_2),
\qquad
\sigma \in \hat{\mathsf{G}}.
\]
Characters of Real $\mathbb{D}^{\hat{\eta}}(\hat{\mathsf{G}})$-representations therefore form a subspace of the space of $\eta$-twisted elliptic characters. In particular, there is no Klein sector.

\subsection{Real pointed fusion categories and their centres}
\label{sec:RealFusion}

In this section, we categorify the $\hat{\theta}$-twisted groupoid algebra $\mathbb{C}^{\hat{\theta}}[\hat{\mathcal{G}}]$ of Section \ref{sec:vbJandl}. We use coefficients $\mathsf{A} = \mathbb{C}^{\times}$. For background on monoidal categories, the reader is referred to \cite{etingof2015}.

To begin, we introduce a Real version of monoidal categories. This can be seen as a modification of the notion of a $\mathbb{Z}_2$-graded extension of a monoidal category which takes into account complex conjugation twists.

\begin{Def}
A Real monoidal category is a $\mathbb{C}$-linear abelian category $\mathcal{C}$ with
\begin{enumerate}[label=(\roman*)]
\item a decomposition $\mathcal{C} = \mathcal{C}^{(1)} \oplus \mathcal{C}^{(-1)}$ into full abelian subcategories,

\item an additive functor $\otimes : \mathcal{C} \times \mathcal{C} \rightarrow \mathcal{C}$ which restricts to $\mathbb{C}$-bilinear functors
\[
%\label{eq:antilinTensor}
\otimes: \mathcal{C}^{(i)} \times {^{i}}\mathcal{C}^{(j)} \rightarrow \mathcal{C}^{(ij)},
\qquad
i,j \in \mathbb{Z}_2,
\]
where ${^{+1}}\mathcal{C}^{(j)} = \mathcal{C}^{(j)}$ and ${^{-1}}\mathcal{C}^{(j)}$ is the complex conjugate category of $\mathcal{C}^{(j)}$,

\item an object $\mathbf{1} \in \mathcal{C}$, together with left and right unitors, and
\item for homogeneous $X_k \in \mathcal{C}$, $k=1,2,3$, natural $\mathbb{C}$-linear associativity isomorphisms $\alpha_{X_3,X_2,X_1}: (X_3 \otimes X_2) \otimes X_1 \rightarrow X_3 \otimes (X_2 \otimes X_1)$
\end{enumerate}
such that the evident triangle and pentagon axioms hold.
%\[
%\begin{tikzpicture}[baseline= (a).base]
%\node[scale=0.85] (a) at (0,0){
%\begin{tikzcd}[column sep=2em, row sep=2.0em]
%{} & ( (X_4 \otimes X_3) \otimes X_2) \otimes X_1 \arrow{ld}[above left]{\alpha_{X_4,X_3,X_2} \otimes \id_{X_1}} \arrow{rd}[above right]{\alpha_{X_4 \otimes X_3, X_2, X_1}} & {} \\
%( X_4 \otimes (X_3 \otimes X_2)) \otimes X_1 \arrow{d}[left]{\alpha_{X_4, X_3 \otimes X_2, X_1}} & {} & (X_4 \otimes X_3) \otimes (X_2 \otimes X_1) \arrow{d}[right]{\alpha_{X_4,X_3,X_2 \otimes X_1}} \\
%X_4 \otimes ((X_3 \otimes X_2) \otimes X_1) \arrow{rr}[below]{\id_{X_4} \otimes \alpha_{X_3,X_2,X_1}} & {} & X_4 \otimes (X_3 \otimes (X_2 \otimes X_1)).
%\end{tikzcd}
%};
%\end{tikzpicture}
%\]
\end{Def}

In particular, underlying a Real monoidal category is an $\mathbb{R}$-linear monoidal category. Note also that the subcategory $\mathcal{C}^{(1)}$ is a $\mathbb{C}$-linear monoidal category which contains the unit object $\mathbf{1}$.

\begin{Def}
\begin{enumerate}[label=(\roman*)]
\item A Real fusion category is a finite semisimple Real monoidal category which is rigid and has a simple monoidal unit.

\item A Real fusion category is called pointed if its simple objects are invertible.
\end{enumerate}
\end{Def}

Real pointed fusion categories and their $\mathbb{C}$-linear monoidal equivalences form a groupoid $\mathsf{RPFus}$. The assignment $\mathcal{C} \mapsto \mathcal{C}^{(1)}$ extends to a functor from $\mathsf{RPFus}$ to the groupoid of pointed fusion categories.

\begin{Ex}
Let $\hat{\mathsf{G}}$ be a finite $\mathbb{Z}_2$-graded group and $\hat{\eta} \in Z^{3+\pi_{\hat{\mathsf{G}}}}(B \hat{\mathsf{G}})$. Let $\Vect_{\mathbb{C}}^{\hat{\eta}}(\hat{\mathsf{G}})$ be the $\mathbb{C}$-linear category of finite dimensional $\hat{\mathsf{G}}$-graded complex vector spaces. We write objects of as $V = \oplus_{\omega \in \hat{\mathsf{G}}} V_{\omega}$. Given $\omega \in \hat{\mathsf{G}}$, let $\mathbb{C}_{\omega} \in \Vect_{\mathbb{C}}^{\hat{\eta}}(\hat{\mathsf{G}})$ be the simple object which is a copy of $\mathbb{C}$ in degree $\omega$. Any simple of $\Vect_{\mathbb{C}}^{\hat{\eta}}(\hat{\mathsf{G}})$ is isomorphic to one of this form. Define a Real monoidal structure $\otimes$ on $\Vect_{\mathbb{C}}^{\hat{\eta}}(\hat{\mathsf{G}})$ by
\[
(V^{(2)} \otimes V^{(1)})_{\omega} = \bigoplus_{\substack{\omega_1, \omega_2 \in \hat{\mathsf{G}} \\ \omega = \omega_2 \omega_1}} V^{(2)}_{\omega_2} \otimes_{\mathbb{C}} {^{\pi(\omega_2)}}V^{(1)}_{\omega_1},
\]
with a similar formula for morphisms. The associator component
\[
(V^{(3)}_{\omega_3} \otimes_{\mathbb{C}} {^{\pi(\omega_3)}}V^{(2)}_{\omega_2} ) \otimes_{\mathbb{C}} {^{\pi(\omega_3 \omega_2)}}V^{(1)}_{\omega_1} \rightarrow V^{(3)}_{\omega_3} \otimes_{\mathbb{C}} ({^{\pi(\omega_3)}}V^{(2)}_{\omega_2} \otimes_{\mathbb{C}} {^{\pi(\omega_3 \omega_2)}}V^{(1)}_{\omega_1})
\]
is $\hat{\eta}([\omega_3 \vert \omega_2 \vert \omega_1])$ times the canonical associator. The pentagon axiom is equivalent to the twisted cocycle condition on $\hat{\eta}$. Let $\mathbf{1} = \mathbb{C}_e$ with right and left unitors
\[
\rho_V : V \otimes \mathbb{C}_e \rightarrow V, \qquad v_{\omega} \otimes c \mapsto {^{\pi(\omega)}}c v_{\omega}
\]
and
\[
\lambda_V : \mathbb{C}_e \otimes V \rightarrow V, \qquad c \otimes v_{\omega} \mapsto c v_{\omega}.
\]
Define the dual $V^*$ of $V$ by $(V^*)_{\omega} = {^{\pi(\omega)}} V^{\vee}_{\omega^{-1}}$. The non-trivial evaluation and coevaluation maps are
\[
\widetilde{\ev} : V \otimes V^* \rightarrow \mathbb{C}_e, \qquad v_{\omega} \otimes {^{\pi(\omega)}}f_{\sigma} \mapsto \delta_{\omega, \sigma^{-1}} \hat{\eta}([\omega \vert \omega^{-1} \vert \omega]) f_{\sigma}(v_{\omega}),
\]
where $\delta_{?,?}$ is a delta function, and
\[
\coev_{\omega} : \mathbb{C}_e \rightarrow \mathbb{C}_{\omega} \otimes \mathbb{C}_{\omega}^*, \qquad 1 \mapsto \hat{\eta}([\omega \vert \omega^{-1} \vert \omega])^{-1} \coev(1), \qquad \omega \in \hat{\mathsf{G}}.
\]
Then $\Vect_{\mathbb{C}}^{\hat{\eta}}(\hat{\mathsf{G}})$ is a Real pointed fusion category. The subcategory $\Vect_{\mathbb{C}}^{\hat{\eta}}(\hat{\mathsf{G}})^{(1)} \simeq \Vect_{\mathbb{C}}^{\eta}(\mathsf{G})$ is the $\mathbb{C}$-linear pointed fusion category associated to $\mathsf{G}$ and the restricted $3$-cocycle $\eta \in Z^3(B \mathsf{G})$ as described in \cite[Example 2.3.8]{etingof2015}. 
\end{Ex}

The group $\Aut_{\mathsf{Grp}_{\slash \mathbb{Z}_2}}(\hat{\mathsf{G}})$ of $\mathbb{Z}_2$-graded group automorphisms of $\hat{\mathsf{G}}$ acts by pullback on $H^{3+\pi_{\hat{\mathsf{G}}}}(B \hat{\mathsf{G}})$. The category $\mathsf{RPFus}$ is described by the following result, which has a well-known $\mathbb{C}$-linear  analogue.

\begin{Prop}
\label{prop:RPFusGrpd}
\phantomsection
\begin{enumerate}[label=(\roman*)]
\item Any object of $\mathsf{RPFus}$ is equivalent to a Real pointed fusion category of the form $\Vect_{\mathbb{C}}^{\hat{\eta}}(\hat{\mathsf{G}})$.

\item There is a short exact sequence of groups
\[
1 \rightarrow H^{2+\pi_{\hat{\mathsf{G}}}}(B \hat{\mathsf{G}}) \rightarrow \pi_0\Aut_{\mathsf{RPFus}}(\Vect_{\mathbb{C}}^{\hat{\eta}}(\hat{\mathsf{G}})) \rightarrow \textnormal{Stab}_{\Aut_{\mathsf{Grp}_{\slash \mathbb{Z}_2}}(\hat{\mathsf{G}})}(\hat{\eta}) \rightarrow 1,
\]
where $\pi_0\Aut_{\mathsf{RPFus}}(\Vect_{\mathbb{C}}^{\hat{\eta}}(\hat{\mathsf{G}}))$ is the group of isomorphism classes of autoequivalences of $\Vect_{\mathbb{C}}^{\hat{\eta}}(\hat{\mathsf{G}}) \in \mathsf{RPFus}$.
\end{enumerate}
\end{Prop}

\begin{proof}
Let $\mathcal{C} \in \mathsf{RPFus}$. Its group $\hat{\mathsf{G}}$ of isomorphism classes of simple objects inherits a natural $\mathbb{Z}_2$-grading from the decomposition $\mathcal{C} = \mathcal{C}^{(1)} \oplus \mathcal{C}^{(-1)}$ and compatibility of the functor $\otimes$. Choose a representative simple object for each element of $\hat{\mathsf{G}}$. Then the components of the associator at simple objects define a cocycle $\hat{\eta} \in Z^{3+\pi_{\hat{\mathsf{G}}}}(B \hat{\mathsf{G}})$. The full inclusion $\Vect_{\mathbb{C}}^{\hat{\eta}}(\hat{\mathsf{G}}) \hookrightarrow \mathcal{C}$ is then an equivalence.

Consider the second statement. After noting that any $\Phi \in \Aut_{\mathsf{RPFus}}(\Vect_{\mathbb{C}}^{\hat{\eta}}(\hat{\mathsf{G}}))$ preserves the set of simple objects with its $\mathbb{Z}_2$-grading, we find that the sequence is right exact. Suppose then that $\Phi$ is the identity on objects. The component of the monoidal data
\[
\Phi(\mathbb{C}_{\omega_2}) \otimes \Phi(\mathbb{C}_{\omega_1})
\xrightarrow[]{\sim}
\Phi(\mathbb{C}_{\omega_2} \otimes \mathbb{C}_{\omega_1})
\]
is multiplication by a complex number, say $\hat{\theta}([\omega_2 \vert \omega_1])$. Compatibility of this data with the associator is the condition $\hat{\theta} \in Z^{2+\pi_{\hat{\mathsf{G}}}}(B \hat{\mathsf{G}})$. Changing $\Phi$ within its isomorphism class changes $\hat{\theta}$ by an exact $2$-cocycle. This completes the proof.
\end{proof}

We can now prove a $2$-categorical analogue of Theorem \ref{thm:RealCentre}.

\begin{Thm}
\label{thm:realFusCat}
There is an $\mathbb{R}$-linear equivalence of categories
\[
Z_{D}(\Vect_{\mathbb{C}}^{\hat{\eta}}(\hat{\mathsf{G}})) \simeq \Vect_{\mathbb{C}}^{\uptau_{\pi}(\hat{\eta})^{-1}}(\Lambda_{\pi} B \hat{\mathsf{G}}),
\]
where the left hand side is the Drinfled centre of $\Vect_{\mathbb{C}}^{\hat{\eta}}(\hat{\mathsf{G}})$.
\end{Thm}

\begin{proof}
Let $(V, \beta) \in Z_{D}(\Vect_{\mathbb{C}}^{\hat{\eta}}(\hat{\mathsf{G}}))$, so that $V \in \Vect_{\mathbb{C}}^{\hat{\eta}}(\hat{\mathsf{G}})$ and $\beta: - \otimes V \Rightarrow V \otimes -$ is a natural isomorphism which satisfies a hexagon axiom. The component of $\beta$ at $\mathbb{C}_{\omega}$ is the data of $\mathbb{C}$-linear isomorphisms
\[
\beta_{\delta, \omega}: \mathbb{C}_{\omega} \otimes_{\mathbb{C}} {^{\pi(\omega)}}V_{\delta} \xrightarrow[]{\sim} V_{\omega \delta \omega^{-1}}  \otimes_{\mathbb{C}} {^{\pi(\delta)}} \mathbb{C}_{\omega}, \qquad \delta \in \hat{\mathsf{G}}.
\]
If $V_{\delta}$ is non-zero for some $\delta \in \hat{\mathsf{G}} \setminus \mathsf{G}$, then $\beta_{\delta,e}(c \otimes v_{\delta}) = v_{\delta} \otimes \overline{c}$. Consider the morphism $m_{\lambda}: \mathbb{C}_e \rightarrow \mathbb{C}_e$ given by multiplication by $\lambda \in \mathbb{C}$. Naturality of $\beta$ in $\mathbb{C}_e$ requires commutativity of the diagram
\[
\begin{tikzpicture}[baseline= (a).base]
\node[scale=1] (a) at (0,0){
\begin{tikzcd}[column sep=5em, row sep=2.0em]
\mathbb{C}_e \otimes_{\mathbb{C}} V_{\delta} \arrow{r}{\beta_{\delta,e}} \arrow{d}[left]{m_{\lambda} \otimes \id_{V_{\delta}}} & V_{\delta} \otimes_{\mathbb{C}} \overline{\mathbb{C}_e} \arrow{d}{\id_{V_{\delta}} \otimes \overline{m_{\lambda}}} \\
\mathbb{C}_e \otimes_{\mathbb{C}} V_{\delta} \arrow{r}[below]{\beta_{\delta,e}} &  V_{\delta} \otimes_{\mathbb{C}}\overline{\mathbb{C}_e},
\end{tikzcd}
};
\end{tikzpicture}
\]
which is the case if and only if $\lambda \in \mathbb{R}$. It follows that $V$ is supported on $\mathsf{G}$. The hexagon axiom for $(V, \beta)$ implies that the remaining structure maps $\beta_{g, \omega}$, $(g, \omega) \in \mathsf{G} \times \hat{\mathsf{G}}$, give $V$ the structure of a $\uptau_{\pi}(\hat{\eta})^{-1}$-twisted vector bundle over $\Lambda_{\pi} B \hat{\mathsf{G}}$.
%Going counterclockwise we get $v_{\omega} \otimes c \mapsto v_{\omega} \otimes \lambda c \mapsto \overline{\lambda c} \otimes v_{\omega}$ while going clockwise we get $v_{\omega} \otimes c \mapsto \overline{c} \otimes v_{\omega} \mapsto \lambda \overline{c} \otimes v_{\omega}$.
\end{proof}

Since it is a Drinfeld double, $Z_{D}(\Vect_{\mathbb{C}}^{\hat{\eta}}(\hat{\mathsf{G}}))$ has canonical $\mathbb{R}$-linear monoidal structure.\footnote{It also has a braiding, but we will not discuss this.} Under the equivalence $\Vect_{\mathbb{C}}^{\uptau_{\pi}(\hat{\eta})} (\Lambda_{\pi} B \hat{\mathsf{G}}) \simeq \mathbb{D}^{\hat{\eta}}(\hat{\mathsf{G}}) \ModR$, this monoidal structure is induced by a quasi-associative coproduct $\Delta: \mathbb{D}^{\hat{\eta}}(\hat{\mathsf{G}}) \rightarrow \mathbb{D}^{\hat{\eta}}(\hat{\mathsf{G}}) \otimes_{\mathbb{C}} \mathbb{D}^{\hat{\eta}}(\hat{\mathsf{G}})$. To describe this, first recall that the subgroup $\textnormal{Inn}(\hat{\mathsf{G}}) \leq \Aut_{\Grp_{\slash \mathbb{Z}_2}}(\hat{\mathsf{G}})$ of inner automorphisms acts trivially on $H^{3+ \pi}(\hat{\mathsf{G}})$. More precisely, for $g_i \in \mathsf{G}$ and $\omega \in \hat{\mathsf{G}}$, we have
\begin{equation}
\label{eq:3CocycConjInv}
\frac{\hat{\eta}([\omega g_3 \omega^{-1} \vert \omega g_3 \omega^{-1} \vert \omega g_3 \omega^{-1} ])}{\hat{\eta}([g_3 \vert g_2 \vert g_1])^{\pi(\omega)}}
=
(dc_{\omega})([g_3 \vert g_2 \vert g_1])
\end{equation}
where $c_{\omega} \in C^2(B \mathsf{G})$ is given by (see \cite[Proposition 8.1]{kim2015})
\[
c_{\omega}([g_2 \vert g_1]) = \frac{\hat{\eta}([\omega g_2 \omega^{-1} \vert \omega \vert g_1])}{\hat{\eta}([\omega \vert g_2 \vert g_1]) \hat{\eta}([\omega g_2 \omega^{-1} \vert \omega g_1 \omega^{-1} \vert \omega])}.
\]
The, for $l_{g \xrightarrow[]{\omega}} \in \mathbb{D}^{\hat{\eta}}(\hat{\mathsf{G}})$ the basis vector corresponding to the morphism $\omega: g \rightarrow \omega g \omega^{-1}$ in $\Lambda_{\pi} B\hat{\mathsf{G}}$, define
\[
\Delta(l_{g \xrightarrow[]{\omega}})
=
\sum_{\substack{ g_1,g_2 \in \mathsf{G} \\ g_2 g_1 = g}} c_{\omega}([g_2 \vert g_1]) l_{g_2 \xrightarrow[]{\omega}} \otimes  l_{g_1 \xrightarrow[]{\omega}}
\]
with associator
\[
\Phi = \sum_{g_1,g_2,g_3 \in \mathsf{G}} \hat{\eta}([g_3 \vert g_2 \vert g_1])l_{g_3 \xrightarrow[]{e}} \otimes l_{g_2 \xrightarrow[]{e}} \otimes l_{g_1 \xrightarrow[]{e}}.
\]
Formally, these are the same definitions as those for the quasi-bialgebra $D^{\eta}(\mathsf{G})$ \cite{dijkgraaf1992}. The defining equation \eqref{eq:3CocycConjInv} implies that $\Phi$ is an associator while the equation (cf. \cite[Corollary 8.3]{kim2015}, which uses slightly different conventions)
\[
%\label{eq:monoidalCond}
\frac{\uptau_{\pi}(\hat{\eta})([\omega_2 \vert \omega_1] g_2)
\uptau_{\pi}(\hat{\eta})([\omega_2 \vert \omega_1] g_1)} {\uptau_{\pi} (\hat{\eta}) ([\omega_2 \vert \omega_1] g_2 g_1)}
=
\frac{c_{\omega_2 \omega_1}([g_2 \vert g_1])}{c_{\omega_1}([g_2 \vert g_1])^{\pi(\omega_2)} \cdot c_{\omega_2}([\omega_1 g_2 \omega_1^{-1} \vert \omega_1 g_1 \omega_1^{-1}])}
\]
implies that $\Delta$ is a morphism of Real algebras. We say that $\mathbb{D}^{\hat{\eta}}(\hat{\mathsf{G}})$ is a Real quasi-bialgebra. Finally, note that $D^{\eta}(\mathsf{G})$ is a complex sub-quasi-bialgebra of $D^{\hat{\eta}}(\hat{\mathsf{G}})_{\textnormal{R}}$.

\subsection{Connection with Real \texorpdfstring{$2$}{}-representation theory}
\label{sec:Real2CharThy}

We relate the categorified group algebra $\Vect_{\mathbb{C}}^{\hat{\eta}}(\hat{\mathsf{G}})$ and the Real $2$-representation theory of $\mathsf{G}$, that is, the $2$-categorical generalization of Sections \ref{sec:vbJandl} and \ref{sec:realChar}. The latter theory is developed in detail in \cite{mbyoung2018c}, so we limit ourselves to a few points.

Let $\Vect_{\mathbb{C}}^{\hat{\eta}}(\hat{\mathsf{G}}) \ModR$ be the bicategory whose objects are left Real $\Vect_{\mathbb{C}}^{\hat{\eta}}(\hat{\mathsf{G}})$-module categories, that is, pairs $(F, \mathcal{M})$ consisting of a finite semisimple $\mathbb{C}$-linear category $\mathcal{M}$ and a $\mathbb{C}$-linear monoidal functor $F : \Vect_{\mathbb{C}}^{\hat{\eta}} (\hat{\mathsf{G}}) \rightarrow \End_{\mathbb{R}}(\mathcal{M})$ such that the functor
\[
F(\mathbb{C}_{\omega}) : {^{\pi(\omega)}}\mathcal{M} \rightarrow \mathcal{M},
\qquad
\omega \in \hat{\mathsf{G}}
\]
is $\mathbb{C}$-linear. The $1$-morphisms of $\Vect_{\mathbb{C}}^{\hat{\eta}}(\hat{\mathsf{G}}) \ModR$ are $\mathbb{C}$-linear functors with intertwining data for the $\Vect_{\mathbb{C}}^{\hat{\eta}}(\hat{\mathsf{G}})$-actions while $2$-morphisms are their compatible $\mathbb{C}$-linear natural transformations. With these definitions, there is an $\mathbb{R}$-linear biequivalence
\begin{equation}
\label{eq:2ModInterp}
\Vect_{\mathbb{C}}^{\hat{\eta}}(\hat{\mathsf{G}}) \ModR
\simeq
2 \Vect_{\mathbb{C}}^{\hat{\eta}}(B\hat{\mathsf{G}}),
\end{equation}
where the right hand side is the bicategory of $\hat{\eta}$-twisted Real $2$-representations of $\mathsf{G}$ on Kapranov--Voevodsky $2$-vector spaces, as defined in \cite[\S 5.4]{mbyoung2018c}. This is the $2$-categorical analogue of Proposition \ref{prop:modInterp}. A central result of \cite{mbyoung2018c} is a categorified character theory for $2 \Vect_{\mathbb{C}}^{\hat{\eta}}(B\hat{\mathsf{G}})$, which we summarize as follows.

\begin{Thm}[{\cite{mbyoung2018c}}]
\label{thm:2CharThy}
A Real $2$-representation $\rho \in 2 \Vect_{\mathbb{C}}^{\hat{\eta}}(B\hat{\mathsf{G}})$ has
\begin{enumerate}[label=(\roman*)]
\item a Real categorical character $\Tr_{\rho} \in \Vect_{\mathbb{C}}^{\uptau_{\pi}^{\refl}(\hat{\eta})}(\Lambda_{\pi}^{\refl} B \hat{\mathsf{G}})$, and

\item a Real $2$-character $\chi_{\rho} \in \Gamma_{\Lambda \Lambda_{\pi}^{\refl} B \hat{\mathsf{G}}}(\uptau \uptau_{\pi}^{\refl}(\hat{\eta})_{\mathbb{C}})$.
\end{enumerate}
\end{Thm}

In particular, the twisted transgression map $\uptau_{\pi}^{\refl}$ plays a central role in the character theory of $2 \Vect_{\mathbb{C}}^{\hat{\eta}}(B\hat{\mathsf{G}})$. An important technical point in \cite{mbyoung2018c} is that the $2$-cochain $\uptau_{\pi}^{\refl}(\hat{\eta})$ is in fact closed. It is possible to verify this directly, but this is a rather unpleasant task. The approach of Section \ref{sec:twistedTransGrpdFin} gives a much more efficient verification.

It can be shown that there is an $\mathbb{R}$-linear monoidal equivalence
\[
Z(2 \Vect_{\mathbb{C}}^{\hat{\eta}}(B\hat{\mathsf{G}}))
\simeq
Z_D(\Vect_{\mathbb{C}}^{\hat{\eta}}(\hat{\mathsf{G}})),
\]
where the right hand size is the monoidal category of pseudonatural transformations of the identity pseudofunctor of $2 \Vect_{\mathbb{C}}^{\hat{\eta}}(B\hat{\mathsf{G}})$. In view of \eqref{eq:2ModInterp}, this equivalence can be proved as a modification of the proof of \cite[Corollary 5.3]{bernaschini2019}. We omit the details. However, let us mention that this equivalence allows to restate Theorem \ref{thm:realFusCat} in a way which is more obviously a $2$-categorical analogue of Theorem \ref{thm:RealCentre}.

\subsection{Discrete torsion in string and M-theory with orientifolds}
\label{sec:discTor}

We show that $\uptau_{\pi}^{\refl}$ encodes the one-loop discrete torsion in unoriented string and M-theory. For analogous results in the oriented setting, see \cite{lupercio2006}.

Consider first discrete torsion in orientifold string theory (or unoriented conformal field theory). A twisted $2$-cocycle $\hat{\theta} \in Z^{2 + \pi_{\hat{\mathsf{G}}}}(B \hat{\mathsf{G}})$ can be seen as an orientifold discrete torsion, in that it is an orientifold compatible $B$-field for orientifold string theory defined on a global quotient by $\hat{\mathsf{G}}$ \cite{distler2011}, \cite{sharpe2011}. The iterated transgression $\uptau \uptau_{\pi}^{\refl}(\hat{\theta})$ is a locally constant function on the groupoid $\Lambda \Lambda_{\pi}^{\refl} B \hat{\mathsf{G}}$. Its value on $(g, \omega) \in \hat{\mathsf{G}}^{\langle 2 \rangle} =\Obj (\Lambda \Lambda_{\pi}^{\refl} B \hat{\mathsf{G}}) = \mathsf{G}^{\langle 2 \rangle}$ is
\[
\hat{\theta}([g^{-1} \vert g])^{- \Delta_{\omega}} \frac{\hat{\theta}([g \vert \omega ])}{\hat{\theta}([\omega \vert g^{\pi(\omega)} ])},
\]
which is precisely the discrete torsion phase factor of $\mathbb{T}^2$ if $\pi(\omega)=1$ and $\mathbb{K}$ if $\pi(\omega) =-1$; see \cite[Eqn. 16]{bantay2003}, \cite[\S 5.1]{sharpe2011}.

Turning to M-theory, we regard a cocycle $\hat{\eta} \in Z^{3 + \pi_{\hat{\mathsf{G}}}}(B \hat{\mathsf{G}})$ as a $M$-theory analogue of discrete torsion, determining a universal orientifold compatible $C$-field \cite{sharpe2011}. Then $\uptau^2 \uptau_{\pi}^{\refl}(\hat{\eta})$ is a locally constant function  on $\Lambda^2 \Lambda_{\pi}^{\refl} B\hat{\mathsf{G}}$. There is a natural bijection
\[
\Obj(\Lambda^2 \Lambda_{\pi}^{\refl} B \hat{\mathsf{G}})
=
\hat{\mathsf{G}}^{\langle 3 \rangle} :=
\{
(g, \omega_1, \omega_2) \in \mathsf{G} \times \hat{\mathsf{G}}^2 \mid (g, \omega_i) \in \hat{\mathsf{G}}^{\langle 2 \rangle}, \; (\omega_1, \omega_2) \in \hat{\mathsf{G}}^{(2)}
\}.
\]
The set $\hat{\mathsf{G}}^{\langle 3 \rangle}$ decomposes according to the degrees of a triple $(g, \omega_1, \omega_2) \in \hat{\mathsf{G}}^{\langle 3 \rangle}$, which we interpret as determining a closed unoriented $3$-manifold $M$ whose orientation double cover is either $\mathbb{T}^3 \sqcup \mathbb{T}^3$ or $\mathbb{T}^3$. 
When $\pi(\omega_1)=\pi(\omega_2)=1$ we have $M= \mathbb{T}^3$ and, from the explicit expression for $\uptau_{\pi}^{\refl}$ from Section \ref{sec:twistedTransGrpdFin}, the value of $\uptau^2 \uptau_{\pi}^{\refl}(\hat{\eta})$ on $(g,\omega_1,\omega_2) \in \hat{\mathsf{G}}^{\langle 3 \rangle}$ is
\begin{equation}
\label{eq:cubeDiscTor}
\frac{\hat{\eta}([\omega_2 \vert  g \vert \omega_1]) \hat{\eta}([\omega_1 \vert \omega_2 \vert g]) \hat{\eta}([ g  \vert \omega_2 \vert \omega_1])}{\hat{\eta}([\omega_2 \vert \omega_1 \vert g]) \hat{\eta}([g \vert \omega_2 \vert \omega_1]) \hat{\eta}([\omega_1 \vert g \vert \omega_2])}.
\end{equation}
The remaining three cases, in which at least one $\omega_i$ is of degree $-1$, correspond to $M = \mathbb{K} \times S^1$. When $\pi(\omega_1)=1 = -\pi(\omega_2)$, the value of $\uptau^2 \uptau_{\pi}^{\refl}(\hat{\eta})$ is
\begin{equation}
\label{eq:unoriCubeDiscTor}
\frac{\hat{\eta}([g^{-1} \vert g \vert \omega_1]) \hat{\eta}([\omega_1 \vert g^{-1} \vert g])}{\hat{\eta}([g^{-1} \vert \omega_1 \vert g])}
\frac{\hat{\eta}([\omega_1 \vert \omega_2 \vert g^{-1}]) \hat{\eta}([g \vert \omega_1 \vert \omega_2]) \hat{\eta}([\omega_2 \vert g^{-1} \vert \omega_1])}{\hat{\eta}([\omega_1 \vert g \vert \omega_2])\hat{\eta}([\omega_2 \vert \omega_1 \vert g^{-1}]) \hat{\eta}([g \vert \omega_2 \vert \omega_1])}.
\end{equation}
When $\pi(\omega_1)=-1 =-\pi(\omega_2)$, the formula is similar. Finally, when $\pi(\omega_1)=\pi(\omega_2)=-1$ the value of $\uptau^2 \uptau_{\pi}^{\refl}(\hat{\eta})$ is
\[
%\label{eq:unoriCubeDiscTor}
\frac{\hat{\eta}([g^{-1} \vert \omega_2 \vert g^{-1}]) \hat{\eta}([g^{-1} \vert g \vert \omega_1]) \hat{\eta}([\omega_1 \vert g \vert g^{-1}])}{\hat{\eta}([g^{-1} \vert g \vert \omega_2]) \hat{\eta}([\omega_2 \vert g \vert g^{-1}]) \hat{\eta}([g^{-1} \vert \omega_1 \vert g^{-1}])} 
\frac{\hat{\eta}([\omega_1 \vert \omega_2 \vert g]) \hat{\eta}([ g \vert \omega_1 \vert \omega_2]) \hat{\eta}([\omega_2 \vert g^{-1} \vert \omega_1])}{\hat{\eta}([\omega_1 \vert g^{-1} \vert \omega_2]) \hat{\eta}([\omega_2 \vert \omega_1 \vert g]) \hat{\eta}([ g \vert \omega_2 \vert \omega_1])}.
\]
In this final case the triple $(g, \omega_1 \omega_2^{-1}, \omega_2)$ is of the previous type and, after some calculation, the expression is of the form \eqref{eq:unoriCubeDiscTor}. The phases \eqref{eq:cubeDiscTor} and \eqref{eq:unoriCubeDiscTor} appear in the work of Sharpe \cite[\S 6.2]{sharpe2011} as $C$-field discrete torsion phase factors in orientifold M-theory on the manifolds $\mathbb{T}^3$ and $\mathbb{K} \times S^1$, respectively. In Sharpe's notation, we have $g_1 = \omega_2$, $g_2 = g$ and $g_3 = \omega_1$. Explicitly, the equality of a summand of \eqref{eq:unoriCubeDiscTor} with Sharpe's phase factor follows from the identity
\[
\frac{\hat{\eta}([g^{-1} \vert g \vert \omega_1]) \hat{\eta}([\omega_1 \vert g^{-1} \vert g])}{\hat{\eta}([g^{-1} \vert \omega_1 \vert g])} = \frac{\hat{\eta}([g \vert g^{-1} \vert \omega_1]) \hat{\eta}([\omega_1 \vert g \vert g^{-1}])}{\hat{\eta}([g \vert \omega_1 \vert g^{-1}])}, \qquad g, \omega_1 \in \mathsf{G}.
\]

%\footnote{In Sharpe's notation, we have $g_1 = \omega_2$, $g_2 = g$ and $g_3 = \omega_1$. The equality of a summand of \eqref{eq:unoriCubeDiscTor} with Sharpe's phase factor follows from the identity
%\[
%\frac{\hat{\eta}([g^{-1} \vert g \vert \omega_1]) \hat{\eta}([\omega_1 \vert g^{-1} \vert g])}{\hat{\eta}([g^{-1} \vert \omega_1 \vert g])} = \frac{\hat{\eta}([g \vert g^{-1} \vert \omega_1]) \hat{\eta}([\omega_1 \vert g \vert g^{-1}])}{\hat{\eta}([g \vert \omega_1 \vert g^{-1}])}, \qquad g, \omega_1 \in \mathsf{G}.
%\]
%}

\footnotesize

%\newpage \listoftodos \newpage

\bibliographystyle{plain}
\bibliography{mybib}

\end{document}